\documentclass[11pt]{article}

\usepackage{arxiv}

\usepackage{amsmath,amssymb,amsthm}
\usepackage[utf8]{inputenc} 
\usepackage[T1]{fontenc}    
\usepackage{amsfonts}       
\usepackage{microtype}      
\usepackage{booktabs}       
\usepackage[x11names=false,table=false]{xcolor} 
\definecolor{AccentBlue}{RGB}{20,55,100}
\usepackage{hyperref}       
\usepackage{graphicx}       
\usepackage{array,needspace,enumitem} 
\usepackage{longtable} 
\usepackage{titlesec} 
\setlist{itemsep=2pt,parsep=2pt,topsep=4pt,partopsep=2pt} 
\titleformat{\section}
  {\large\bfseries\color{AccentBlue}}{\thesection.}{0.5em}{}%
  [\vspace{2pt}{\color{AccentBlue!35}\titlerule[0.4pt]}]
\titleformat{\subsection}
  {\normalsize\bfseries}{\thesubsection}{0.5em}{}
\titleformat{\subsubsection}
  {\normalsize\bfseries\itshape}{\thesubsubsection}{0.5em}{}
\titlespacing*{\section}{0pt}{14pt plus 4pt minus 2pt}{7pt plus 1pt}
\titlespacing*{\subsection}{0pt}{10pt plus 2pt minus 1pt}{3pt plus 1pt}
\titlespacing*{\subsubsection}{0pt}{8pt plus 2pt minus 1pt}{2pt}

\makeatletter
\renewcommand{\@toptitlebar}{%
  {\color{AccentBlue!65}\hrule height 0.4pt}%
  \vskip 0.25in
  \vskip -\parskip%
}
\renewcommand{\@bottomtitlebar}{%
  \vskip 0.29in
  \vskip -\parskip
  {\color{AccentBlue!65}\hrule height 0.4pt}%
  \vskip 0.09in%
}
\renewenvironment{abstract}{%
  {\color{AccentBlue!65}\hrule height 0.4pt}%
  \vskip 6pt
  \centerline{\large\bfseries\scshape\color{AccentBlue}Abstract}%
  \vskip 4pt
  \begin{quote}\noindent\ignorespaces
}{%
  \end{quote}%
  \vskip 2pt
  {\color{AccentBlue!65}\hrule height 0.4pt}%
}
\makeatother

\title{Invariant Rings of Adjacent-Quadratic Triangular Derivations}
\author{\href{https://orcid.org/0009-0004-4344-141X}{\includegraphics[scale=0.06]{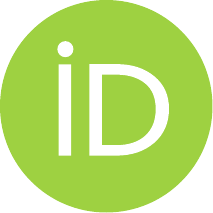}\hspace{1mm}Ahmed M. Adly} \\[1ex] {\small Independent Researcher}}
\date{}

\newcommand{\Pshift}[2]{P_{#1,#2}}
\newcommand{\Dmap}[1]{D_{#1}}
\newcommand{\AQflow}[2]{\Psi_{#1,#2}}
\newcommand{\Afring}[1]{A_{#1}}

\theoremstyle{plain}
\newtheorem{theorem}{Theorem}[section]
\newtheorem{lemma}[theorem]{Lemma}
\newtheorem{proposition}[theorem]{Proposition}
\newtheorem{corollary}[theorem]{Corollary}
\theoremstyle{definition}
\newtheorem{definition}[theorem]{Definition}

\newtheorem{conjecture}[theorem]{Conjecture}
\theoremstyle{remark}
\newtheorem{remark}[theorem]{Remark}
\newtheorem*{remark*}{Remark}

\renewcommand{\shorttitle}{Adjacent-Quadratic Invariant Rings}

\hypersetup{
colorlinks=true,
linkcolor=AccentBlue,
citecolor=AccentBlue,
urlcolor=AccentBlue,
pdftitle={Invariant Rings of Adjacent-Quadratic Triangular Derivations},
pdfauthor={Ahmed M. Adly},
pdfsubject={math.AC, math.RA},
pdfkeywords={locally nilpotent derivation, invariant ring, triangular derivation},
}

\begin{document}
\maketitle

\begin{abstract}
Let \(k\) be a field of characteristic zero and \(D_N(a_n)=a_{n+1}a_{n+2}\) the adjacent-quadratic locally nilpotent derivation on \(A_N=k[a_0,\dots,a_N]\). We determine the invariant ring at \(N=5\), where deleting \(a_0\) gives the polynomial base kernel \(B^{D_5}=K=k[p,q,J_1,J_2]\), so localization gives the polynomial-line invariant algebra \((A_5^{D_5})_{pq}=K_{pq}[W]\) with \(W=F_0/(p^2q^4)\), but this does not by itself determine the global kernel. An invariant \(H\) of weight \(24\) and \(a_0\)-degree \(2\) lies outside \(K[F_0]\), and adjoining it gives the six-generator hypersurface \(C_5\cong k[P,Q,J_1,J_2,F,H]/(S)\); the obstruction is explained by algebraic residue behavior defeating naive denominator descent. An invariant \(F_{46}^{\mathrm{int}}\) of weight \(46\) and \(a_0\)-degree \(4\) lies outside \(C_5\) and satisfies \(p^2F_{46}^{\mathrm{int}}\in C_5\). To pass from this candidate to the full invariant ring, a second slice chart arising from \(f_{13}\), together with exact \(q\)- and \(p\)-saturation, yields \(\boxed{A_5^{D_5}=C_5[F_{46}^{\mathrm{int}}]}\), hence finite generation, with a codimension-two complete-intersection presentation and explicit Hilbert series. At \(N=6\) the transported \(B\)-level ring is exact, while the \(A_6\) analysis is established through total degree \(13\) and the full \(A_6\) invariant ring remains open.
\end{abstract}


\clearpage
\setcounter{tocdepth}{1}
\tableofcontents
\clearpage
\section{Introduction}
\label{sec:intro}

Let $k$ be a field of characteristic zero and
$\Afring{N}=k[a_0,\dots,a_N]$, with the boundary convention
$a_{N+1}=a_{N+2}=0$. The adjacent-quadratic derivation
$\Dmap{N}(a_n)=a_{n+1}a_{n+2}$ is locally nilpotent and triangular,
so its kernel $\Afring{N}^{\Dmap{N}}$ is the ring of invariants of
the corresponding $\mathbb{G}_a$-action. The question is the explicit
structure of these kernels.

For $N\le 4$, the base cases admit explicit invariant and local
descriptions. The case $N=5$ is the first in which the local
invariant picture is straightforward while its globalization is
not: deleting $a_0$ leaves a polynomial kernel whose localization is
a polynomial line, and that line does not descend without additional
generators.

Indeed, with $B=k[a_1,\dots,a_5]$, $x=a_3$, $y=a_2$, $z=a_1$,
$p=a_4$, $q=a_5$, one has
$K=B^{\Dmap{5}}=k[p,q,J_1,J_2]$, where
$J_1=2pz-y^2$ and $J_2=2qy-x^2$. The element $s=x/(pq)$ satisfies
$\Dmap{5}(s)=1$, and the slice it defines gives the localized form
$A_{pq}^{\Dmap{5}}=K_{pq}[W]$, with $W=F_0/c$ for the invariant $F_0$
and the pole-clearing denominator $c=p^2q^4$. The invariant $F_0$
carries $a_0$ into the picture; the question is what else must be
adjoined to obtain the global kernel.

The answer begins with a failure. The weight-$24$ invariant $H$, of
$a_0$-degree $2$, satisfies $\Dmap{5}(H)=0$ but lies outside
$K[F_0]$, since every element of $K[F_0]$ of weight $24$ has
$a_0$-degree at most $1$. The reason is visible in the residue
field $L_q=k(p,u,x^2)$ at $q$ (with $u=J_1\bmod q$ and $\rho=G_0/c$ for
$F_0=ca_0+G_0$): the initial form
\[
r_q=\operatorname{in}_q(q^4\rho)=\frac{x^7}{35p^2},
\qquad
r_q^2\in L_q,
\qquad
r_q\notin L_q,
\]
is algebraic of degree $2$, so the transcendence argument behind
the naive $q$-adic denominator bound does not apply, and the pole it
forbids cancels in exactly the way $H$ exhibits.

Adjoining the forced generator gives
$C_5=k[p,q,J_1,J_2,F_0,H]$, cut out by the single weight-$28$
relation $S=q^4H-11025F_0^2-105q^2J_2^5J_1-9J_2^7$. This algebra
agrees with the invariant ring through weight $45$, but not
globally.

At weight $46$, the invariant $F_{46}^{\mathrm{int}}$, of
$a_0$-degree $4$, has leading $a_0$-coefficient
$121550625\,p^6q^8$ outside $J_4=(p^8q^8)$, so
$F_{46}^{\mathrm{int}}\notin C_5$. This is the $p$-adic counterpart
of the first failure: a top coefficient outside the expected top
ideal, defeating the $p$-adic bound. The two arguments are not
identical --- the residue analysis on the $p$-side has its own
initial form and degree --- but the shape of the obstruction is the
same.

With $R=C_5[F_{46}^{\mathrm{int}}]$ identified, it remains to
prove that $R$ is the full invariant ring. The plinth element
$f_{13}=pJ_2^3/35+pq^2J_1J_2/4$, satisfying
$\Dmap{5}(g_{13})=f_{13}$ for an explicit $g_{13}$ while
$q\nmid f_{13}$, supplies a second slice chart covering the locus
where the first chart degenerates. The two localized descriptions
agree on the overlap, and exact $q$- and $p$-saturation contract the
$p$-localized kernel to $R$, giving
\[
\boxed{A_5^{D_5}=C_5[F_{46}^{\mathrm{int}}],}
\]
finitely generated.

The resulting ring has a codimension-two complete-intersection
presentation with explicit Hilbert series. At \(N=6\) the transported
\(B\)-level ring is exact, while the \(A_6\) analysis is exact over
\(\mathbb{Q}\) through total degree \(13\) with degree \(14\) blocked;
the \(A_6\) ring remains open.

\section{The adjacent-quadratic family}
\label{sec:definitions}

Let $k$ be a field of characteristic zero,
$\Afring{N} = k[a_0,\dots,a_N]$, with boundary conditions
$a_{N+1} = a_{N+2} = 0$.

\begin{definition}[Family]
\label{def:family}
The adjacent-quadratic shift is the discrete map
$\Pshift{\lambda}{N} \colon \Afring{N} \to \Afring{N}$ given on
coordinates by
\[
\Pshift{\lambda}{N}(a)_n = a_n + \lambda a_{n+1} a_{n+2},
\qquad 0 \le n \le N.
\]
The associated triangular derivation is
$\Dmap{N}(a_n) = a_{n+1}a_{n+2}$, with flow
$\AQflow{t}{N} = \exp(t\Dmap{N})$.
\end{definition}

\begin{lemma}[Triangular inverse]
\label{lem:inverse}
The adjacent-quadratic shift $\Pshift{\lambda}{N}$ is a triangular
automorphism of $\Afring{N}$.
Given $b = \Pshift{\lambda}{N}(a)$, the inverse is computed by descending
recursion $a_N = b_N$, $a_{N-1} = b_{N-1}$, and
$a_n = b_n - \lambda a_{n+1}a_{n+2}$ for $n = N-2$ down to $0$.
Its Jacobian matrix is unitriangular, so $\det D\Pshift{\lambda}{N} = 1$;
it fixes the top two coefficients $a_{N-1}$ and $a_N$ but does not
preserve total degree.
\end{lemma}

\begin{remark}
For $\lambda \ne 0$ (and $N \ge 2$) the fixed-point locus of
$\Pshift{\lambda}{N}$ is cut out by $a_1a_2,\dots,a_{N-1}a_N$.
\end{remark}

\begin{remark}[Characteristic zero]
\label{rem:char-zero}
Every result is over a field of characteristic zero. The hypothesis is
used to divide by $3,5,35$ in the invariant $F_0$, for the Jacobian
scalars $4$ and $420=2^2\cdot3\cdot5\cdot7$, in the slice theorem and
factorial-closedness arguments, and for the binomial divisions in the
multfloor proof. No positive-characteristic statement for $N\ge4$ is
claimed here.
\end{remark}

\begin{proposition}[Discrete-versus-continuous defect]
\label{prop:defect}
The discrete map $\Pshift{\lambda}{N}$ is the first-order coordinate
update of the flow $\AQflow{t}{N}$ evaluated at $t = \lambda$.
For $N \ge 3$ and $\lambda \ne 0$ it is never the exponential:
\[
\Pshift{\lambda}{N} \ne \exp(\lambda \Dmap{N}).
\]
The bound is sharp: at $N = 1$, $\Dmap{1} = 0$ so both sides are the
identity, while at $N = 2$, $\Dmap{2}^2 = 0$ on generators, so
$\exp(\lambda\Dmap{2})$ and $\Pshift{\lambda}{2}$ agree as algebra
automorphisms even though $\Dmap{2}^2 \ne 0$ as an operator, e.g.\
$\Dmap{2}^2(a_0^2) = 2(a_1a_2)^2$.
The divergence appears at order $\lambda^2$: at $N=3$ the gap is the
explicit term $(\lambda^2/2)a_2^2a_3$, and at $N=4$ an explicit nonzero
$\lambda^2$-coefficient gives a counterexample to equality, detailed in
the $N \le 4$ base section to follow.
\end{proposition}

\subsection*{Notation}
$V_w$ is the Fibonacci-weight-$w$ homogeneous component of the
ambient polynomial ring, so that the invariant piece is
$(A_N^{D_N})_w = \ker(D_N|_{V_w})$. The series $s_6(w)$ counts
monomials of weight $w$ in the six generators (Section
\ref{sec:global}); $s_5(w)$ and $s_K(w)$ are the analogous counts
in the five $K[F_0]$-generators and in $K = k[p,q,J_1,J_2]$
respectively. Pole orders along prime divisors are $e_i =
\operatorname{ord}_{\pi_i}(c)$ in the saturation procedure, with
$e_p = 2$ and $e_q = 4$ the $p$- and $q$-pole orders of $\rho$.
Denominators are $c = p^2q^4$ at $N=5$,
$c_6 = a_5^2a_6^4$ and $c' = c_6^2$ at $N=6$.

\subsection*{Notation table}
\label{subsec:notation-table}
\begin{center}
\refstepcounter{table}\label{tab:notation}
\textbf{Table~\thetable. Notation.}\\[1ex]
\renewcommand{\arraystretch}{1.1}\begin{tabular}{>{\raggedright\arraybackslash}p{5.6cm}>{\raggedright\arraybackslash}p{9.6cm}}
\toprule
symbol & meaning, first defined \\ \midrule
$F_0$, $c = p^2q^4$, $G_0 = F_0-ca_0$ & fifth invariant, denominator, $B$-part (Definition~\ref{def:F0}) \\
$W = F_0/c = a_0+\rho$ & invariant pole element (not a slice); $\rho = G_0/c$ \\
$s = x/(pq)$ & local slice, $\Dmap{5}(s) = 1$ (Lemma~\ref{lem:slice}) \\
$H$ ($H_B$ at $N=6$) & hidden invariant, weight $24$, $a_0$-degree $2$ \\
$C_5 = k[p,q,J_1,J_2,F_0,H]$ & six-generator subalgebra (Theorem~\ref{thm:global-hypersurface}) \\
$R=C_5[F_{46}^{\mathrm{int}}]$, $E=p^2F_{46}^{\mathrm{int}}$, $r=p^2T-E$ & closed ring and torsion relation (Theorems~\ref{thm:quotient-4648}, \ref{thm:relation-ideal}) \\
$S$, $\operatorname{wt}(S) = 28$ & weight-$28$ relation cutting out $C_5$ \\
$J_m = (p^{2m}q^{4\lceil m/2\rceil})$ & top ideal at $a_0$-degree $m$ (Theorem~\ref{def:Jm}) \\
$F_{46}^{\mathrm{int}},F_{46}^{\mathrm{rat}}=F_{46}^{\mathrm{int}}/26460$ & second invariant, weight $46$, $a_0$-degree $4$; $T\mapsto F_{46}^{\mathrm{int}}$ (Theorem~\ref{thm:F46}) \\
$v_p, v_q$ ($v_5, v_6$ at $N=6$) & prime-divisor valuations on $\operatorname{Frac}(B)$ \\
$\deg_{\mathrm{ord}}$, $\operatorname{wt}$ & ordinary total degree; Fibonacci weight $(8,5,3,2,1,1)$ on $(a_0,\dots,a_5)$, extended by $\operatorname{wt}(a_0) = 13$ at $N=6$ \\
$\deg_{a_0}$ & $a_0$-degree (exact on $B[a_0]$; tops multiply); $W$-degree coincides with it ($W = a_0+\rho$ monic) \\
\bottomrule
\end{tabular}
\end{center}
A slice means only $s$ with $\Dmap{5}(s) = 1$; a pole element only
$W$ with $\Dmap{5}(W) = 0$; a denominator is always given with its
formula.

\section{\texorpdfstring{The $N \le 4$ base}{The N ≤ 4 base}}
\label{sec:base}

Throughout, $k$ has characteristic zero and $\Pshift{\lambda}{N}$
denotes the adjacent-quadratic shift of Definition~\ref{def:family}.
The discrete map is the first-order coordinate update of the flow
$\AQflow{t}{N} = \exp(t\Dmap{N})$; for $N \ge 3$ and $\lambda \ne 0$
it is never the exponential itself (Proposition~\ref{prop:defect}):
\[
\Pshift{\lambda}{N} \ne \exp(\lambda \Dmap{N}),
\]
and the two diverge at order $\lambda^2$ in every continuous case below.

\begin{proposition}[$N=2$ kernel]
\label{prop:N2-kernel}
$A_2^{\Dmap{2}} = k[a_1,a_2]$, a polynomial ring.
\end{proposition}

\begin{proof}
At $N = 2$, $a_3 = a_4 = 0$, so $\Dmap{2}$ kills $a_1$ and $a_2$
while $\Dmap{2}(a_0) = a_1a_2 \ne 0$. Write $F = \sum_{k=0}^m C_k
a_0^k$ with $C_k \in k[a_1,a_2]$; freeness gives $\Dmap{2}(F) =
\sum_{k \ge 1} kC_ka_1a_2a_0^{k-1}$, since $\Dmap{2}$ vanishes on the
coefficients. In the domain $A_2$, $mC_ma_1a_2 = 0$ with $a_1a_2 \ne
0$ forces $C_m = 0$ for $m \ge 1$ (characteristic zero), so $F = C_0
\in k[a_1,a_2]$.
\end{proof}

\begin{theorem}[$N=3$ invariants and Euler gap]
\label{thm:N3}
For $N=3$, with $a_4 = a_5 = 0$, the derivation acts by
$\Dmap{3}(a_0) = a_1a_2$, $\Dmap{3}(a_1) = a_2a_3$,
$\Dmap{3}(a_2) = \Dmap{3}(a_3) = 0$.
\begin{enumerate}
\item The rational function
\[
H_3 = a_0 - \frac{a_1^2}{2a_3}
\]
satisfies $\Dmap{3}(H_3) = 0$; it is a continuous invariant of the flow
$\AQflow{t}{3}$ on the open set $a_3 \ne 0$. (The subscript distinguishes
this $N=3$ rational invariant from the weight-$24$ polynomial hidden
invariant $H$ of the $N=5$ theory.)
\item The discrete map fixes $a_2$ and $a_3$ and sends
$b_0 = a_0 + \lambda a_1a_2$, $b_1 = a_1 + \lambda a_2a_3$.
The rational function
\[
I_0 = a_0 - \frac{a_1^2}{2a_3} + \frac{\lambda}{2}a_1a_2
\]
satisfies $I_0 \circ \Pshift{\lambda}{3} = I_0$ identically on
$a_3 \ne 0$, and $I_0 \to H_3$ as $\lambda \to 0$.
\item Since $\Dmap{3}^2(a_0) = a_2^2a_3$ and $\Dmap{3}^3(a_0) = 0$,
the exact flow on the $a_0$-coordinate is
$\AQflow{\lambda}{3}(a_0) = a_0 + \lambda a_1a_2
+ \frac{\lambda^2}{2}a_2^2a_3$,
while $\Pshift{\lambda}{3}(a_0) = a_0 + \lambda a_1a_2$.
The Euler gap is exactly
\[
\AQflow{\lambda}{3}(a_0) - \Pshift{\lambda}{3}(a_0)
= \frac{\lambda^2}{2}a_2^2a_3.
\]
\end{enumerate}
\end{theorem}

\begin{theorem}[$N=4$ invariants and commutator]
\label{thm:N4}
For $N=4$, $\Dmap{4}(a_0) = a_1a_2$, $\Dmap{4}(a_1) = a_2a_3$,
$\Dmap{4}(a_2) = a_3a_4$, $\Dmap{4}(a_3) = \Dmap{4}(a_4) = 0$.
\begin{enumerate}
\item The polynomials $2a_0a_3 - a_1^2$ and $2a_1a_4 - a_2^2$ satisfy
$\Dmap{4}(2a_0a_3 - a_1^2) = \Dmap{4}(2a_1a_4 - a_2^2) = 0$;
they are continuous invariants of $\AQflow{t}{4}$.
\item On the open set $a_3a_4 \ne 0$, with $u = a_2$ and
$c = \lambda a_3a_4 \ne 0$, the rational function
\[
I_1 = a_1 - \frac{a_2^2}{2a_4} + \frac{\lambda}{2}a_2a_3
\]
satisfies $I_1 \circ \Pshift{\lambda}{4} = I_1$ identically.
Its $\lambda \to 0$ limit is the continuous invariant
$a_1 - a_2^2/(2a_4)$.
\item Let $\Delta_c F(u) = F(u+c) - F(u)$ and let
\[
S_1(u) = \frac{u^2-cu}{2c}, \qquad
S_2(u) = \frac{u^3}{3c} - \frac{u^2}{2} + \frac{cu}{6}, \qquad
S_3(u) = \frac{u^4}{4c} - \frac{u^3}{2} + \frac{cu^2}{4},
\]
so that $\Delta_c S_k(u) = u^k$ for $k = 1,2,3$.
Then
\[
I_0 = a_0 - \lambda I_1 S_1(a_2)
- \frac{\lambda}{2a_4}S_3(a_2)
+ \frac{\lambda^2a_3}{2}S_2(a_2)
\]
satisfies $I_0 \circ \Pshift{\lambda}{4} = I_0$ identically on
$a_3a_4 \ne 0$; the degenerate loci $a_3a_4 = 0$ are excluded here.
\item The derivation orbit of $a_0$ is
$\Dmap{4}^2(a_0) = a_2^2a_3 + a_1a_3a_4$,
$\Dmap{4}^3(a_0) = 3a_2a_3^2a_4$,
$\Dmap{4}^4(a_0) = 3a_3^3a_4^2$, $\Dmap{4}^5(a_0) = 0$,
so $\exp(\lambda\Dmap{4})(a_0)$ carries nonzero
$\lambda^2$-, $\lambda^3$-, and $\lambda^4$-terms while
$\Pshift{\lambda}{4}(a_0) = a_0 + \lambda a_1a_2$.
\item Writing $[\Pshift{\lambda}{4},\Pshift{\mu}{4}]_0$ for the ordered
point difference
$(\Pshift{\mu}{4} \circ \Pshift{\lambda}{4}
- \Pshift{\lambda}{4} \circ \Pshift{\mu}{4})_0$,
\[
[\Pshift{\lambda}{4},\Pshift{\mu}{4}]_0
= \lambda\mu(\lambda-\mu)a_2a_3^2a_4,
\]
and components $1$ through $4$ of the two compositions agree.
The reverse ordering gives the negation.
\end{enumerate}
Each identity above is a direct expansion.
\end{theorem}

\begin{lemma}[Discrete-to-continuous bridge]
\label{lem:bridge}
Work in characteristic $\ne 2,3$. Let $U$ be a nonempty open set on
which all denominators below are nonzero, and let $I_\lambda$,
rational in $(\lambda, a)$ and regular at $\lambda = 0$ on $U$ after
clearing denominators, satisfy
$I_\lambda \circ \Pshift{\lambda}{N} = I_\lambda$ on $U$ as a rational
identity. Let $H = I_\lambda|_{\lambda = 0}$ on $U$. Then
$\Dmap{N}(H) = 0$ on $U$.
\end{lemma}
\begin{proof}
Clear denominators to a polynomial identity over
$k[\lambda, a, \mathrm{denom}^{-1}]$ and differentiate formally in
$\lambda$ (a $k$-linear derivation; the divisions by $2,3$ in
$S_1,S_2,S_3$ use the hypothesis), then set $\lambda = 0$. With
$P_0 = \mathrm{id}$ and formal derivative
$\frac{d}{d\lambda}|_0 \Pshift{\lambda}{N}(a)_n = a_{n+1}a_{n+2} =
\Dmap{N}(a_n)$, the formal chain rule gives
$(\partial_\lambda I)_0 + dH(\Dmap{N}) = (\partial_\lambda I)_0$,
hence $\Dmap{N}(H) = 0$. No topology or limits are used.
Theorems~\ref{thm:N3}
and~\ref{thm:N4} are the instances: the $N=3$ invariant is
polynomial in $\lambda$, and the $N=4$ cleared forms
$\lambda S_1 = u^2/(2a_3a_4) - \lambda u/2$,
$\lambda S_3/(2a_4) = u^4/(8a_3a_4^2) - \lambda u^3/(4a_4)
+ \lambda^2u^2a_3/8$ and $\lambda^2S_2$ vanishes at $\lambda = 0$ are regular at
$\lambda = 0$ with limit
$L = a_0 - a_1a_2^2/(2a_3a_4) + a_2^4/(8a_3a_4^2)$ satisfying
$\Dmap{4}(L) = 0$.
\end{proof}

\begin{remark}[Finite-field $N=3$ orbit lengths]
\label{thm:finiteN3}
Let $\Pshift{\lambda}{3}$ act over a finite field $\mathbb{F}_q$ of
characteristic $p$, with $\lambda \ne 0$ and
$c = \lambda a_2a_3$. Then $a_3' = a_3$, $a_2' = a_2$,
$a_1' = a_1 + c$, $a_0' = a_0 + \lambda a_1a_2$, and:
\begin{enumerate}
\item If $c = 0$, then $a_1$ is fixed; the orbit length is $1$ if
$\lambda a_1a_2 = 0$ and exactly $p$ otherwise, since the additive
group of $\mathbb{F}_q$ has exponent $p$.
\item If $c \ne 0$ and $p$ is odd, the orbit length is exactly $p$:
over one full $a_1$-period the accumulated shift of $a_0$ is
$\lambda a_2(pa_1 + c\,p(p-1)/2) = 0$.
\item If $c \ne 0$ and $p = 2$, the orbit length is exactly $4$:
over two steps the $a_0$-shift is $\lambda^2a_2^2a_3 \ne 0$,
and doubling it gives $0$ in characteristic $2$.
\end{enumerate}
The exact counts for $\lambda \ne 0$, $q = p^f$, are:
$q^3 + q(q-1)$ fixed points; $q(q-1)^2/p$ cycles of length $p$ from
the $c = 0$ stratum; and $(q-1)^2q^2/L$ cycles of length $L$ from the
$c \ne 0$ stratum, with $L = p$ for $p$ odd and $L = 4$ for $p = 2$.
The classification is proved by the algebraic argument above for
every finite field. Exhaustive enumeration over the tested fields
$\mathbb{F}_3$, $\mathbb{F}_5$, $\mathbb{F}_7$, $\mathbb{F}_{11}$, $\mathbb{F}_2$, and $\mathbb{F}_4$ confirms the
lengths and counts on those fields; enumeration over tested fields
does not imply all-$\mathbb{F}_q$ results beyond the proof.
\end{remark}

Below $N = 5$ nothing hides. The next section computes the $a_0$-free base on which everything turns.

\section{\texorpdfstring{The $N=5$ five-variable kernel}{The N=5 five-variable kernel}}
\label{sec:kernel}

Deleting $a_0$ leaves the five-variable problem. Put
\[
x = a_3, \qquad y = a_2, \qquad z = a_1, \qquad p = a_4, \qquad q = a_5,
\]
\[
B = k[x,y,z,p,q] = k[a_1,\dots,a_5], \qquad K = B^{\Dmap{5}},
\]
with
\[
\Dmap{5}(x) = pq, \qquad \Dmap{5}(y) = px, \qquad
\Dmap{5}(z) = xy, \qquad \Dmap{5}(p) = \Dmap{5}(q) = 0.
\]
Let
\[
J_1 = 2pz - y^2 = 2a_1a_4 - a_2^2, \qquad
J_2 = 2qy - x^2 = 2a_2a_5 - a_3^2,
\]
and write $K_0 = k[p,q,J_1,J_2]$ for the subalgebra they generate (henceforth $K=K_0$ once Theorem~\ref{thm:descent} identifies it with $B^{\Dmap{5}}$; the symbol $R$ is reserved for the closed ring $R=C_5[F_{46}^{\mathrm{int}}]$).

\begin{lemma}[Kernel generators]
\label{lem:kernel-gens}
The four polynomials $p$, $q$, $J_1$, $J_2$ lie in $K$:
\[
\Dmap{5}(p) = \Dmap{5}(q) = \Dmap{5}(J_1) = \Dmap{5}(J_2) = 0,
\]
since $\Dmap{5}(J_1) = 2xy\cdot p - 2y\cdot px = 0$ and
$\Dmap{5}(J_2) = 2pq\cdot x - 2x\cdot pq = 0$.
Hence $K_0 \subseteq K$.
Moreover $p$, $q$, $J_1$, $J_2$ are algebraically independent:
the $4 \times 4$ Jacobian minor in columns $(a_1,a_2,a_4,a_5)$ is
$4a_4a_5 \ne 0$ over $k$ (characteristic zero for the scalar $4$).
So $K_0$ is a polynomial ring in four variables.
\end{lemma}

\begin{lemma}[Local slice and localized kernel]
\label{lem:slice}
On the localization $B_{pq}$ the element
\[
s = \frac{x}{pq}
\]
satisfies $\Dmap{5}(s) = 1$; it is a local slice.
With $Y = J_2/(2q)$ and $Z = (J_1 + Y^2)/(2p)$, one has
$\Dmap{5}(Y) = \Dmap{5}(Z) = 0$ and the reconstruction identities
\[
x = pqs, \qquad
y = Y + \tfrac12 p^2qs^2, \qquad
z = Z + \tfrac12 pqYs^2 + \tfrac18 p^3q^2s^4,
\]
as direct expansion shows. Hence $B_{pq} = K_{0,pq}[s]$, and since
$D_5(s) = 1$, the local-slice theorem for locally nilpotent
derivations (Freudenburg~\cite{freudenburg2017}, Corollary~1.26)
gives the localized kernel
\[
B_{pq}^{\Dmap{5}} = K_{0,pq} = k[p,q,J_1,J_2,(pq)^{-1}].
\]
In particular every $f \in K$ has the form $f = P/(p^aq^b)$ for some
$P \in K_0$ and $a, b \ge 0$ (independent exponents: $1/p = q\cdot
(pq)^{-1}$ and $1/q = p\cdot(pq)^{-1}$ in $K_{0,pq}$).
\end{lemma}

The descent from $K_{0,pq}$ to $K_0$ is controlled by the prime divisors
$(p)$ and $(q)$. Since $B$ is a polynomial ring, hence a unique
factorization domain, each defines a discrete valuation, still denoted
$v_p$ and $v_q$, on $\operatorname{Frac}(B)$, with valuation rings
\[
V_p = B_{(p)} = \{h \in \operatorname{Frac}(B) : v_p(h) \ge 0\},
\qquad
V_q = B_{(q)} = \{h \in \operatorname{Frac}(B) : v_q(h) \ge 0\},
\]
the localizations of $B$ at the prime ideals $(p)$ and $(q)$.

\begin{lemma}[Contraction with divisibility identities]
\label{lem:contraction}
Let $\varphi \colon k[p,q,U,V] \to B$ be the substitution map
$(p,q,U,V) \mapsto (p,q,J_1,J_2)$, injective by
Lemma~\ref{lem:kernel-gens}. Then
\[
\varphi^{-1}((p)B) = (p), \qquad \varphi^{-1}((q)B) = (q).
\]
Indeed, reducing modulo $p$ gives $J_1 \equiv -y^2$ and
$J_2 \equiv 2qy-x^2$, with the divisibility identities
\[
(J_1 + y^2)/p = 2z \in B, \qquad (J_2 + x^2)/q = 2y \in B.
\]
The map $\alpha \colon k[q,U,V] \to B/(p) = k[x,y,z,q]$,
$(q,U,V) \mapsto (q,-y^2,2qy-x^2)$, is injective: writing a relation
as $\sum_{i=0}^n f_i(q,U)V^i$ with $V = 2qy-x^2 = -x^2+2qy$,
the leading term in $x$ is $(-1)^n f_n(q,-y^2)x^{2n}$;
$f_n(q,-y^2) \ne 0$ if $f_n \ne 0$, since
$\sum_j g_j(q)(-y^2)^j = 0$ forces each $g_j = 0$ by $y$-degree,
so $f_n = 0$ by induction on $n$.
Thus $P(0,q,J_1,J_2) = 0$ in $B/(p)$ implies $P(0,q,U,V) = 0$ as a
polynomial, so $p \mid P$ already in $K_0$.
Modulo $q$, $J_1 \equiv 2pz-y^2$ and $J_2 \equiv -x^2$,
and $\beta \colon k[p,U,V] \to B/(q) = k[x,y,z,p]$,
$(p,U,V) \mapsto (p,2pz-y^2,-x^2)$, is injective:
write $\sum_{j=0}^m g_j(p,V)U^j$ with $U = 2pz-y^2$;
the leading term in $z$ is $g_m(p,-x^2)(2p)^m z^m$,
and $g_m(p,-x^2) \ne 0$ if $g_m \ne 0$ by $x$-degree, so
induction on $m$ gives injectivity, whence $q \mid P$ in $K_0$.
Moreover $p \nmid J_1, J_2, q$ and $q \nmid J_1, J_2, p$ in $B$,
so the divisors are genuine, e.g.\ $J_1/p \notin B$.
\end{lemma}

\begin{theorem}[Descent $B^{\Dmap{5}} = K$]
\label{thm:descent}
With $B$, $K$, $K_0$ as above,
\[
\boxed{B^{\Dmap{5}} = K = k[p,q,J_1,J_2].}
\]
Explicitly, the intersection obligation holds:
\[
B \cap K_{0,pq} = K_0 \quad \text{(hence } B \cap K_{pq} = K\text{ once $K_0=K$)}.
\]
\end{theorem}

\begin{proof}
$K_0 \subseteq K$ is Lemma~\ref{lem:kernel-gens} and
$K \subseteq K_{0,pq}$ is Lemma~\ref{lem:slice}.
Let $f \in K$, so $f = P/(p^aq^b)$ with $P \in K_0$.
Since $f \in B$, we have $P = p^aq^bf \in B$.
If $a \ge 1$ then $P \in (p)B$, so by
Lemma~\ref{lem:contraction} $P = pP_1$ with $P_1 \in K_0$; cancelling
and repeating gives $P \in (p)^aK_0$.
The same argument with $q$ gives $P \in (p)^a(q)^bK_0$ (here $p,q$ are nonassociated primes in the UFD $K_0$, so $p^a\mid P$ and $q^b\mid P$ imply $p^aq^b\mid P$), whence
$f \in K_0$. Thus $K \subseteq K_0$, and $B \cap K_{0,pq} = K_0$ follows by
the same descent applied to any $f \in B \cap K_{0,pq}$.
\end{proof}

\begin{lemma}[Dixmier divisibility for the $a_1a_2$-image]
\label{lem:dixmier-divisibility}
Let $\mu \in K$ and $\nu \in B$ satisfy
$a_1a_2\mu + \Dmap{5}(\nu) = 0$. Then $c = p^2q^4$ divides $\mu$
in $K$.
\end{lemma}

\begin{proof}
Localize at $pq$: $s = x/(pq)$ has $\Dmap{5}(s) = 1$
(Lemma~\ref{lem:slice}), so $B_{pq} = K_{0,pq}[s]$ freely and
$\Dmap{5} = d/ds$. Write $\nu = \sum_{i=0}^t N_is^i$ (unique,
$N_i \in K_{0,pq}$) and $a_1a_2 = \sum_{k \in \{0,2,4,6\}} e_ks^k$
with the exact $s$-coefficients $e_k \in K_{0,pq}$, zero otherwise. Comparing
$s$-coefficients in $a_1a_2\mu + \nu' = 0$ (characteristic zero
for division by $k+1$) gives $\nu = N_0 - \mu T$ with the fixed
integrator $T = \sum_k e_ks^{k+1}/(k+1)$. Expanding $T$ in the
free $x$-basis ($x = a_3$) over
$\operatorname{Frac}(k[a_1,a_2,a_4,a_5])$ gives
\[
T = -\frac{x^7}{35p^2q^4} + \frac{a_2x^5}{5p^2q^3}
  - \frac{(a_1a_4 + a_2^2)x^3}{3p^2q^2}
  + \frac{a_1a_2x}{pq},
\]
in particular $T$ has only odd $x$-powers with top coefficient
$-1/(35p^2q^4) \ne 0$. Since $P$, $Q$, $U = J_1$ are $x$-free and
$V = J_2 = 2a_2q - x^2$ has $x$-degree $2$, every element of
$K_{0,pq}$ --- hence $N_0$ and $\mu$ --- has only even $x$-powers.
Thus the odd-$x$ part of $\nu$ is exactly $-\mu T$'s, and each
odd-$x$ coefficient lies in $L := k[a_1,a_2,a_4,a_5]$
($B = L[x]$ free). Write $\mu = \sum_{j=0}^d \mu_jV^j$
($V$-filtration on the free ring $K$, $\mu_d \ne 0$,
$\mu_j \in S := k[P,Q,U]$). The top odd-$x$ term of $\mu T$
is $x^{2d+7}$ with coefficient $(-1)^{d+1}\mu_d/(35p^2q^4)$
(leading $x$-part of $V^d$ is $(-x^2)^d$; no other term reaches
$2d+7$), so $\mu_d/(35p^2q^4) \in L$. Applying
$\pi \colon L \to k[a_2,P,Q]$, $a_1 \mapsto 0$, gives
$\sum_k c_k(-a_2^2)^k/(35p^2q^4) \in k[a_2,P,Q]$ for
$\mu_d = \sum c_kU^k$; by $a_2$-freeness each
$c_k \in (p^2q^4)$ in $k[P,Q]$, so $\mu_d \in (p^2q^4)S$.
Peel: $\mu' = \mu - p^2q^4h_dV^d$ (with $\mu_d = p^2q^4h_d$)
has $V$-degree $< d$ and still satisfies the hypothesis, since
$a_1a_2 \cdot p^2q^4h_dV^d = -\Dmap{5}(G_0h_dV^d)$ (from
$a_1a_2c = -\Dmap{5}(G_0)$ with $h_dV^d \in K$ invariant).
Induction on $d$ gives $\mu \in (p^2q^4)K$ ($\mu = 0$ is
trivial). The $e_k$ and $N_i$ coefficient data above are exact over
$\mathbb{Q}$ (see Supplementary Material).
\end{proof}

The denominator $c = p^2q^4$ will clear denominators in the
$a_0$-extension.

\section{\texorpdfstring{The explicit $N=5$ invariant}{The explicit N=5 invariant}}
\label{sec:invariant}

\begin{definition}[The explicit invariant]
\label{def:F0}
The degree-seven polynomial $F_0 \in \Afring{5}$ is
\[
\boxed{
F_0 = p^2q^4a_0 - pq^3xyz + \tfrac13pq^2x^3z
  + \tfrac13q^2x^3y^2 - \tfrac15qx^5y + \tfrac1{35}x^7.
}
\]
The denominators $3$, $5$, $35$ are covered by the standing
characteristic-zero hypothesis. Write $F_0 = ca_0 + G_0$ with
$c = p^2q^4$ and $G_0 \in B$ the remaining five terms.
\end{definition}

\begin{theorem}[$D_5(F_0) = 0$]
\label{thm:F0-invariant}
With $\Dmap{5}(a_0) = zy = a_1a_2$,
\[
\boxed{\Dmap{5}(F_0) = 0,}
\]
so $F_0 \in \Afring{5}^{\Dmap{5}}$ and
$k[p,q,J_1,J_2,F_0] \subseteq \Afring{5}^{\Dmap{5}}$.
\end{theorem}

\begin{proof}
Since $\Dmap{5}(p) = \Dmap{5}(q) = 0$, computing term by term:
\begin{align*}
\Dmap{5}(p^2q^4a_0) &= p^2q^4zy, \\
\Dmap{5}(-pq^3xyz) &= -p^2q^4yz - p^2q^3x^2z - pq^3x^2y^2, \\
\Dmap{5}(\tfrac13pq^2x^3z)
  &= p^2q^3x^2z + \tfrac13pq^2x^4y, \\
\Dmap{5}(\tfrac13q^2x^3y^2)
  &= pq^3x^2y^2 + \tfrac23pq^2x^4y, \\
\Dmap{5}(-\tfrac15qx^5y) &= -pq^2x^4y - \tfrac15pqx^6, \\
\Dmap{5}(\tfrac1{35}x^7) &= \tfrac15pqx^6.
\end{align*}
Grouping by monomial, $p^2q^4zy - p^2q^4yz = 0$,
$-p^2q^3x^2z + p^2q^3x^2z = 0$,
$-pq^3x^2y^2 + pq^3x^2y^2 = 0$,
$(\frac13 + \frac23 - 1)pq^2x^4y = 0$, and
$-\frac15pqx^6 + \frac15pqx^6 = 0$.
Hence $\Dmap{5}(F_0) = 0$.
\end{proof}

\begin{corollary}[Localized $A$-presentation]
\label{cor:A-localized}
With $c = p^2q^4$ and $W = F_0/c = a_0 + \rho$,
$\rho = G_0/c \in B_{pq}$, $D_5(W) = 0$ and
\[
A_{pq}^{\Dmap{5}} = K_{pq}[W],
\]
a polynomial ring in one variable over $K_{pq}$. The equality holds
after localizing at $pq$; it is not a global statement.
\end{corollary}

\begin{proof}
$D_5(c) = 0$ and $D_5(F_0) = 0$ give $D_5(W) = 0$.
The slice $s = x/(pq)$ of Lemma~\ref{lem:slice} exhibits
$A_{pq} = K_{pq}[W][s]$ with $D_5 = d/ds$ on the slice factor, so the
local-slice theorem gives the displayed kernel.
\end{proof}

\begin{lemma}[Ordinary degree $7$, Fibonacci weight $14$]
\label{lem:F0-degree-weight}
With the Fibonacci grading
$\operatorname{wt}(a_0,\dots,a_5) = (8,5,3,2,1,1)$,
\[
\boxed{\deg_{\mathrm{ord}}(F_0) = 7, \qquad \operatorname{wt}(F_0) = 14.}
\]
Indeed each of the six monomials has ordinary degree $7$, and
\[
\renewcommand{\arraystretch}{1.2}\begin{array}{c|c|c}
\text{monomial} & \text{ordinary degree} & \text{Fibonacci weight} \\ \hline
p^2q^4a_0 & 2+4+1 = 7 & 2+4+8 = 14 \\
pq^3xyz & 1+3+1+1+1 = 7 & 1+3+2+3+5 = 14 \\
pq^2x^3z & 1+2+3+1 = 7 & 1+2+6+5 = 14 \\
q^2x^3y^2 & 2+3+2 = 7 & 2+6+6 = 14 \\
qx^5y & 1+5+1 = 7 & 1+10+3 = 14 \\
x^7 & 7 & 14
\end{array}
\]
\end{lemma}

\begin{proposition}[Independence via the Jacobian criterion]
\label{prop:F0-independence}
The five invariants $p$, $q$, $J_1$, $J_2$, $F_0$ are algebraically
independent over $k$.
\end{proposition}

\begin{proof}
For the direction used, let $f_1,\dots,f_5$ be the five polynomials and
suppose $P(f_1,\dots,f_5) = 0$ for nonzero $P$ of minimal total degree.
Differentiating gives $\sum_i (\partial P/\partial X_i)(f)\,
\partial f_i/\partial a_j = 0$ for every $j$. In characteristic zero
some $\partial P/\partial X_i$ is a nonzero polynomial of strictly
smaller degree, so by minimality $(\partial P/\partial X_i)(f) \ne 0$
for that $i$: the row vector $((\partial P/\partial X_i)(f))$ is
nonzero and annihilates the Jacobian, whose rows are therefore
dependent over $\operatorname{Frac}$ and whose maximal minors all
vanish. Hence one nonvanishing maximal minor implies algebraic
independence (classical Jacobian criterion; see
e.g.\ Freudenburg~\cite{freudenburg2017}). In columns
$(a_0,a_1,a_2,a_4,a_5)$,
\[
\det \frac{\partial(p,q,J_1,J_2,F_0)}
  {\partial(a_0,a_1,a_2,a_4,a_5)}
  = 4a_4^3a_5^5 \ne 0
\]
as a polynomial. The matrix is block-triangular:
$p$, $q$, $J_1$, $J_2$ are free of $a_0$ while
$\partial F_0/\partial a_0 = p^2q^4$, so the determinant is
$p^2q^4$ times the $4 \times 4$ minor $4a_4a_5$ of the kernel
generators.
\end{proof}

\begin{remark}[Observation: weight-shift, not a theorem]
The quotient $W = F_0/(p^2q^4)$ has Fibonacci weight $8$, while the
polynomial $F_0$ has weight $14 = 8 + 6$, the summand $6$ being the
weight of the denominator $p^2q^4$. No generation consequence is
drawn from this here.
\end{remark}

The subalgebra $K[F_0]$ is therefore assembled. Whether it equals the
invariant ring is decided by weight $24$.

\section{\texorpdfstring{The first obstruction and the corrected hypersurface}{The first obstruction and the corrected hypersurface}}
\label{sec:global}

The localized kernel is polynomial, but its global descent fails at the first weight where an $a_0^2$-direction can occur. The algebra $K[F_0]$ predicts no such direction before weight $28$; the explicit invariant $H$ supplies one at weight $24$.

\begin{remark}[The proposed five-generator descriptions are false]
\label{rem:global-ring}
The equalities
$\Afring{5}^{\Dmap{5}} = K[F_0] = k[p,q,J_1,J_2,F_0]$, the
polynomiality of the full invariant ring in those five generators,
and the identification of its Hilbert series with
$1/((1-u)^2(1-u^4)(1-u^6)(1-u^{14}))$ are all false. The
counterexample in each case is the invariant $H$ below, of Fibonacci
weight $24$ and $a_0$-degree $2$, satisfying $\Dmap{5}(H) = 0$ and
$H \notin K[F_0]$.
\end{remark}

\begin{remark}[The proposed Hilbert series]
\label{rem:hilbert-series}
With ordinary degrees
$(p,q,J_1,J_2,F_0) = (1,1,2,2,7)$ and Fibonacci weights
$(1,1,6,4,14)$, the five generators are algebraically independent,
so $k[p,q,J_1,J_2,F_0]$ is polynomial with series
$1/((1-t)^2(1-t^2)^2(1-t^7))$ and
$1/((1-u)^2(1-u^4)(1-u^6)(1-u^{14}))$. These are the series of
$K[F_0]$ itself, not of $\Afring{5}^{\Dmap{5}}$: exact kernel
dimensions agree through weight $23$ and exceed the series from
weight $24$ on ($213 > 212$, then $238,267,296,330,364,404$
against $236,264,292,325,358,396$ over weights $25$--$30$).
\end{remark}

\begin{remark}[Refuting counterexample $H$]
\label{rem:counterexample-H}
The invariant $H \in \Afring{5}$ has Fibonacci weight $24$,
$17$ terms, content $1$, and $a_0$-degree $2$ with leading term
$11025\,a_0^2 a_4^4 a_5^4$, and satisfies $\Dmap{5}(H) = 0$.
At weight $24$ every monomial of $K[F_0]$ has $a_0$-degree at most
$1$ (only $F_0$ carries $a_0$, in degree $1$, and $F_0^2$ already
weighs $28 > 24$), while $H$ has $a_0$-degree $2$; hence
$H \notin K[F_0]$. For every $w < 24$ the graded pieces coincide,
$\ker(\Dmap{5}|_{V_w}) = K[F_0]_w$ (by direct spanning through
weight $14$ and exact rational rank comparison for $15 \le w \le 23$;
see Appendix~\ref{sec:generation-cert} and
Appendix~\ref{app:ranks}), so $24$ is the smallest weight at which
the inclusion is strict. Thus $H$ is the first obstruction to
$K[F_0]$. Its $W$-expansion is
\[
H = 11025p^4q^4W^2+
\frac{105J_2^5J_1}{q^2}+
\frac{9J_2^7}{q^4},
\]
whose last two terms carry the $q^{-4}$ poles analyzed below.
\end{remark}

\begin{proposition}[First obstruction, existence]
\label{prop:H1}
$H\in\Afring{5}^{\Dmap{5}}\setminus K[F_0]$.
Indeed $\Dmap{5}(H)=0$ by direct expansion, and non-membership is the
$a_0$-degree count of Remark~\ref{rem:counterexample-H}. Hence
$\Afring{5}^{\Dmap{5}}\supsetneq K[F_0]$.
\end{proposition}
\begin{proposition}[First obstruction, minimality]
\label{prop:H2}
$(A_5^{D_5})_w=(K[F_0])_w$ for all $w<24$.
This uses the $a_0$-degree-one closure (every invariant of
$a_0$-degree $\le1$ lies in $K[F_0]$) together with the spanning and
rank arguments cited in Remark~\ref{rem:counterexample-H}.
In particular $24$ is the first weight at which $K[F_0]$ fails.
\end{proposition}
\begin{theorem}[$H$ as intrinsic obstruction]
\label{thm:H-intrinsic}
The first obstruction space
$(A_5^{\Dmap{5}})_{24}/(K[F_0])_{24}$ is one-dimensional
($213$ versus $212$), and the class of $H$ is a distinguished nonzero
generator of this quotient (Propositions~\ref{prop:H1}--\ref{prop:H2}),
characterized without reference to its $17$-term expression by:
the minimal obstruction to the $K[F_0]$-saturation
module $M_2 = c^2K$ (its $W$-top $H_2 = 11025p^4q^4$ lies outside
$M_2$, $v_q = 4 < 8$, while every local regularity condition holds),
and the realization of the residue dependence $r_q^2 \in L_q$ as a
global invariant (Section~\ref{subsec:residue-transcendence}).
Cohomological, conductor, or canonical-module characterizations are not
claimed.
\end{theorem}

\section{The corrected intermediate algebra}
\label{sec:corrected-algebra}

Adjoining $H$ gives a hypersurface, not a polynomial ring.

\begin{theorem}[Hypersurface presentation of the six-generator subalgebra]
\label{thm:global-hypersurface}
Put $C_5 = k[p,q,J_1,J_2,F_0,H] \subseteq \Afring{5}^{\Dmap{5}}$.
Since $\Dmap{5} \ne 0$ on six variables,
$\operatorname{trdeg}_k \Afring{5}^{\Dmap{5}} = 5$; with the five
algebraically independent generators
(Proposition~\ref{prop:F0-independence}),
$\operatorname{trdeg}_k C_5 = 5$. Let
$\pi \colon k[P,Q,J_1,J_2,F,H] \to C_5$ be the presentation map:
$\ker\pi$ is prime ($C_5$ a domain) of height $6 - 5 = 1$. Work in the
abstract polynomial ring, with $P,Q,J_1,J_2,F,H$ indeterminates and
$S = Q^4H - R_0$ where $R_0 = 11025F^2 + 105Q^2J_2^5J_1 + 9J_2^7$.
Then $R_0|_{Q=0}^{\mathrm{abs}} = 11025F^2 + 9J_2^7 \ne 0$ as an abstract
polynomial, so the prime $Q$ does not divide $R_0$ in
$k[P,Q,J_1,J_2,F]$; hence $\gcd(Q^4,R_0)=1$ and $S$, linear in $H$ with
coprime coefficients, is primitive, hence irreducible, and $(S)$ is
prime of height $1$ (distinguish this from evaluation in the image,
where $R_0|_{q=0}=9x^{14}-9x^{14}=0$ is exactly the residue cancellation
producing $H$). With $S = 0$ on the generators, $\ker\pi = (S)$. Thus
\[
\boxed{C_5 \;\cong\; k[P,Q,J_1,J_2,F,H]\,/\,(S),}
\]
where $H$ is the weight-24 invariant of
Remark~\ref{rem:counterexample-H} and $S$ is the weight-28 relation
\[
\boxed{S \;=\; q^{4}H \;-\; 11025\,F_0^{2}
 \;-\; 105\,q^{2}J_2^{5}J_1 \;-\; 9\,J_2^{7}.}
\]
The six graded generators carry Fibonacci weights
\[
\operatorname{wt}(p,q,J_1,J_2,F_0,H) = (1,1,6,4,14,24)
\]
and $S$ has weight $28$. As a complete intersection, $C_5$ has the
exact Fibonacci-weighted Hilbert series
\[
\boxed{H_{\mathrm{wt}}(u) \;=\;
 \frac{1-u^{28}}{(1-u)^{2}(1-u^{4})(1-u^{6})(1-u^{14})(1-u^{24})}.}
\]
With $s_6(w)$ the monomial count in the six polynomial generators,
$(C_5)_w$ has dimension $s_6(w) - s_6(w-28)$ for $w \ge 28$; exact
ranks give $(C_5)_w = (\Afring{5}^{\Dmap{5}})_w$ at weights
$5,15,23,24,25,30,35,40$. The five generators remain algebraically
independent, so the failure of the five-generator model is precisely
the missing sixth generator modulo the single weight-28 relation.
This theorem concerns the subalgebra $C_5$ only; the equality
$C_5=\Afring{5}^{\Dmap{5}}$ is false, failing first at weight $46$
(Section~\ref{sec:second-witness}).
\end{theorem}

\begin{remark}[Geometry of the hypersurface $X_5 = \operatorname{Spec}(C_5)$]
\label{rem:singlocus}
$X_5 = \operatorname{Spec}(C_5)$ is the
hypersurface $V(S) \subset \mathbb{A}^6$ cut out by
$S = q^4H-11025F_0^2-105q^2J_2^5J_1-9J_2^7$.
Since $S$ is linear in $H$,
$\partial S/\partial H = q^4$, $\partial S/\partial F_0 = -22050F_0$,
$\partial S/\partial J_1 = -105q^2J_2^5$,
$\partial S/\partial J_2 = -525q^2J_2^4J_1-63J_2^6$,
$\partial S/\partial q = 4q^3H-210qJ_2^5J_1$, $\partial S/\partial p = 0$:
$\partial_HS=0$ gives $q=0$, then $\partial_{F_0}S=0$ gives $F_0=0$,
then $\partial_{J_2}S=-J_2^4(525q^2J_1+63J_2^2)$ gives $J_2=0$; the
remaining partials vanish automatically there, and $S$ vanishes on
$V(q,F_0,J_2)$. Hence the singular locus is exactly the linear space
$\operatorname{Sing}(X_5) = V(q,F_0,J_2) \cong \mathbb{A}^3_{p,J_1,H}$,
of dimension $3$. Since
$C_5 \cong k[P,Q,J_1,J_2,F,H]/(S)$, $C_5$ is a hypersurface in the
regular polynomial ring $k[P,Q,J_1,J_2,F,H]$. Hence $C_5$ is
Cohen--Macaulay, and with
singular locus of codimension $5-3=2$ it is normal by Serre's criterion
(R1$+$S2); see e.g.\ Bruns--Herzog~\cite{brunsherzog1998}. By graded
Nakayama (loc.\ cit.), the six generators are minimal as a homogeneous
generating set: $S$ lies in $\mathfrak{m}^2$ (every term is a product
of at least two positive-weight elements), and the six residue classes
are linearly independent in $\mathfrak{m}/\mathfrak{m}^2$ (at weights
$1$, $4$, $6$ by independence of $p,q,J_1,J_2$; at $14$ since $F_0
\notin K$; at $24$ since $H \notin K[F_0]$), so
$\dim_k\mathfrak{m}/\mathfrak{m}^2 = 6$. No quotient-geometric
claim about $\operatorname{Spec}(\Afring{5}^{\Dmap{5}})$ (orbits,
stabilizers, fibers, Makar--Limanov) is made: the quotient geometry of the full ring is not established here.
\end{remark}

\subsection{The Failure of Residue-Field Transcendence}
\label{subsec:residue-transcendence}

On the five-variable base $B = k[a_1,\dots,a_5]$ the local picture is
exact: the kernel is
$K = B^{\Dmap{5}} = k[p,q,J_1,J_2]$, the localized extension
$A_{pq}^{\Dmap{5}} = K_{pq}[W]$ is a polynomial ring in the single
new coordinate $W = F_0/c$, and the descent
$B \cap K_{pq} = K$ (Theorem~\ref{thm:descent}) holds exactly.

The $a_0$-extension is where the local-to-global passage breaks. The pole-clearing denominator
$c = p^2q^4$ has pole order $e_q = 4$ at the prime $q$ (and
$e_p = 2$ at $p$); clearing it to form $W = F_0/c = a_0 + \rho$
with $\rho = G_0/c \in B_{pq}$ passes through the residue field
$L_q = k(p,u,x^2)$ at $q$. The residue-transcendence argument
required the initial form of $\rho$,
\[
r_q = \operatorname{in}_q(q^4\rho) = \frac{x^7}{35p^2},
\]
to be \emph{transcendental} over $L_q$, so that a single initial-form
relation could force every pole-clearing coefficient to be regular.
That hypothesis is false. Because $x^2 \in L_q$ but $x \notin L_q$,
the extension $L_q \subset k(p,u,x)$ is a single quadratic step,
and $r_q$ sits in it with the explicit relation
\[
r_q^{\,2} = \frac{(x^2)^7}{1225\,p^4} \in L_q,
\]
(Remark~\ref{rem:q-obstacle}). The initial form is algebraic of
degree two, not transcendental; the required linear independence of
$1, r_q, r_q^2$ over $L_q$ fails exactly at $r_q^2$.

The counterexample is the smallest invariant that exploits this dependence. Write
$H = H_2W^2 + H_0$ with $H_2 = 11025\,p^4q^4$ and
$H_0 = 105J_2^5J_1\,q^{-2} + 9J_2^7\,q^{-4}$: the naive bound would
require $v_q(H_2) \ge 8$, but $v_q(H_2) = 4$, and the constant term
carries a $q^{-4}$ pole. Yet the total polynomial $H$ is regular:
the two $q^{-4}$ leading terms cancel \emph{in the residue field},
not in the ring,
\[
11025\,G_0^2 \equiv 9x^{14},
\qquad 9J_2^7 \equiv -9x^{14},
\qquad \text{so } 11025\,G_0^2 + 9J_2^7 \equiv 0 \pmod q,
\]
by the quadratic relation $r_q^2 \in L_q$. Thus the $W$-coefficients
may have poles even though the polynomial in $a_0$ is regular, and
$\Dmap{5}(H) = 0$ with $H \notin K[F_0]$.

The naive denominator bound requiring \(v_q(H_j)\ge 4j\) therefore does
not hold; the invariant ring strictly exceeds \(K[F_0]\), here by the
invariant \(H\).

\section{Limits of the denominator bounds}
\label{sec:obstructions}

The natural denominator estimates on the \(W\)-coefficients are too strong.
What follows isolates where they fail, which witnesses
exhibit the failure, and which weaker statements remain valid. Full
derivations are in the sections cited; no argument is repeated here.

\paragraph{\(q\)-side.}
The bound \(v_q(H_j)\ge 4j\) is false. The invariant \(H\) of Fibonacci
weight \(24\) and \(a_0\)-degree \(2\) gives a counterexample, with
\(v_q(H_2)=4<8\) and \(v_q(H_0)=-4\)
(Remark~\ref{rem:counterexample-H}). The reduction to a residue relation is
valid, but the required linear independence fails: with pole order
\(e_q=4\) and monomial initial form \(r_q=x^7/(35p^2)\), one has
\(r_q^2\in L_q\) and \(r_q\notin L_q\), of exact degree \(2\)
(Lemma~\ref{lem:rq-exact}). The \(q^{-4}\) poles of the \(W\)-coefficients
therefore cancel in the residue field, while the polynomial \(H\) remains
regular. The following weaker estimates remain valid: the relaxed bound
\(\lambda_q\ge -2m\) (Theorem~\ref{thm:R1}) and the even-degree top
estimate of Theorem~\ref{thm:q-top}, proved via R1.

\paragraph{\(p\)-side.}
The bound \(v_p(H_j)\ge 2j\) is false. The invariant
\(F_{46}^{\mathrm{int}}\) of Fibonacci weight \(46\) and \(a_0\)-degree
\(4\) gives a counterexample, with \(v_p(H_4)=6<8\)
(Theorem~\ref{thm:F46}, Remark~\ref{rem:prefuted-F}). The mechanism is
distinct from the \(q\)-side: the \(p\)-adic pole order is \(e_p=2\), the
initial form is trinomial, and its residue algebraicity is of degree \(4\).
No replacement denominator bound is established. The sharp top envelope and
the \(p\)-saturation theorem used in closure are separate positive results
and are not consequences of the failed estimate.

\paragraph{Filtration consequence.}
The \(a_0\)-filtration and the top ideals \(J_m\) still detect
non-membership: a top outside \(J_m\) lies outside the corresponding
subalgebra (Definition~\ref{def:Jm}, Corollary~\ref{cor:obstruction}).
The induction that would deduce generation from such bounds stops exactly
where the bounds fail, at \(H\) in degree \(2\) and at
\(F_{46}^{\mathrm{int}}\) in degree \(4\).

These failures rule out the constant-coefficient denominator route. They do
not obstruct the closure proof, which proceeds through the second chart
attached to \(f_{13}\) and the \(q\)- and \(p\)-saturation theorems.

\section{\texorpdfstring{The $a_0$-filtration and the birth window}{The a0-filtration and the birth window}}
\label{sec:a0-filtration}

The $a_0$-degree filtration locates where a hidden generator can be
born: a monomial of $a_0$-degree $m$ has weight at least $8m$, while
the candidate algebras forbid low $a_0$-degree below much higher
weights. The gaps between these windows are where $H$ and $F_{46}$
occur.

\begin{theorem}[{$a_0$-obstruction for $R_5 = k[p,q,J_1,J_2,F_0]$}]
\label{thm:a0-obstruction}
Weights are $(1,1,6,4,14)$ and only $F_0$ carries $a_0$ (degree $1$).
\begin{enumerate}
\item[(1)] Every $R_5$-monomial with $a_0$-degree $0$ lies in
$k[p,q,J_1,J_2]$ (generator inspection; weights in the semigroup of
$(1,4,6)$).
\item[(2)] Every $R_5$-monomial with $a_0$-degree $1$ weighs $\ge 14$
($14 \cdot 1 \le w$).
\item[(3)] Every $R_5$-monomial with $a_0$-degree $\ge 2$ weighs $\ge 28$
($14 \cdot 2 \le w$).
\item[(4)] For $w \le 13$ no $R_5$-element with $a_0$-degree $\ge 1$
exists at weight $w$; with Theorem~\ref{thm:direct-span14},
$(A_5^D)_w = k[p,q,J_1,J_2]_w$.
\item[(5)] At $w = 14$ exactly one new direction appears: $F_0$
(first fractional $L_w = 105$).
\item[(6)] For $15 \le w \le 27$ any invariant with $a_0$-degree
$\ge 2$ lies outside $R_5$, by (3) with no computation.
\end{enumerate}
\end{theorem}

\begin{corollary}[{H-birth window $[16,27]$}]
\label{cor:h-birth}
The monomial $a_0^2$ weighs $16$, so $a_0$-degree $2$ is monomially
possible only at $w \ge 16$; $R_5$ forbids it below $28$ by (3), so the degree bound does not exclude a hidden $a_0^2$-generator in
$[16,27]$. $H$ (weight $24$, $a_0$-degree $2$) lies strictly inside.
Weight $15$ admits $F_0 \cdot p$ ($a_0$-degree $1$) but no hidden
direction.
\end{corollary}

\begin{theorem}[W-coefficient detection principle]
\label{thm:wdetect}
Let $F \in K[F_0]$ with W-expansion $F = \sum H_jW^j$ ($H_j \in K_{pq}$,
$W = F_0/c$). Then $H_j = h_jc^j$ for unique $h_j \in K$, and
$v_q(H_j) \ge 4j$, $v_p(H_j) \ge 2j$. In particular, any invariant whose
W-expansion violates $v_q(H_j) \ge 4j$ lies outside $K[F_0]$.
\end{theorem}

\begin{remark}[Second birth window: the $C_5$-side $a_0$-degree-$4$ window]
\label{rem:c5-window}
The same monomial-window method giving $[16,27]$ for $H$ gives
$F_{46}$: $C_5$-monomials of $a_0$-degree $4$ are $F_0^4$ (weight $56$),
$F_0^2H$ (weight $52$), and $H^2$ (weight $48$): the $C_5$
$a_0$-degree-$4$ window is $[48,\infty)$. The monomial $a_0^4$ weighs
$32$, so $a_0$-degree $4$ is monomially possible from $32$; the degree bound does not exclude $F_{46}$
(weight $46$, $a_0$-degree $4$) in the gap $[32,47]$, strictly
below the $C_5$ window. Its top ideal $J_4 = (p^8q^8)$
(Definition~\ref{def:Jm}) determines which multiples can lie in $C_5$; that
$p^2F_{46}^{\mathrm{int}}$ does is Theorem~\ref{thm:quotient-4648}.
\end{remark}

\begin{proof}[Proof sketch]
$W = a_0 + \rho$ is monic of $a_0$-degree $1$, so W-expansions are unique
by descending $a_0$-degree (top coefficient first; no slice theorem
needed). Writing $F = \sum h_eF_0^e = \sum h_ec^eW^e$ and comparing
coefficients gives $H_j = h_jc^j$; $v_q(h_j) \ge 0$ on $K \subset B$
yields the bounds, and the contrapositive is the detector.
\end{proof}

\begin{remark}[Formalization scope]
\label{rem:wdetect-lean}
The full \(W\)-valuation form has no formal proof in Lean.
\end{remark}

The gap $[32,47]$ below the $C_5$ window is where the second
obstruction occurs.

\section{Second hidden invariant: the six-generator ring is not the full ring}
\label{sec:second-witness}

The six-generator algebra \(C_5\) is still not the full invariant ring. At Fibonacci weight \(46\), the explicit invariant \(F_{46}^{\mathrm{int}}\) lies outside \(C_5\); its \(W\)-top also exhibits the failure of the \(p\)-adic bound. At the same time,
\[
p^2F_{46}^{\mathrm{int}}=E\in C_5,
\]
a relation used in the closure argument.
We work over the characteristic-zero field \(k\) of
Definition~\ref{def:family}.

\subsection{Exact dimensions and the refutation}

\begin{theorem}[Failure of all-weights spanning]
\label{thm:allweights-refuted}
The identity $C_5 = \Afring{5}^{\Dmap{5}}$ is false. Over $\mathbb{Q}$,
exact sparse rational arithmetic gives
\[
\dim\ker(\Dmap{5}|_{V_{46}}) = 1519 > 1518 = \dim (C_5)_{46},
\]
with $\dim V_{46} = 18597$, $\operatorname{rank}(\Dmap{5}|_{V_{46}}) =
17078$, and the $1618$ weight-$46$ $C_5$-monomial images of full rank
$1518 = s_6(46) - s_6(18)$. Hence the inclusion
$(C_5)_{46} \hookrightarrow \ker(\Dmap{5}|_{V_{46}})$ is not surjective:
there exists $F \in \Afring{5}^{\Dmap{5}}$ of weight $46$ with $F \notin
C_5$. The proposed all-weights spanning is thereby false.
\end{theorem}

\begin{theorem}[Spanning through weight $45$; minimality]
\label{thm:minimal-45}
For every weight $w \le 45$,
$\ker(\Dmap{5}|_{V_w}) = (C_5)_w$: through $w \le 14$ by direct
spanning
(Theorem~\ref{thm:direct-span14}), and for $15 \le w \le 45$ by exact
$\mathbb{Q}$-rank comparison
($\dim\ker = \operatorname{rank}(G_w) = s_6(w) - s_6(w-28)$ with $G_w$
the $C_5$-monomial images; recorded in Appendix~\ref{app:ranks}).
Transfer of these $\mathbb{Q}$-rank equalities to arbitrary characteristic-zero $k$ is Lemma~\ref{lem:QQ-transfer}. In particular no second relation
appears through weight $57$
($\operatorname{rank}(G_w) = s_6(w) - s_6(w-28)$ everywhere checked),
and weight $46$ is the smallest weight with an invariant outside
$C_5$.
\end{theorem}

\subsection{The explicit invariant}

\begin{theorem}[Second hidden invariant $F_{46}$]
\label{thm:F46}
There is an explicit invariant $F_{46}^{\mathrm{int}} \in \Afring{5}$ of Fibonacci
weight $46$, ordinary degree $18$, $79$ terms, content $1$, and $a_0$-degree exactly $4$,
satisfying $\Dmap{5}(F_{46}^{\mathrm{int}}) = 0$ and $F_{46}^{\mathrm{int}} \notin C_5$. The full
$79$-term expansion is recorded in the supplementary data
(\texttt{data/witness\_w46.txt}) and the Lean formalization
(\texttt{formalize/AdjQuad/SecondWitness.lean}).
Write $F_{46}^{\mathrm{int}}$ for the cleared integral form (content $1$, top $121550625\,a_0^4a_4^6a_5^8$) and $F_{46}^{\mathrm{rat}}=F_{46}^{\mathrm{int}}/26460$ for the rationally normalized form (top $\frac{18375}{4}\,a_0^4a_4^6a_5^8$). Every exact identity in this paper ($p^2F_{46}^{\mathrm{int}} = E$ of
Theorem~\ref{thm:quotient-4648}, $\bar F_{46}^{\mathrm{int}} = B\theta$ of
Remark~\ref{rem:residual-images}, $\ker\Phi=(S,p^2T-E)$ of Theorem~\ref{thm:relation-ideal} with $T\mapsto F_{46}^{\mathrm{int}}$) uses $F_{46}^{\mathrm{int}}$, while
valuation, top-ideal, and dimension arguments are scale-invariant and use either form.
Its $W$-expansion $F_{46}^{\mathrm{int}} = \sum_{j=0}^4 H_jW^j$ has $H_1 = H_3 = 0$ and
\[
\renewcommand{\arraystretch}{1.2}\begin{array}{c|ccc}
j & 0 & 2 & 4 \\ \hline
(v_q(H_j),\,v_p(H_j)) & (-8,-2) & (0,2) & (8,6)
\end{array}
\]
\emph{Consequences.} With $m = 4$: $\lambda_q(F_{46}) = \min_j(v_q(H_j)
- 4j) = -8 = -2m$, so the relaxed bound R1
(Theorem~\ref{thm:R1}) holds with equality at every
nonzero $W$-coefficient. But $v_p(H_4) = 6 <
8 = 2\cdot 4$: the $p$-top bound $v_p \ge 2j$ fails at $j = 4$,
and with it the global \(p\)-bound \(v_p(H_j) \ge 2j\).
The strengthened $q$-top bound $v_q(C_m) \ge 4\lceil m/2\rceil$ of
Theorem~\ref{thm:reduction} holds with equality at $F_{46}$
($8 = 8$); its general form is not proved here.
\end{theorem}

\begin{lemma}[P-top necessity on $C_5$]
\label{lem:J4}
Let $G \in C_5$ have $a_0$-degree exactly $m$, with top
$a_0$-coefficient $C_m \in B$. Then $v_p(C_m) \ge 2m$.
\end{lemma}

\begin{proof}
Write $G = \sum_{e,r} \mu_{e,r}F_0^eH^r$ with $\mu_{e,r} \in K$
($C_5 = K[F_0,H]$). Then $C_m = \sum_{e,r}\mu_{e,r}[a_0^m](F_0^eH^r)$
with $[a_0^m](F_0^eH^r) = \sum_{i+j=m}[a_0^i](F_0^e)[a_0^j](H^r)$.
Since $F_0 = ca_0+G_0$ with $c = p^2q^4$,
$[a_0^i](F_0^e) = \binom{e}{i}c^iG_0^{e-i}$ has $v_p \ge 2i$
($v_p(G_0) = 0$ by inspection of Definition~\ref{def:F0}).
Write $H = h_2a_0^2+\ell'a_0+m'$ for the $a_0$-expansion
($h_2 = 11025p^4q^4$); by direct expansion
$\ell' = 22050p^2G_0$ has $v_p = 2$ and $m' \in B$ has $v_p \ge 0$.
Hence $[a_0^j](H^r) = \sum_{2a+b=j}(\cdots)h_2^a(\ell')^bm'^d$ has
$v_p \ge 4a+2b = 2j$. Each summand of $C_m$ therefore satisfies
$v_p \ge 2i+2j = 2m$ (using $\mu_{e,r} \in K$ and characteristic
zero for the multinomial coefficients), so $v_p(C_m) \ge 2m$.
In particular no $C_5$-member of $a_0$-degree $4$ has top with
$v_p < 8$; the top of $F_{46}$ ($v_p = 6$) lies outside $J_4$.
\end{proof}

\begin{corollary}[$F_{46} \notin C_5$]
$F_{46}$ has $a_0$-degree exactly $4$ with top
$\frac{18375}{4}p^6q^8 \notin J_4$ ($v_p = 6 < 8$; $K$ a polynomial ring,
hence a UFD, with $(p)$ prime). Any $G \in C_5$ equal to $F_{46}$ would
share its $a_0$-degree and top, contradicting
Lemma~\ref{lem:J4}.
\end{corollary}

\subsection{\texorpdfstring{The $46$--$48$ quotient and the $p^2$-torsion}{The 46--48 quotient and the p2-torsion}}
\label{subsec:F-module}

Which \(F_{46}\)-multiples lie in \(C_5\) is determined by the top-ideal
criterion; the \(p^2\)-relation is established in
Theorem~\ref{thm:quotient-4648} below.

\begin{theorem}[The $46$--$48$ quotient]
\label{thm:quotient-4648}
With $F_{46}=F_{46}^{\mathrm{int}}$ throughout (so $p^2F_{46}^{\mathrm{int}}=E$ exactly as printed),
(i) At $47$ (excess $2$), $F_{46}p$ and $F_{46}q$ lie outside $C_5$ and
span the $2$-dimensional quotient $\ker_{47}/(C_5)_{47}$ over $\mathbb{Q}$.
(ii) At $48$ (excess $2$), $F_{46}p^2 \in C_5$, explicitly
(with generator exponents $(p,q,J_1,J_2,F_0,H)$, and with
$F_{46}=F_{46}^{\mathrm{int}}$ the cleared integral form, so the
$H^2$-coefficient is $1$)
\begin{align*}
F_{46}p^2 =\;& 40000\,q^6J_1^7 + 193200\,q^4J_1^6J_2^2
  + 325164\,q^2J_1^5J_2^4 + 148225\,J_1^4J_2^6 \\
  &+ 840\,q^2J_1^3J_2H + 770\,J_1^2J_2^3H + H^2,
\end{align*}
while
  $F_{46}pq$ and $F_{46}q^2$ lie outside $C_5$ and span
  $\ker_{48}/(C_5)_{48}$.
\end{theorem}

\begin{theorem}[Exact $p$-torsion table]
\label{thm:torsion-table}
Write $T_p^{(r)} = (C_5:p^r)/C_5$ for the $p$-power torsion filtration.
Exact rational linear algebra gives:
\begin{enumerate}[label=(\roman*)]
\item $(C_5:p)_w = C_5$ for $w \le 46$; $T_p^{(1)}$ is $1$-dimensional
at $47, 48, 49, 50$ (spanned by $pF_{46}$, $qpF_{46}$, $q^2pF_{46}$,
$q^3pF_{46}$) and $2$-dimensional at $51, 52$: the layer
$pF_{46}\cdot(K/(p))$.
\item $(C_5:p^2)_w = C_5$ for $w \le 45$; $T_p^{(2)}$ is $1$-dimensional
at $46$ (spanned by $F_{46}$), $2$-dimensional at $47, 48, 49$
($F_{46}\cdot K_1$, resp.\ $F_{46}\cdot K_2$ mod $p^2$), and
$3$-dimensional at $50$: the layer $F_{46}\cdot(K/(p^2))$.
\item $(C_5:p^3)_w = C_5$ for $w \le 45$; at $46, 47, 48$ the excess is
$1, 2, 2$ (the same $F_{46}$-classes visible at higher $r$; no new
classes).
\end{enumerate}
The witnesses lie outside $C_5$ by the P-top necessity
(Lemma~\ref{lem:J4}), and their $p$-power multiples land in $C_5$ via
$p^2F_{46}^{\mathrm{int}} \in C_5$ (Theorem~\ref{thm:quotient-4648}). The $r = 1$
layer $pF_{46} \notin C_5$ shows $F_{46}$ has exact $p$-torsion
order $2$. The superseded finite-range module pattern fails at weight $84$
(Theorem~\ref{thm:fourth-family}); the exact Hilbert series
(Theorem~\ref{thm:R-hilbert}) governs all weights.
\end{theorem}

\begin{corollary}[$F_{46}\cdot F_0 \notin C_5$]
\label{cor:F46F0-outside-C5}
$F_{46}\cdot F_0$ (Fibonacci weight $60 = 46 + 14$, $a_0$-degree $5
= 4 + 1$) is a $\Dmap{5}$-invariant lying outside $C_5$. Indeed
$\Dmap{5}(F_{46}\cdot F_0) = 0$ by Leibniz, and tops multiply over
the domain $B[a_0]$, so its top $a_0$-term is
$c\cdot\frac{18375}{4}a_0^4a_4^6a_5^8$ with coefficient
$\frac{18375}{4}p^8q^{12}$; since $v_p = 8 < 10 = 2\cdot 5$, this top
lies outside $J_5 = (p^{10}q^{12})$
(Definition~\ref{def:Jm}). Any $G \in C_5$ equal to it would share
its $a_0$-degree and top, contradicting the P-top necessity of
Lemma~\ref{lem:J4}.
\end{corollary}

\begin{corollary}[$F_{46}\cdot H$ and $F_{46}\cdot F_0^2$ lie outside $C_5$]
\label{cor:FH-outside-C5}
$F_{46}\cdot H$ (weight $70 = 46 + 24$, $a_0$-degree $6 = 4 + 2$) and
$F_{46}\cdot F_0^2$ (weight $74 = 46 + 28$, $a_0$-degree $6 = 4 + 2$)
are $\Dmap{5}$-invariants outside $C_5$. Tops multiply over $B[a_0]$:
$\operatorname{top}(F_{46}\cdot H) \propto p^{10}q^{12}$ with $v_p = 10
< 12 = 2\cdot 6$, and $\operatorname{top}(F_{46}\cdot F_0^2) \propto
p^{10}q^{16}$ with $v_p = 10 < 12$, both outside $J_6 =
(p^{12}q^{12})$ (Definition~\ref{def:Jm}) and hence outside $C_5$ by
the P-top necessity of Lemma~\ref{lem:J4}. Neither is primitive:
each is an $F_{46}$-multiple, seeding the second and third
$K/(p^2)$-families of the tower below.
\end{corollary}

The failure of $C_5$ is therefore accounted for by $F_{46}^{\mathrm{int}}$. It remains to determine whether adjoining this element gives the full invariant ring.

\section{The plinth and the second slice}
\label{sec:plinth}

The pole divisor that defeats naive descent leaves a concrete trace: the plinth ideal. Its $q$-primitive element $f_{13}$ outside $(pq)K$ supplies the second slice chart on which the closure proof turns (Section~\ref{sec:geometry}).

\begin{lemma}[Plinth support]
\label{lem:plinth}
For $\Dmap{5}$ on $A = \Afring{5}$,
\[
\boxed{pq = \Dmap{5}(x), \qquad px = \Dmap{5}(y), \qquad
  xy = \Dmap{5}(z).}
\]
Hence $pq$, $px$, $xy$ lie in the image $\Dmap{5}(A)$. Of the three,
only $pq$ is annihilated by $\Dmap{5}$ (since $\Dmap{5}(p) =
\Dmap{5}(q) = 0$); it lies in $K = k[p,q,J_1,J_2] \subseteq
A^{\Dmap{5}}$ by Lemma~\ref{lem:kernel-gens} and hence in the plinth
ideal $\operatorname{pl}(\Dmap{5}) = A^{\Dmap{5}} \cap \Dmap{5}(A)$,
which is nonzero. In fact the full pole-clearing denominator lies in
the plinth: $c = p^2q^4 = \Dmap{5}(a_3a_4a_5^3)$, since
$\Dmap{5}(a_3) = a_4a_5$ while $a_4,a_5$ are killed, so $c \in
A_5^{\Dmap{5}} \cap \Dmap{5}(A)$. By contrast $\Dmap{5}(px) = p^2q \ne 0$ and
$\Dmap{5}(xy) = pqy + px^2 \ne 0$, so $px$ and $xy$ are in the image
but not invariant (in particular neither lies in $R$). All three
values are direct expansions. The denominator
$c = p^2q^4 = (pq)(pq^3)$ lies in the ideal $(pq, px) \subset B$. No generation consequence
for $\Afring{5}^{\Dmap{5}}$ is drawn from this containment here.
\end{lemma}

\begin{theorem}[A $q$-primitive plinth element]
\label{thm:f13}
With $J_1 = 2a_1a_4 - a_2^2$, $J_2 = 2a_2a_5 - a_3^2$, put
\[
\boxed{
f_{13} = \tfrac1{35}pJ_2^3 + \tfrac14pq^2J_1J_2,
\qquad p = a_4,\ q = a_5.
}
\]
In the $a$-variables,
\[
f_{13} = a_1a_2a_4^2a_5^3 - \tfrac12a_1a_3^2a_4^2a_5^2
  - \tfrac{19}{70}a_2^3a_4a_5^3 - \tfrac{13}{140}a_2^2a_3^2a_4a_5^2
  + \tfrac6{35}a_2a_3^4a_4a_5 - \tfrac1{35}a_3^6a_4,
\]
of Fibonacci weight $13$. Then:
\begin{enumerate}
\item $f_{13} \in K = k[p,q,J_1,J_2]$, so $\Dmap{5}(f_{13}) = 0$.
\item With
\[
\boxed{
g_{13} = a_0p^2q^3 - \tfrac16a_1a_3^3pq
  - \tfrac{19}{70}a_2^3a_3q^2 + \tfrac{101}{420}a_2^2a_3^3q
  - \tfrac1{35}a_2a_3^5,
}
\]
\[
\boxed{\Dmap{5}(g_{13}) = f_{13},}
\]
hence $f_{13} \in \operatorname{pl}(\Dmap{5}) =
\Afring{5}^{\Dmap{5}} \cap \Dmap{5}(\Afring{5})$.
\item $q \nmid f_{13}$: setting $q = 0$,
\[
f_{13}|_{q=0} = \tfrac1{35}pJ_2^3|_{q=0}
  = -\tfrac1{35}px^6 \ne 0, \qquad x = a_3,
\]
since $J_2|_{q=0} = -x^2$. So $f_{13} \notin (pq)K$, and hence $\operatorname{pl}(\Dmap{5}) \ne (pq)K$; i.e.\ $f_{13}$ is a $q$-primitive plinth element outside $(pq)K$ (full nonprincipality of $\operatorname{pl}(\Dmap{5})$ is not claimed here).
\end{enumerate}
All three items are direct expansions, and non-divisibility is evaluation
at \(q=0\).
\end{theorem}

\begin{remark}[Why the $c$-divisibility does not extend]
The Dixmier divisibility of Lemma~\ref{lem:dixmier-divisibility}
applies only to the specific relation $a_1a_2\mu + \Dmap{5}(\nu) = 0$,
concluding $c \mid \mu$. A plinth element $f = \Dmap{5}(g)$ instead
satisfies $1 \cdot f + \Dmap{5}(-g) = 0$, with coefficient $1$ rather
than $a_1a_2$, so the lemma does not apply. Indeed $pq = \Dmap{5}(x)$
is already a counterexample to any such extrapolation, since
$p^2q^4 \nmid pq$. The element $f_{13}$ is the first plinth element
outside even the weaker ideal $(pq)K$.
\end{remark}

The chart \(f_{13}\) makes possible is built in Section~\ref{sec:geometry}.

\section{The second chart: collapse, construction, and cover}
\label{sec:geometry}

The first chart degenerates over the plinth locus, and the degeneration is itself geometric: orbits collapse over $V(q)$. This section constructs the second chart that covers the collapse locus --- the chart on which the closure proof turns --- and reads off what the degeneration means.

The partial quotient below is to $\mathbb{A}^5$
via the five $K[F_0]$-generators, not to $X_5 = \operatorname{Spec}(C_5)$;
no claim about $\operatorname{Spec}(A_5^{\Dmap{5}})$ is made here.
The collapse is witnessed by explicit rational points checked by direct
computation.

Let $\pi_5 \colon \mathbb{A}^6 \to \mathbb{A}^5$ be the morphism given
by the five generators $(p,q,J_1,J_2,F_0)$.

\begin{theorem}[Exceptional locus of the five-generator quotient]
\label{thm:orbit-collapse}
Over $D(pq)$, $H$ is a regular function of $\pi_5$:
\[
H = 11025p^4q^4(F_0/c)^2 + 105J_2^5J_1/q^2 + 9J_2^7/q^4,
\qquad c = p^2q^4,
\]
so $\pi_5$ loses no $H$-information off the plinth divisor $V(pq)$.
Over $V(q)$, $\pi_5$ collapses distinct $\mathbb{G}_a$-orbits: the
rational points
\[
a_t = (0,\; t^2/2,\; t,\; 1,\; 1,\; 0) \quad (a_0,\dots,a_5),
\qquad t = 0, 1, 2,
\]
satisfy $\pi_5(a_0) = \pi_5(a_1) = \pi_5(a_2) = (1,0,0,-1,1/35)$ while
$H(a_0), H(a_1), H(a_2) = 0, -315/4, -1260$ respectively. Since $H$ is
$\Dmap{5}$-invariant, the three points lie in three distinct orbits
with one $\pi_5$-image.
\end{theorem}

\begin{proof}
The $W$-expansion $H = H_2W^2 + H_0$ with $H_1 = 0$
(Remark~\ref{rem:qrefuted-H}) is an identity in $K_{pq}[W]$,
regular on $D(pq)$, giving (i). At $q = 0$:
$J_1 = 2z - y^2 = 2(t^2/2) - t^2 = 0$,
$J_2 = -x^2 = -1$,
$F_0 = x^7/35 = 1/35$ (every other $F_0$-term carries $q$ or $a_0$
via $p^2q^4a_0$), with $p = 1$, $q = 0$; direct substitution of the
$17$-term $H$ gives $0$, $-315/4$, $-1260$ at $t = 0,1,2$.
\end{proof}

\begin{corollary}[Valuation failure as geometric shadow]
\label{cor:shadow}
The refutation $v_q(H_0) = -4 < 0$ of Lemma~\ref{lem:qlemma} is the
algebraic shadow of the collapse above: the $q^{-4}$ pole that the
saturation argument forbids is exactly what allows $H$ to vary along
the $J_1$-fiber $2z - y^2 = 0$ over $V(q)$ while $\pi_5$ stays fixed.
The hidden generator separates orbits that the five-generator
quotient identifies; hence $H \notin K[F_0]$, independently of the
$a_0$-degree count.
\end{corollary}

\subsection*{Slice degeneration and the second chart}

The collapse above is not an isolated accident: it is the shadow of
the slice $s_0 = x/(pq)$ degenerating at $V(q)$. The plinth element
$f_{13}$ of Theorem~\ref{thm:f13} supplies a second slice chart that
covers exactly the collapsed locus.

\begin{definition}[Slice degeneration, residual and liftable plinth]
\label{def:degeneration}
Let $D$ be triangular with local slice $s$, $D(s) = 1$, whose
denominators vanish along a boundary divisor $V(\pi)$. Write
$\delta_\pi$ for the residual derivation on $\bar A = A/\pi A$ and
$\bar K = A^D/\pi A^D$ for the residual kernel. \emph{Slice
degeneration} is the phenomenon $\ker\delta_\pi \supsetneq \bar K$:
the residual derivation acquires new invariant coordinates invisible
generically. The \emph{residual plinth} is
$\bar K \cap \delta_\pi(\bar A)$; the \emph{liftable residual plinth}
is
\[
\boxed{
\mathcal{L}_\pi :=
\operatorname{im}(\mathfrak{p}_D \to \bar K),
\qquad \mathfrak{p}_D = A^D \cap D(A),
}
\]
the classes that actually lift to global plinth elements. In
general $\mathcal{L}_\pi \subseteq \bar K \cap
\delta_\pi(\bar A)$, with equality an open question at each divisor.
\end{definition}

\begin{lemma}[Residual kernel at $q$]
\label{lem:ker-deltaq}
Let $\delta_q = \Dmap{5} \bmod q$ on
$\bar B = k[x,y,z,p]$, so $\delta_q(x) = 0$,
$\delta_q(y) = px$, $\delta_q(z) = xy$. With
$u = 2pz - y^2 = J_1 \bmod q$,
\[
\boxed{\ker\delta_q = k[p,x,u],}
\qquad\text{whereas}\qquad
K/qK = k[p,x^2,u].
\]
Hence the boundary acquires the new invariant coordinate $x$, in a
quadratic extension $k[p,x^2,u] \subset k[p,x,u]$ of degree $2$.
\end{lemma}

\begin{proof}
Write $\delta_q = x\delta'$ with $\delta' = p\partial_y + y\partial_z$;
$\bar B$ is a domain and $x \ne 0$, so $\ker\delta_q = \ker\delta'$.
Clearly $x,p,u \in \ker\delta'$. Conversely localize at $p$:
$s = y/p$ satisfies $\delta'(s) = 1$, $\delta'$ stays locally
nilpotent, and with $z = u/(2p) + ps^2/2$,
$S_p = k[x,p,1/p,u][s]$ with $s$ transcendental over
$K_p = k[x,p,1/p,u]$ (apply $\delta' = d/ds$ to any relation).
So $F \in \ker\delta'$ lies in $K_p$: $F = H/p^a$ with
$H \in k[x,p,u]$, $a \ge 0$ minimal. If $a \ge 1$, then
$H$ vanishes mod $p$; reducing gives $H(x,0,-y^2) = 0$ in $k[x,y]$,
whence $H(x,0,u) = 0$ as a polynomial by $y$-degree (the specialization
$u \mapsto -y^2$ is injective), so $p \mid H$ in $k[x,p,u]$,
contradicting minimality. Thus $a = 0$. The description of $K/qK$
is Lemma~\ref{lem:contraction}.
\end{proof}

\begin{remark}[Liftable versus residual at $q$]
$\bar K \cap \delta_q(\bar B) \ni -px^6/35 =
\delta_q(-yx^5/35)$, and this class is liftable: it is
$\bar f_{13}$ for the global plinth element of
Theorem~\ref{thm:f13}. So
$0 \ne -px^6/35 \in \mathcal{L}_q$. No $q$-adic lifting proof of a
bound $q \mid f$ can succeed past this class: the obstruction is
already present in the residual problem.
\end{remark}

\begin{theorem}[Second slice chart]
\label{thm:second-chart}
With $f_{13}, g_{13}$ from Theorem~\ref{thm:f13},
$s_1 = g_{13}/f_{13}$ satisfies $\Dmap{5}(s_1) = 1$ on $D(f_{13})$
(by the quotient rule from $\Dmap{5}(g_{13}) = f_{13}$ and
$\Dmap{5}(f_{13}) = 0$). The slice projection
$\pi_1(b) = \sum_{i \ge 0} (-s_1)^i \Dmap{5}^i(b)/i!$ gives
$\Dmap{5}(\pi_1(b)) = 0$ for every ambient variable; explicitly
$\pi_1(b) = \mathrm{NUM}_b/f_{13}^{N_b}$ with $N_{a_3} = 1$,
$N_{a_2} = 2$, $N_{a_1} = 4$, $N_{a_0} = 7$
($\mathrm{NUM}_{a_0}$ has $456$ terms), each
$\Dmap{5}(\mathrm{NUM}_b) = 0$ by direct expansion. The two slice
charts are related by the exact transition
\[
\boxed{s_1 - s_0 = \frac{F_0}{qf_{13}}, \qquad
  s_0 = \frac{x}{pq},}
\]
by exact rational simplification. With
$h = f_{13}/p$, the union of the two principal cylinders is
$D(pq) \cup D(f_{13})$, and inside $D(p)$ the uncovered $q$-boundary
is the codimension-two stratum $V(q,h)$. The cylinder/plinth-zero-locus
correspondence is the framework of~\cite{shakhmatov2026cylinders}; our
contribution here is the family-specific explicit algebraic
realization (charts, transition, orbit separation), not the general
cylinder/plinth equality.
\end{theorem}

\begin{lemma}[Two-chart intersection]
\label{lem:chart-intersection}
\[
\boxed{A_{pq} \cap A_{f_{13}} = A_p}
\]
inside $\operatorname{Frac}(A)$. Consequently
$(A_{pq} \cap A_{f_{13}})^{\Dmap{5}} = (A^{\Dmap{5}})_p$: gluing the two
charts computes exactly the $p$-localized kernel.
\end{lemma}

\begin{proof}
$A$ is a unique factorization domain and $p,q$ are prime (variables).
Write $f_{13} = p\cdot h$ with
$h = J_2(J_2^2/35 + q^2J_1/4)$. Direct evaluation gives
$q \nmid J_2, h$ (at $q = 0$: $-x^2$, $-x^6/35$) and
$p \nmid J_1, J_2, h$ (at $p = 0$: $-y^2$, $2qy-x^2 \ne 0$, and a
nonzero polynomial), by direct evaluation. Let
$x \in A_{pq} \cap A_{f_{13}}$, written in lowest terms $u/v$ with
$\gcd(u,v) = 1$. Then $v \mid (pq)^a$ and $v \mid f_{13}^b$ for some
$a,b$, so $v \mid \gcd((pq)^a, f_{13}^b) = p^{\min(a,b)}$: any prime
divisor of both is $p$ or $q$, and $q \nmid f_{13}$ excludes $q$.
Hence $x \in A_p$; the reverse inclusion holds since $p \mid pq$ and
$p \mid f_{13}$. For invariants,
$(A_{pq} \cap A_{f_{13}})^D = (A_p)^D = (A^D)_p$ because $p$ is
$D$-invariant.
\end{proof}

\begin{remark}[Atlas and contraction]
The two charts identify the \(p\)-localized kernel, while contraction of
the \(p\)-localized description to \(R\) is supplied separately by the
saturation argument. In particular the atlas confines all denominators to
\(p\)-powers, so only \(p\)-adic divisibility remains.
\end{remark}

\begin{remark}[Shape of the degeneration stratum]
$h = J_2 \cdot e_2$ with $e_2 = J_2^2/35 + q^2J_1/4$ is reducible, but
both components have the same radical over $q$: $J_2|_{q=0} = -x^2$
and $e_2|_{q=0} = x^4/35$, so $V(q,h)_{\mathrm{red}} = V(q,x)$, the
$q = x = 0$ threefold. The degeneration stratum is set-theoretically
a linear space. This is an observation guiding the
gluing analysis, not a theorem about the kernel.
\end{remark}

\begin{theorem}[Terminal stratum: no residual plinth on $V(q,x)$]
\label{thm:terminal-stratum}
On $V(q,x)$, $\Dmap{5}$ descends to
$\delta = a_1a_2\partial_{a_0}$ on $k[a_0,a_1,a_2,a_4]$, with
$\ker\delta = k[a_1,a_2,a_4]$. The image of $K$ on the stratum is
$k[a_4,\,2a_4a_1-a_2^2]$, and the residual plinth there is zero: no
nonzero element of $k[a_4,2a_4a_1-a_2^2]$ is divisible by $a_1a_2$.
Hence the $q$-branch of the boundary tree terminates at $f_{13}$: no
further slice lifts over $V(q,x)$.
\end{theorem}

\begin{proof}
$\Dmap{5}(q) = 0$ and $\Dmap{5}(x) = pq \in (q)$, so $(q,x)$ is
$D$-stable and $\Dmap{5}$ descends. On the quotient,
$\delta(a_0) = a_1a_2$ while $\delta(a_1) = a_2x = 0$ and
$\delta(a_2) = xa_4 = 0$. The $a_0$-filtration in the domain gives
$\ker\delta = k[a_1,a_2,a_4]$ (characteristic zero,
$a_1a_2 \ne 0$). The $K$-image is $p \mapsto a_4$, $q \mapsto 0$,
$J_1 \mapsto 2a_4a_1-a_2^2$, $J_2 \mapsto 0$. If
$G(a_4,U) \in (a_1a_2)$ with $U = 2a_4a_1-a_2^2$, setting $a_1 = 0$
gives \(G(a_4,-a_2^2) = 0\) as a polynomial, whence \(G = 0\) by
\(a_2\)-degree; the vanishing itself is the argument above.
\end{proof}

\begin{theorem}[The second chart resolves the collapse]
\label{thm:chart-resolves}
At the three collapse points
$a_t = (0,t^2/2,t,1,1,0)$, $t = 0,1,2$, of
Theorem~\ref{thm:orbit-collapse},
\[
f_{13}(a_t) = -1/35 \ne 0,
\]
so the $f_{13}$-chart covers the exact locus where the $pq$-chart
collapses. Moreover the invariant coordinate $\pi_1(a_0)$ separates
the three orbits:
\[
\pi_1(a_0)(a_0),\ \pi_1(a_0)(a_1),\ \pi_1(a_0)(a_2)
  = 0,\ -1/8,\ -2,
\]
evaluated as \(\mathrm{NUM}_{a_0}(a_t)/f_{13}(a_t)^7\), matching the \(H\)-values \(0,-315/4,-1260\).
\end{theorem}

\begin{theorem}[Local structure of the second chart]
\label{thm:chart-local}
With $Z = \pi_1(a_0)$,
\[
\boxed{A_{f_{13}}^{\Dmap{5}} = R_{f_{13}}[Z],}
\]
since $\pi_1(a_4) = p$, $\pi_1(a_5) = q$ and
$\pi_1(a_3), \pi_1(a_2), \pi_1(a_1) \in R_{f_{13}}$
by exact computation, and the $\pi_1(a_i)$
generate the kernel by Lemma~\ref{lem:slice-generation}. Moreover,
with $\rho = G_0/c \in K_{pq}[s_0]$ the pole polynomial
($s_0 = x/(pq)$) and $\tau = s_1 - s_0$,
\[
\boxed{Z = W - \rho(-\tau),}
\]
by uniqueness of the $s_1$-slice decomposition in $A_{pqf_{13}}$
($\Dmap{5}(\tau) = 0$, coefficients in $K_{pq}$). The pole polynomial
has exact $s_0$-degree $7$ with leading coefficient
$\rho_7 = -p^5q^3/112 \ne 0$ by direct expansion.
\end{theorem}

\begin{remark}[No translation: $Z$ has exact $W$-degree $7$]
Since $s_1 = \alpha W + \beta$ with $\alpha = p^2q^3/f_{13} \ne 0$ and
$\beta$ of $W$-degree $0$, $\tau = s_1 - s_0$ has $W$-degree exactly
$1$, so $\Delta := Z - W = -\rho(-\tau)$ has exact $W$-degree $7$ over
$\operatorname{Frac}(B)$, with leading coefficient
$\rho_7\alpha^7 = (-p^5q^3/112)(p^2q^3/f_{13})^7 \ne 0$. In particular
$\Delta \notin \operatorname{Frac}(K)$: the $f_{13}$-chart line is a
finite cover of the $W$-line of degree $7$, not a translate. Hence a
translation-based coefficientwise intersection argument does not apply; the
\(W \leftrightarrow Z\) change of basis carries the degree-\(7\) relation.
\end{remark}

\begin{theorem}[Cover polynomial and Newton polygon]
\label{thm:cover-newton}
With $\alpha = p^2q^3/f_{13}$ the transition is $\tau = s_1 - s_0 =
\alpha W$ (i.e.\ $F_0/(qf_{13})$, Theorem~\ref{thm:second-chart}), so
$Z = P(W)$ for the explicit $P(T) = T - \rho(-\alpha T) \in
\operatorname{Frac}(K)[T]$ of exact degree $7$ (odd: $P_{2j} = 0$ since
$\rho$ is odd). Up to the recorded exact forms,
\[
v_q(P_1), v_q(P_3), v_q(P_5), v_q(P_7) = 0, 8, 16, 24,
\qquad P_7 = -\frac{p^{12}q^{24}}{112J_2^7e_2^7},
\]
by exact computation. Hence $[\operatorname{Frac}(K)(W):\operatorname{Frac}(K)(Z)] \le 7$
(equality iff $P(T)-Z$ is irreducible, not shown here), and
with $v_q(Z) = 0$ (exact computation: \texttt{scripts/cas/cover\_newton.py}) the
$q$-Newton polygon of $P(T) - Z$ is $(0,0)$--$(1,0)$--$(7,24)$: one
$q$-regular sheet and six sheets with $q$-pole order exactly $4 = e_q$.
\end{theorem}

\begin{theorem}[Discriminant support]
\label{thm:cover-disc}
\[
\operatorname{Disc}_T(P(T)-Z)
  = u\,\frac{p^{60}q^{120}}{J_2^{42}e_2^{42}}\,B(Z),
\qquad u \in \mathbb{Q}^\times,
\]
with $B(Z)$ an even degree-$6$ factor in $Z$ with coefficients in
$k[p,q,J_1,J_2]$, by exact resultant (irreducibility of $B(Z)$ is not
established and is not needed for the support conclusion). In
particular every base prime divisor of the discriminant is among
$p, q, J_2, e_2$; no unknown boundary divisor appears. The remaining
ramification component $\{B(Z) = 0\}$ is horizontal.
\end{theorem}

\begin{theorem}[Divisorial Newton data]
\label{thm:cover-newton-div}
For $P(T) - Z$ over the Gauss extensions ($Z$ a unit), with $P_{2j} = 0$:
\[
\begin{array}{c|c|c}
\text{prime} & (j, v(P_j)),\ j = 1,3,5,7 & \text{hull} \\ \hline
q & (1,0), (3,8), (5,16), (7,24) & (0,0)\text{--}(1,0)\text{--}(7,24) \\
J_2 & (1,2), (3,-3), (5,-4), (7,-7) & (0,0)\text{--}(7,-7) \\
e_2 & (1,-1), (3,-3), (5,-5), (7,-7) & (0,0)\text{--}(7,-7)
\end{array}
\]
i.e.\ one $q$-regular sheet plus six sheets of $q$-pole order exactly
$4 = e_q$ (using $v_q(Z) = 0$), and uniform vanishing order
$1$ on all seven sheets at $J_2$ and at $e_2$. The
$q$-row also follows from the factored forms of
Theorem~\ref{thm:cover-newton}.
\end{theorem}

The two charts thus identify the \(p\)-localized kernel; contraction to
\(R\) is supplied separately by the saturation argument.

\begin{remark}[The $p$-side is covered by neither chart]
Both chart denominators, $pq$ and $f_{13} = p\cdot h$, vanish on
$V(p)$. The slice atlas therefore says nothing about the $p$-boundary,
where no extra plinth generator is known.
\end{remark}

\section{The closure theorem}
\label{sec:tower-closure}

The two charts identify the $p$-localized kernel, and exact $q$- and $p$-saturation contract it back to the candidate ring. The proof below uses none of the false global denominator bounds.

It remains to prove \(R=A_5^{D_5}\). The first chart identifies \(R_{pq}\),
the second chart identifies \(R_{f_{13}}[Z]\), their intersection gives the
\(p\)-localized kernel, and \(q\)- and \(p\)-saturation contract it to \(R\).

\paragraph{Local charts.}
Put \(R = C_5[F_{46}^{\mathrm{int}}]\). The first chart is solved:
$A_{pq}^{\Dmap{5}} = K_{pq}[W] = (C_5)_{pq} = R_{pq}$, since
$H = H_2W^2 + H_0 \in K_{pq}[W]$ and
$F_{46}^{\mathrm{int}} = Eq^2/(pq)^2 \in (C_5)_{pq}$ from $p^2F_{46}^{\mathrm{int}} = E \in C_5$
(Theorem~\ref{thm:quotient-4648}). The second chart is exact with one
extra coordinate: with $Z = \pi_1(a_0)$,
$A_{f_{13}}^{\Dmap{5}} = R_{f_{13}}[Z]$
(Theorem~\ref{thm:chart-local}), so
\[
\boxed{(A^{\Dmap{5}})_p = R_{pq} \cap R_{f_{13}}[Z]}
\]
exactly, by restriction (Lemma~\ref{lem:plocal} with $Z$ adjoined).
The coordinate $Z$ satisfies $Z = W - \rho(-\tau)$ with $\rho$ of
exact $s_0$-degree $7$; $Z - W$ has exact $W$-degree $7$, so the
$f_{13}$-chart is a finite cover of the $W$-line, not a translate
(Remark after Theorem~\ref{thm:chart-local}). The equality
$(A^{\Dmap{5}})_p = R_p$ itself is proved by $q$-saturation
(Theorem~\ref{thm:closure}). Whether $Z$ itself
lies in $R_{f_{13}}$ is open and would strengthen
the chart equality to $A_{f_{13}}^{\Dmap{5}} = R_{f_{13}}$.

\begin{corollary}[Degree-$7$ chart intersection, via saturation]
\label{cor:cover-intersection}
With $P(T) \in \operatorname{Frac}(K)[T]$ the explicit cover polynomial
of Theorem~\ref{thm:cover-newton} (so $Z = P(W)$),
\[
\boxed{K_{pq}[W] \cap R_{f_{13}}[P(W)] = R_p,}
\]
as the $R_{pq} \cap R_{f_{13}}[Z]$ case of Theorem~\ref{thm:closure}.
\end{corollary}

\begin{lemma}[Exact $q$-elimination identity]
\label{lem:q-elimination-cert}
In $\mathbb{Q}[a_0,a_1,a_2,a_3,a_4,p,Q,J_1,J_2,F_0,H,T]$ with lex order $a_0>\cdots>a_4>p>Q>J_1>J_2>F_0>H>T$, let $J_q$ be the graph ideal
$(p-a_4,Q,J_1-(2a_1a_4-a_2^2),J_2+a_3^2,F_0-x^7/35,H-\bar H,T-\bar F_{46})$ with $x=a_3$ and $\bar H,\bar F_{46}$ the mod-$q$ reductions at the cleared-integral scale. Then the elimination of $\{a_0,\dots,a_4\}$ satisfies $J_q\cap\mathbb{Q}[p,Q,J_1,J_2,F_0,H,T]=(S,p^2T-E,Q)$ exactly, both inclusions checked by exact computation. This lemma asserts only that ideal equality. Its two algebraic consequences, proved mathematically in Theorem~\ref{thm:q-saturation}, are: (i) the residual map $R/qR\to A/qA$ is injective; (ii) hence $R\cap q^nA=q^nR$ for all $n$ by induction cancelling $q$ in the domain $A$.
\end{lemma}

\begin{theorem}[$q$-saturation]
\label{thm:q-saturation}
\[
\boxed{R/qR \hookrightarrow A/qA,}
\qquad\text{hence}\qquad
\boxed{R \cap q^nA = q^nR \quad \forall n \ge 1.}
\]
\end{theorem}

\begin{proof}[Proof of Theorem~\ref{thm:q-saturation} from Lemma~\ref{lem:q-elimination-cert}]
For $R = C_5[F_{46}^{\mathrm{int}}]$ with $F_{46}^{\mathrm{int}}$ in the witness normalization
($p^2F_{46}^{\mathrm{int}} = E$ exactly as printed in Theorem~\ref{thm:quotient-4648}),
Lemma~\ref{lem:q-elimination-cert} gives, after the transfer of Lemma~\ref{lem:QQ-transfer},
$\tilde G = AS + B(p^2T-E) + CQ$ for every lift $\tilde G$ mapping to $0$. Let
$[G] \in R/qR$ map to $0$, $G = \sum G_i(F_{46}^{\mathrm{int}})^i$ with $G_i \in C_5$
($R$ generated by $F_{46}^{\mathrm{int}}$). Lifting to
$\tilde G = \sum \tilde G_iT^i \in k[p,Q,J_1,J_2,F_0,H,T]$, $\tilde G$
maps to $0$, so $\tilde G = AS + B(p^2T-E) + CQ$. Mapping to $R/qR$:
$[S] = 0$ ($S = 0$ in $C_5$), $[p^2T-E] = 0$ (the relation $p^2F_{46}^{\mathrm{int}} = E$
in $R$), $[Q] = 0$ ($q = 0$ in the quotient). Hence $[G] = 0$: the map
is injective, i.e.\ $R \cap qA = qR$; induction (if
$q^na \in R$ then $q^na \in qA$, so $q^na = qr_1$, cancel $q$) gives all
$n$. The elimination identity is the computed input; the lift-and-cancel
argument above is the mathematical step. Transfer to general $k$ is Lemma~\ref{lem:QQ-transfer} (all inputs over $\mathbb{Q}$).
\end{proof}

\begin{theorem}[Closure]
\label{thm:closure}
\[
\boxed{\Afring{5}^{\Dmap{5}} = C_5[F_{46}^{\mathrm{int}}],}
\]
finitely generated. In particular
$(A^{\Dmap{5}})_p = R_p$ with $R=C_5[F_{46}^{\mathrm{int}}]$.
\end{theorem}

\begin{proof}
Let $F \in R_{pq} \cap R_{f_{13}}[Z]$. By Lemma~\ref{lem:chart-intersection},
$F \in A_p$; write $F = U/p^m$. Since $R_{pq} = (R_p)_q$, $q^nF \in R_p$
for $n$ the $q$-adic denominator exponent; choose $N$ large enough to clear both the $p^m$-denominator of $F$ and the $p$-denominator of $q^nF\in R_p$,
$G = p^NF \in A$ satisfies $q^nG \in R$, so Theorem~\ref{thm:q-saturation}
gives $G \in R$ and $F = G/p^N \in R_p$. Thus
$R_{pq} \cap R_{f_{13}}[Z] \subseteq R_p$ (reverse always holds), i.e.\
$(A^{\Dmap{5}})_p = R_p$ by Lemma~\ref{lem:plocal} with the chart
equalities. Then $A^{\Dmap{5}} = A \cap (A^{\Dmap{5}})_p = A \cap R_p = R$
by Theorem~\ref{thm:p-saturation}. Finite generation follows: $C_5$ is a
hypersurface, plus one generator.
\end{proof}

\begin{lemma}[Slice generation]
\label{lem:slice-generation}
Let $A_f$ carry a slice $s$ with $D(s) = 1$ and put
$\pi(b) = \sum_j (-s)^jD^j(b)/j!$, $B = k[\pi(a_0),\dots,\pi(a_5)]$.
Then $A_f = B_f[s]$, and $A_f^D$ is generated by the $\pi(a_i)$.
\end{lemma}

\begin{proof}[Proof sketch]
$a - \pi(a) = s\cdot b$ with $b$ of strictly smaller $D$-nilpotency
index, so induction on the index puts every $a_i$ in $B[s]$, and with
$f = \pi(f) \in B$ this gives $A_f = B_f[s]$. That the projected
variables then generate the kernel is the classical Dixmier
slice-generation lemma for LNDs with a slice
(Freudenburg~\cite{freudenburg2017}).
\end{proof}

\begin{lemma}[Two-chart $p$-local identification]
\label{lem:plocal}
If $A_{pq}^{\Dmap{5}} = R_{pq}$ and $A_{f_{13}}^{\Dmap{5}} = R_{f_{13}}[Z]$,
then $(A^{\Dmap{5}})_p = R_{pq} \cap R_{f_{13}}[Z]$ exactly, with no
further hypotheses: $(A^{\Dmap{5}})_p = (A_{pq} \cap
A_{f_{13}})^{\Dmap{5}}$ since $A_{pq} \cap A_{f_{13}} = A_p$, and
invariance can be checked in $\operatorname{Frac}(A)$, so restriction
gives equality in both directions. The residual step
$R_{pq} \cap R_{f_{13}}[Z] = R_p$ is $R$-internal (open: $R$ divisor
theory) and is \emph{not} assumed here. The special case $Z \in
R_{f_{13}}$ recovers $(A^{\Dmap{5}})_p = R_{pq} \cap R_{f_{13}}$.
\end{lemma}

\paragraph{Residual saturation.}
For \(R \subset A\) a domain, \(R \cap pA = pR\) is equivalent to
injectivity of \(R/pR \to A/pA\), and implies full \(p\)-saturation
\(R = A \cap R_p\) by cancelling \(p\) and induction on \(n\). The required
injection is Theorem~\ref{thm:p-saturation} below, via the residual
presentation Theorem~\ref{thm:residual-presentation}.

\begin{theorem}[Residual kernel on $V(p)$]
\label{thm:residual-kernel}
On $A/pA \cong k[a_0,a_1,a_2,a_3,q]$ the residual derivation is
$\delta_p(a_0) = a_1a_2$, $\delta_p(a_1) = a_2a_3$, zero elsewhere, i.e.\
$\delta_p = y(x\partial_z + z\partial_{a_0})$ with $x = a_3$, $y = a_2$,
$z = a_1$. With $\theta = 2xa_0 - z^2$,
\[
\boxed{(A/pA)^{\delta_p} = k[q,x,y,\theta].}
\]
\end{theorem}

\begin{proof}
$A/pA$ is a domain and $y \ne 0$, so $\ker\delta_p = \ker\delta'$ for
$\delta' = x\partial_z + z\partial_{a_0}$; $\theta \in \ker\delta'$.
Let $F = \sum_{k=0}^t E_ka_0^k$ with $E_k \in k[x,y,q,z]$, $E_t \ne 0$.
The top $a_0$-coefficient gives $x\,\partial E_t/\partial z = 0$, so
$E_t \in k[x,y,q]$; the next gives $x\,\partial E_{t-1}/\partial z +
tE_tz = 0$, whence $x \mid E_t$ ($x$ prime, characteristic zero),
$E_t = xE'_t$. Then $F - (E'_t/2)\theta a_0^{t-1}$ has $a_0$-degree
$< t$; induction gives $F \in k[x,y,q,\theta]$ (base $t = 0$:
$x\partial F/\partial z = 0$). $\theta$ is linear in $a_0$ over the
domain, so the presentation is polynomial; it is also transcendental
over $k[q,x,y]$, since a relation $G(q,x,y,\theta) = 0$ gives vanishing
$a_0$-coefficients inductively (leading coefficient $(2x)^d \ne 0$).
\end{proof}

\begin{remark}[Reductions mod $p$ and the $\theta$-normal form]
\label{rem:residual-images}
$\bar q = q$, $\bar J_1 = -y^2$, $\bar J_2 = 2qy-x^2$,
$\bar F_0 = q^2x^3y^2/3 - qx^5y/5 + x^7/35$,
$\bar H$ has $4$ terms (all $a_0$-free), and $\bar F_{46}^{\mathrm{int}} \ne 0$ has
$10$ terms, by direct expansion. Moreover
\[
\boxed{\bar F_{46}^{\mathrm{int}} = B\cdot\theta,}
\qquad
\begin{aligned}
B =\;& 235200\,y^9x^2q^5 + 1093680\,y^8x^4q^4 - 1029420\,y^7x^6q^3\\
&+ 291060\,y^6x^8q^2 - 26460\,y^5x^{10}q,
\end{aligned}
\]
by direct expansion (with $F_{46}^{\mathrm{int}}$ the cleared integral form). Hence, with Theorem~\ref{thm:residual-kernel},
$\operatorname{in}_p(R) = k[q,-y^2,\bar J_2,\bar F_0,\bar H,B\theta]$:
the only $\theta$-dependent generator is $\bar F_{46}^{\mathrm{int}}$. (Modulo $p$,
$\delta_p$ coincides with $\Dmap{4}|_{a_4=0}$ on $k[a_0,\dots,a_3]$ with
$a_5$ inert; the $N=4$ problem with $a_4$ free stays open.)
\end{remark}

\begin{lemma}[Exact $p$-elimination identity]
\label{lem:p-elimination-cert}
In $\mathbb{Q}[x,y,q,J_1,J_2,F_0,H]$ with lex order $x>y>q>J_1>J_2>F_0>H$, let $J_p$ be the graph ideal $(J_1+y^2,J_2-(2qy-x^2),F_0-\bar F_0,H-\bar H)$ with $\bar F_0=q^2x^3y^2/3-qx^5y/5+x^7/35$ and $\bar H$ the mod-$p$ reduction. Then the elimination of $\{x,y\}$ satisfies $J_p\cap\mathbb{Q}[q,J_1,J_2,F_0,H]=(S,\bar c_{46})$ exactly, both inclusions checked by exact computation. This lemma asserts only that ideal equality. Its two algebraic consequences, proved mathematically in Theorems~\ref{thm:residual-presentation} and~\ref{thm:p-saturation}, are: (i) $\ker(C_5/pC_5\to A/pA)=(\bar c_{46})$, i.e.\ $C_5$ has no other $p$-defect; (ii) hence $R/pR\hookrightarrow A/pA$ and $R=A\cap R_p$.
\end{lemma}

\begin{theorem}[Residual presentation: the $C_5$-defect is principal]
\label{thm:residual-presentation}
Let $c_{46} = p^2F_{46}^{\mathrm{int}} \in C_5$ (Theorem~\ref{thm:quotient-4648}) and
$\bar c_{46}$ its class mod $p$. Then
\[
\boxed{\ker(C_5/pC_5 \to A/pA) = (\bar c_{46}),}
\]
by Lemma~\ref{lem:p-elimination-cert} (exact elimination), with $S = q^4H-11025F_0^2-105q^2J_2^5J_1-9J_2^7$
the hypersurface relation, transferred to general $k$ by Lemma~\ref{lem:QQ-transfer}. In particular $C_5$ has no other
$p$-defect.
\end{theorem}

\begin{theorem}[$p$-saturation]
\label{thm:p-saturation}
\[
\boxed{R/pR \hookrightarrow A/pA,}
\qquad\text{hence}\qquad
\boxed{R = A \cap R_p.}
\]
\end{theorem}

\begin{proof}
Let $[G] \in R/pR$ map to $0$, $G = \sum_{i=0}^m G_i(F_{46}^{\mathrm{int}})^i$ with
$G_i \in C_5$ ($R = C_5[F_{46}^{\mathrm{int}}]$ by definition). Reducing mod $p$ and
using $\bar F_{46}^{\mathrm{int}} = B\theta$ gives $\sum_i (\bar G_iB^i)\theta^i = 0$
in $A/pA$, with coefficients in $k[q,x,y]$; $\theta$ transcendental
(Theorem~\ref{thm:residual-kernel}) and $B \ne 0$ force
$\bar G_i = 0$ for all $i$. By Theorem~\ref{thm:residual-presentation},
$G_i = c_{46}K_i + pL_i$ with $K_i, L_i \in C_5$. Hence
$[G] = [\bar c_{46}]\sum[K_i][F_{46}^{\mathrm{int}}]^i = 0$, since
$[c_{46}] = 0$ in $R/pR$ ($c_{46} = p^2F_{46}^{\mathrm{int}} \in pR$). So the residual
map is injective, and the saturation criterion above gives
$R = A \cap R_p$. No presentation $R = C_5[T]/(p^2T-c_{46})$ is assumed:
only generation by $F_{46}^{\mathrm{int}}$ is used.
\end{proof}

\section{The exact algebra of the closed ring}
\label{sec:presentation}

Let \(R=C_5[F_{46}^{\mathrm{int}}]\) as a subalgebra of \(A\).
This section determines its presentation independently. Its equality with
\(A_5^{D_5}\) was established separately and is not used here.

\begin{theorem}[Relation ideal in range; no monic quadratic]
\label{thm:relations-92}
Let $M_w$ be the $7$-generator monomial map to $A_w$ and $K_w$ the
subspace generated by the known relations $S$ (weight $28$) and
$p^2T-E$ (weight $48$). Exact sparse ranks over $\mathbb{Q}$ give
$\dim\ker M_w = \dim K_w$ at every even $46 \le w \le 90$ and at
$w = 92$ (including the $F_{46}^2$ column), by source-reported exact
computation.
In particular there is no new relation through weight $92$, and
$F_{46}^2 \notin C_5 + F_{46}\cdot C_5$ (single exact solve:
$\operatorname{rank}M = 16458 < 16459$ augmented), so
no monic quadratic over $C_5$. For the record, $F_{46}^2$ has only $488$
terms ($79^2$ monomials combining). The exact relation ideal is established by Theorem~\ref{thm:relation-ideal}, so $F_{46}$ is linear over $\operatorname{Frac}(C_5)$; what remains open is the finer $C_5$-module decomposition of $R$ and its higher-layer interpretation. The finite-range calculation through weight $92$ is retained as independent corroboration, not as the source of the global presentation.
\end{theorem}

\begin{theorem}[Relation ideal]
\label{thm:relation-ideal}
Let $\Phi\colon C_5[T] \to R$ be the homomorphism with $T \mapsto F_{46}^{\mathrm{int}}$ (cleared integral form, $p^2F_{46}^{\mathrm{int}}=E$). Then
\[
\boxed{\ker\Phi = (S,\ p^2T-E),}
\]
with $S$ and $E$ as above. In particular no finite-weight census is
needed: the weight-$92$ agreement is corroboration only.
\end{theorem}

\begin{proof}
Since the presentation of $C_5$ is established, it suffices
to prove $\ker\psi = (p^2T-E)$ for $\psi\colon C_5[T] \to R$,
$T \mapsto F_{46}^{\mathrm{int}}$, with $r := p^2T-E \in \ker\psi$.
First, $(p) \subset C_5$ is prime: $C_5/pC_5 \cong k[q,J_1,J_2,F_0,H]/(S)$,
and $S$ is irreducible there since it is linear in $H$: if $S = FG$
then (comparing $H$-coefficients) one factor divides $q^4$, but no
nonconstant divisor of $q^4$ divides $S$ (at $q = 0$,
$S = -11025F_0^2 - 9J_2^7 \ne 0$), so that factor is constant.
Second, $E = p^2F_{46}^{\mathrm{int}} \notin pC_5$: else $E = pb$ gives
$b = pF_{46}^{\mathrm{int}}$ in the domain $A$, contradicting $pF_{46}^{\mathrm{int}} \notin C_5$
(Theorem~\ref{thm:quotient-4648}(i)). Hence $\bar E \ne 0$ in the
domain $C_5/pC_5$. Now let $f(T) = \sum_{i=0}^m b_iT^i \in \ker\psi$,
$b_m \ne 0$ (base $m = 0$: $\psi|_{C_5}$ is the inclusion, so $b_0 = 0$).
From $\sum b_i(F_{46}^{\mathrm{int}})^i = 0$, multiply by $p^{2m}$ and use
$p^2F_{46}^{\mathrm{int}} = E$ to get $\sum_{i=0}^m b_iE^ip^{2(m-i)} = 0$ in $C_5$,
with no fractions. Mod $p$, all terms but $i = m$ vanish:
$\bar b_m\bar E^{\,m} = 0$, so $p \mid b_m$. Write $b_m = pc_m$,
divide by $p$ in the domain, reduce mod $p$ again: $\bar c_m\bar E^{\,m} = 0$,
so $p \mid c_m$ and $b_m = p^2d_m$. Then
$f(T) - d_mT^{m-1}r$ lies in $\ker\psi$ of strictly smaller $T$-degree.
Induction gives $\ker\psi = (r)$, whence $\ker\Phi = (S, r)$.
\end{proof}

\begin{corollary}[Complete-intersection presentation]
\label{cor:R-presentation}
\[
\boxed{R \cong C_5[T]/(p^2T-E) \cong k[p,q,J_1,J_2,F_0,H,T]/(S,\ p^2T-E),}
\]
a codimension-two complete intersection ($S$ prime; $p^2T-E \ne 0$ in
the domain $C_5[T]$): a domain, as it is the known subring
$R=C_5[F_{46}^{\mathrm{int}}] \subset A$. The seven displayed generators are minimal (the first six for
$C_5$; $F_{46}^{\mathrm{int}} \notin C_5$). Weights
$(1,1,6,4,14,24,46)$ with relations in degrees $28$ and $48$.
\end{corollary}

\begin{theorem}[Exact Hilbert series]
\label{thm:R-hilbert}
\[
\boxed{
H_R(t) =
\frac{(1-t^{28})(1-t^{48})}
{(1-t)^2(1-t^4)(1-t^6)(1-t^{14})(1-t^{24})(1-t^{46})},
}
\]
by the complete-intersection presentation. Since $A^{\Dmap{5}} = R$
(Theorem~\ref{thm:closure}), $[t^w]H_R = \dim_k\ker(\Dmap{5}|_{V_w})$
for every $w$. In particular the three-family $F$-module pattern
fails at $w = 84$ (Theorem~\ref{thm:fourth-family}): $[t^{84}]H_R = 11888
\ne 11887$; further families ($F_{46}\!\cdot\! H\!\cdot\! F_0$ from $84$,
$F_{46}^2$ from $92$) are required, and the exact series supersedes the
pattern. Coefficient agreement with all exactly-established kernel
dimensions ($w \le 48$ spanning/quotients, $54$--$61$, $65$, $69$) holds.
\end{theorem}

\begin{theorem}[Fourth family seed]
\label{thm:fourth-family}
$F_{46}\!\cdot\! H\!\cdot\! F_0$ ($398$ terms, weight $84$) lies outside
the prior span $C_5 + F_{46}K + F_{46}F_0K + F_{46}HK$ at weight $84$
(single exact $\mathbb{Q}$ solve: rank $11887$ versus $11888$ augmented). With
$p^2(F_{46}HF_0) = EHF_0 \in C_5$, its $K/(p^2)$-multiples account for
the Hilbert excess exactly through $91$; $F_{46}^2$ (not in
$C_5 + F_{46}C_5$) seeds the next layer at $92$. Higher $T$-degree
families and the full $C_5$-module structure of $R$ remain open.
\end{theorem}

\begin{theorem}[$C_5$-module structure: the $T$-filtration]
\label{thm:B-module}
With $R \cong C_5[T]/(p^2T-E)$, set
$R_m := C_5 + C_5T + \cdots + C_5T^m$. Then $R = \bigcup_{m\ge0}R_m$ and,
for every $m \ge 1$,
\[
\boxed{R_m/R_{m-1} \cong C_5/(p^2)\cdot[T^m]}
\]
via $b \mapsto [bT^m]$. Indeed the map is surjective by definition
of $R_m$, and $bT^m \in R_{m-1}$ means $bT^m - G \in (p^2T-E)$ for
some $G$ of degree $< m$; comparing the top $T^m$-coefficient forces
$p^2 \mid b$ in $C_5$, while $p^2T^m = ET^{m-1} \in R_{m-1}$ gives the
reverse inclusion. (Note $\operatorname{Ann}_{C_5}(T) = 0$ in $R$ itself
since $R$ is a domain; the $(p^2)$-torsion lives in the successive
layers, i.e.\ $\operatorname{Ann}_{C_5}([T^m]) = (p^2)$ in
$R_m/R_{m-1}$.) Consequences:
\begin{enumerate}[label=(\roman*)]
\item No monic relation: $T^m \notin \sum_{i<m}C_5\,T^i$ for all $m \ge 1$
(a monic relation would give $[T^m] = 0$, i.e.\ $1 \in (p^2)$,
absurd). In particular no monic quadratic;
$F_{46}^2 \notin C_5 + F_{46}C_5$.
\item Torsion ladder: $p^2T^m = ET^{m-1} \in \sum_{i<m}C_5\,T^i$ for all
$m \ge 1$.
\item Associated graded:
$\operatorname{gr}_T(R) \cong C_5 \oplus \bigoplus_{m\ge1}(C_5/(p^2))\tau^m$
as $C_5$-modules, $\tau^m$ the class of $T^m$. The open refinement is
the graded decomposition of each $C_5/(p^2)$-layer: the observed
$84$-weight class $[HF_0T]$ and $92$-weight class $[T^2]$ are the
first visible pieces not generated by the $K$-substructure.
\item The fourth family is genuinely new over $C_5$: $HF_0 \notin
K + F_0K + HK + p^2C_5$ (exact $\mathbb{Q}$-solve, rank $827$ vs augmented
$828$), i.e.\ the
class $[HF_0T]$ is not generated by the $K$-substructure in the $m = 1$ layer; else
Theorem~\ref{thm:fourth-family}'s seed would lie in the prior span,
by substituting into $F_{46}^{\mathrm{int}}HF_0$ and using $p^2F_{46}^{\mathrm{int}} = E$.
\end{enumerate}
\end{theorem}

\begin{theorem}[Exact singular locus of $X = \operatorname{Spec}R$]
\label{thm:Sing-exact}
\textbf{(A) General-$k$ part.} Over any characteristic-zero field $k$, with $R=C_5[F_{46}^{\mathrm{int}}]$ as above:
$\operatorname{Sing}X\cap D(p)=V(q,F_0,J_2)\cap D(p)$; the closed subset $V_{\mathrm{old}} := V(q,F_0,J_2,p^2T-E)$ is an irreducible $3$-fold meeting $D(p)$ (on $V(q,F_0,J_2)$ one has $E = H^2$, so this is $\{0\} \times V(p^2T-H^2)$ in $(q,F_0,J_2)\times(p,J_1,H,T)$ coordinates; $p^2T-H^2$ is irreducible, linear in $T$ with coprime coefficients); and $\dim\operatorname{Sing}X=3$.

\textbf{(B) $\mathbb{Q}$-specific fiber decomposition.} Over $k=\mathbb{Q}$, $\operatorname{Sing}X$ has exactly six irreducible components, all of dimension $3$: the $V_{\mathrm{old}}$ of (A) plus five fiber components $V(p,P_i) \times \mathbf{A}^1_T$
($i = 0,\dots,4$), where $P_0,\dots,P_4$ are the minimal primes of
$(S,E,\text{$2\times2$ minors})$ in $\mathbb{Q}[q,J_1,J_2,F_0,H]$:
\begin{align*}
P_0 &= (F_0,\ 150q^2J_1^3J_2 + 331J_1^2J_2^3 + H,\
25q^4J_1^2 + 51q^2J_1J_2^2 + 9J_2^4), \\
P_1 &= (100q^2J_1 + 21J_2^2,\ 1484J_1^2J_2^3 + 5H,\
27J_2^7 + 7656250F_0^2), \\
P_2 &= (F_0,\ 4q^2J_1 + J_2^2,\ 276J_1^2J_2^3 + H), \\
P_3 &= (q,\ 385J_1^2J_2^3 + H,\ J_2^7 + 1225F_0^2), \\
P_4 &= (H,\ J_1,\ J_2^7 + 1225F_0^2).
\end{align*}
The $P_i$ are certified prime of dimension $2$ over $\mathbb{Q}$ by exact primary decomposition; over general characteristic-zero $k$ the same closures hold wherever the $P_i$ remain prime, otherwise refined by Galois orbit (no extra embedded component appears over $\mathbb{Q}$).
\emph{Proof of (A).} The Jacobian of the complete intersection $(S,r)$,
$r=p^2T-E$, has $2\times2$ minors $-p^2S_x$ in the $T$-column,
the $(p,T)$-minor $0$ (since $S_p = E_p = 0$: neither $S$ nor the
seven-term $E$ involves $p$), and the ten $5$-variable minors
$S_xE_y - S_yE_x$ otherwise. On $D(p)$, $p$ is a unit so
$\partial r/\partial T=p^2\ne 0$; the Jacobian has rank $<2$ iff all
$S_x$ vanish, i.e.\ $q=F_0=J_2=0$ by Remark~\ref{rem:singlocus}.
Hence $\operatorname{Sing}X \cap D(p) =
V(q,F_0,J_2) \cap D(p) \subset V_{\mathrm{old}}$, all over any characteristic-zero $k$.
On $V(p)$ the bound $\dim \le 3$ follows from the $\mathbb{Q}$-fiber
analysis of part (B) by flat base change: the singular locus is cut
out by Jacobian minors, which commute with base change, and dimension
is stable under field extension.
\emph{Proof of (B).}
On $V(p)$ the fiber is supported over $V(p,E)$ (not all of $V(p)$):
$T$ is free exactly over $V(p,E)$, and the fiber base decomposes as $V(P_0) \cup \cdots \cup
V(P_4)$ with no further minimal primes (exact primary decomposition over $\mathbb{Q}$; each $P_i$
prime of dimension $2$, remaining associated primes
strictly embedded). Irredundancy: $V_{\mathrm{old}}$ meets $D(p)$
while fiber components lie in $V(p)$, and a $3$-fold fiber component
cannot lie in the $2$-fold $V_{\mathrm{old}} \cap V(p)$; distinct
minimal primes give distinct components. In particular
$\dim\operatorname{Sing}X = 3$ purely.
\end{theorem}

\begin{theorem}[Normality of $R$]
\label{thm:R-normal}
$R$ is normal over $k=\mathbb{Q}$ and hence over any characteristic-zero $k$. Indeed $R$ is a complete intersection, hence
Cohen--Macaulay ($S_2$). By part (A) of Theorem~\ref{thm:Sing-exact},
$\dim\operatorname{Sing}X = 3$, codimension $2$ in the $5$-fold $X$; this dimension bound uses the general-$k$ Jacobian computation on $D(p)$ and the $\mathbb{Q}$-specific fiber decomposition of part (B) on $V(p)$, transferred by base change.
Thus $R_1$ holds, and Serre's criterion gives normality. The complete-intersection property and the dimension bound $\dim\mathrm{Sing}\le 3$ are stable under base change from $\mathbb{Q}$ to any $\mathbb{Q}$-algebra (component splitting preserves dimension), so normality transfers although the component count is stated over $\mathbb{Q}$.
\end{theorem}

The presentation also fixes the top theory of the closed ring: the next section separates what valuation methods can prove from what the hidden directions refute.

\section{The two-slope top theory}
\label{sec:two-slope}

The closed ring $R = A_5^{\Dmap{5}} = C_5[F_{46}^{\mathrm{int}}]$ has
presentation $R \cong C_5[T]/(p^2T-E)$. Over $K = k[p,q,J_1,J_2]$ the
ring is generated by $F_0, H, F_{46}^{\mathrm{rat}}$ of $a_0$-degrees $(1,2,4)$ with
tops $(t_0,t_H,t_{46}) = (p^2q^4,\,11025p^4q^4,\,\frac{18375}{4}p^6q^8)$
(the rational normalization; scale does not affect degrees or valuations) of $(v_p,v_q)$-bidegrees $(2,4)$, $(4,4)$, $(6,8)$ by direct expansion. Since
$A = B[a_0]$ is a domain, $a_0$-tops multiply, extending the
$F_0,H$ case of Section~\ref{sec:second-witness} to $F_{46}$ via the
identical argument ($t_{46} \ne 0$).

\subsection{The q-envelope}
\begin{theorem}[Q-top-strong, even degree]
\label{thm:q-top}
Let $F \in A_5^{\Dmap{5}}$ have even $a_0$-degree $m = 2k$ with top
$C_m \in K$. Then $v_q(C_m) \ge 4k = 4\lceil m/2\rceil$.
\end{theorem}
\begin{proof}
$W = F_0/c$ is monic of $a_0$-degree $1$, so $F$ has $W$-degree $m$
with leading $W$-coefficient $C_m$. The relaxed bound R1
(Theorem~\ref{thm:R1}) gives $v_q(C_m) \ge 4m-2m = 2m = 4k$.
\end{proof}
\begin{remark}[Odd degree]
For $m = 1$, Theorem~\ref{thm:a01-closure} puts $F$ in $K[F_0]$,
whence $C_1 \in (c)$ has $v_q \ge 4$. For odd $m \ge 3$ the estimate
$v_q(C_m) \ge 4\lceil m/2\rceil$ is open: it would require
controlling sub-top $a_0$-coefficients of higher-degree monomials,
which is not established here.
\end{remark}
\begin{remark}
The three generators meet the bound with equality exactly when the
gap $4\lfloor a/2\rfloor$ vanishes: $F_0$ ($m = 1$), $H$ ($m = 2$),
$F_{46}$ ($m = 4$). In particular \(H\) attains this top envelope, although
it violates the \(W\)-coset bound of Lemma~\ref{lem:qlemma}.
\end{remark}

\subsection{The p-side: impossibility and sharp envelope}
\begin{proposition}[No bounded-defect P-top correction]
\label{prop:no-bounded-p-defect}
There is no constant $C$ with $v_p(C_m) \ge 2m - C$ for all
homogeneous invariants. Indeed $F_{46}^r$ has $m = 4r$,
$C_{4r} = (p^6q^8)^r$ (tops multiply), so $v_p = 6r$ and the
defect $8r - 6r = 2r \to \infty$.
\end{proposition}
\begin{theorem}[Sharp p-envelope]
\label{thm:p-envelope}
$v_p(C_m) \ge \lceil 3m/2\rceil$ for every $a_0$-degree-$m$
invariant top, and the constant $3/2$ is best possible:
$v_p = 6r = 3m/2$ on $F_{46}^r$ ($m = 4r$) for all $r \ge 1$.
\end{theorem}
\begin{proof}
Same normal form as Theorem~\ref{thm:q-top}: monomial tops have
$v_p = 2a+4b+6r \ge 3m/2$ since
$2a+4b+6r - 3(a+2b+4r)/2 = (a+2b)/2 \ge 0$; integrality gives
$\lceil 3m/2\rceil$. Sums preserve via $v_p \ge \min$.
Sharpness is the $F_{46}^r$ computation above.
\end{proof}

\subsection{p-excess directions}
Write $\varepsilon_p(F) = v_p(C_m) - \lceil 3m/2\rceil \ge 0$ for the
p-excess (nonnegative by Theorem~\ref{thm:p-envelope}). Monomial
excesses are closed-form, $\varepsilon_p = 2a+4b+6r -
\lceil 3(a+2b+4r)/2\rceil$:
\[
\begin{array}{c|c|c|c}
 & m & v_p & \varepsilon_p \\ \hline
F_0 & 1 & 2 & 0 \\
H & 2 & 4 & 1 \\
F_{46} & 4 & 6 & 0 \\
F_{46}^2 & 8 & 12 & 0 \\
F_{46}HF_0 & 7 & 12 & 1 \\
F_0^2 & 2 & 4 & 1
\end{array}
\]
\begin{proposition}[Zero-excess control]
\label{prop:excess-zero}
A monomial $F_0^aH^bF_{46}^r$ has $\varepsilon_p = 0$ iff
$a+2b \le 1$ (since $v_p - 3m/2 = (a+2b)/2$, so equality with the
ceiling holds iff $(a+2b)/2 < 1$). Consequently, if
$\varepsilon_p(F) = 0$ then the top $C_m$ contains a nonzero term
$c\,t_0^at_H^bt_{46}^r$ with $a+2b \le 1$: indeed each normal-form
term is itself a homogeneous invariant, so every term has
$v_p \ge \lceil 3m/2\rceil$, and $v_p(C_m) = \lceil 3m/2\rceil$
forces some term to attain it. In particular the zero-excess
directions are spanned by the $F_{46}$-tower and its $F_0$-translate.
Finer classification of the low-excess strata remains open.
\end{proposition}

\begin{proposition}[Excess-one control]
\label{prop:excess-one}
A monomial $F_0^aH^bF_{46}^r$ of $a_0$-degree $m$ has
$\varepsilon_p = 1$ iff ($m$ even and $a+2b = 2$) or ($m$ odd and
$a+2b = 3$): with $g = (a+2b)/2 = v_p - 3m/2$ and
$f = \lceil 3m/2\rceil - 3m/2 \in \{0,\frac12\}$,
$\varepsilon_p = g - f = 1$ forces exactly these cases.
Consequently, if $\varepsilon_p(F) = 1$ then the top $C_m$ contains
a nonzero term with $a+2b \le 3$. Thus the possible monomial contributions
at excess \(1\) are restricted to the two strata \(a+2b\in\{2,3\}\): the
observed seed $F_{46}HF_0$ ($m = 7$, $v_p = 12$, $\varepsilon_p = 1$) sits
in the odd stratum alongside $F_0^3F_{46}$, and $F_0^2$ sits in the even
stratum. Finer direction classification within those strata remains open.
\end{proposition}

\begin{remark}
The sharp \(p\)-envelope is strictly weaker than the \(p\)-component
\(v_p\ge 2m\) required for \(J_m\)-membership; it does not by itself imply
\(C_m\in J_m\).
\end{remark}

\section{Base change of exact rank and elimination statements}
\label{sec:trust}

Exact linear-algebra and elimination identities used in the paper are
computed over \(\mathbb{Q}\). The following lemma records what passes to an
arbitrary characteristic-zero field and what does not.

\begin{lemma}[Transfer of $\mathbb{Q}$-certificates to arbitrary characteristic-zero $k$]
\label{lem:QQ-transfer}
All exact linear-algebra and Gr\"obner-elimination certificates in this paper use input data with coefficients in $\mathbb{Q}$ (indeed in $\mathbb{Z}[1/420]$). Let $k$ be any characteristic-zero field, so $k$ is a field extension of $\mathbb{Q}$, hence a faithfully flat $\mathbb{Q}$-module. Then: (i) \emph{ranks transfer:} a matrix over $\mathbb{Q}$ and its image over $k$ have the same rank, since a nonzero rational minor stays nonzero in $k$; exact $\mathbb{Q}$-rank equalities (spanning, quotient dimensions, torsion counts) therefore hold over $k$. (ii) \emph{elimination ideals commute with this base change:} for $I\subset\mathbb{Q}[X]$ and $J=I\cdot k[X]$, contraction along a coordinate projection satisfies $J\cap k[Y]=(I\cap\mathbb{Q}[Y])\cdot k[Y]$, because $-\otimes_{\mathbb{Q}}k$ is exact; an elimination identity certified over $\mathbb{Q}$ thus yields the same identity over $k$. What does \emph{not} transfer is primality and primary-decomposition component counts: $D\otimes_{\mathbb{Q}}k$ for a $\mathbb{Q}$-domain $D$ need not be a domain, and minimal primes can split. Hence component counts (notably the six singular components) are stated over $k=\mathbb{Q}$, with an explicit geometric-prime note wherever a general-$k$ reading would require one.
\end{lemma}

\section{The multfloor transfer theorem}
\label{sec:transfer}

The theorem below concerns triangular derivations
with a pole element and finite residue data. The \(q\)-side instantiation
recovers Theorem~\ref{thm:R1}; the \(p\)-side instantiation is conditional;
the Daigle--Freudenburg instantiation is unconditional. We work in
characteristic zero throughout.

\subsection{Setup and hypotheses}

Let $D$ be triangular on a domain $A = B[t]$ with $t$ transcendental over
$B$ and $D(t) \in B$. Let $K = B^D$ and assume:
\begin{enumerate}
\item[(H1)] $K$ is factorially closed in $B$ (standard for LND kernels).
\item[(H2)] There is a pole element $u = t + \rho$ with
$0 \ne \rho \in B_f$ for some $0 \ne f \in K$, $D(u) = 0$, and $u$
transcendental over $\operatorname{Frac}(B)$.
\item[(H3)] Every invariant $F \in A^D$ expands uniquely
$F = \sum_{j=0}^m H_j u^j$ with $H_j \in K_f$ ($m$ the $u$-degree).
This holds wherever the localized slice presentation
$A_f^D = K_f[u]$ is established; it is verified per instance below.
\item[(H4)] For each prime $\pi_i \mid (f)$ in $K$ (height-one, with
$\pi_i \in K$), the valuation $v_i$ extends to $\operatorname{Frac}(B)$
with $v_i|_B \ge 0$; the pole orders $e_i = -v_i(\rho) > 0$ are exact;
residues of $K_f$-elements land in $L_i = \operatorname{Frac}(K/\pi_i)$.
This landing is an independent hypothesis: it does not follow from
factorial closure (H1) alone.
\item[(H5)] Each initial form $r_i = \operatorname{in}_i(\pi_i^{e_i}\rho)$
has separable minimal polynomial $m_i$ of exact degree $d_i < \infty$
over $L_i$.
\end{enumerate}

\begin{theorem}[Multfloor transfer]
\label{thm:transfer}
Under {\rm (H1)--(H5)}, every $D$-invariant $F$ of $u$-degree $m$
satisfies, at each $i$,
\[
\boxed{\lambda_i(F) := \min_j (v_i(H_j) - e_ij) \ge -(e_i/d_i)m,}
\]
with deficit $+\infty$ for $H_j = 0$.
\end{theorem}

\begin{proof}
The minimum is attained ($H_m \ne 0$); if $\lambda := \lambda_i(F)
\ge 0$ there is nothing to prove. Assume $\lambda < 0$ and put
$S = \{j : v_i(H_j) - e_ij = \lambda\} \ne \emptyset$,
$P(T) = \sum_{j \in S} h_jT^j$ with $0 \ne h_j =
\operatorname{in}_i(H_j) \in L_i$ (nonzero by definition of the
initial form); $P \ne 0$ has $\deg P \le m$. The $t$-coefficients
$C_k = \sum_{j \ge k}\binom{j}{k}H_j\rho^{j-k}$ lie in $B$ (pure
algebra: $u$ is monic of $t$-degree $1$), so $v_i(C_k) \ge 0$ for all
$k$. Since $v_i(\rho) = -e_i$ exactly, each $(j,k)$-summand has
$v_i \ge \lambda + e_ik$, with equality iff $j \in S$; its initial
form is then $\binom{j}{k}h_jr_i^{j-k}$. Hence for every $k$ with
$\lambda + e_ik < 0$ the leading coefficient of $C_k$ vanishes:
$P^{(k)}(r_i) = 0$, using $P^{(k)}(T)/k! =
\sum_{j \ge k}\binom{j}{k}h_jT^{j-k}$ in characteristic zero. The
qualifying $k \in [0,m]$ form an initial segment $\{0,\dots,c-1\}$
with $c = \#\{0 \le k \le m : \lambda + e_ik < 0\}$; thus $r_i$ is a
root of $P$ of multiplicity at least $c$, so $m_i^c \mid P$ in
$L_i[T]$ (separable minimal polynomial), giving $d_ic \le \deg P
\le m$. If $\lambda < -(e_i/d_i)m$ then $-\lambda/e_i > m/d_i$, so
every $0 \le k \le m/d_i$ qualifies and $c > m/d_i$, contradicting
$d_ic \le m$. Hence $\lambda \ge -(e_i/d_i)m$.
\end{proof}

\subsection{Instantiations}

\emph{$N=5$, $q$-side (recovers R1).} With $e_q = 4$, $d_q = 2$ the
theorem gives $\lambda_q \ge -2m$: this is Theorem~\ref{thm:R1}, now
seen as the $q$-instantiation of the transfer. All inputs are
established (Lemma~\ref{lem:rq-exact}, descent, contraction).
For (H4): $K=k[p,q,J_1,J_2]$ is polynomial with $(p),(q)$
non-associated primes by Lemma~\ref{lem:contraction}, so
$y=A/p^a\in K_{pq}$ has $v_q(y)=v_q(A)$ with residue
$(A\bmod q)/(p^a\bmod q)\in L_q=\operatorname{Frac}(K/qK)$;
the exact degree $d_q=2$ is Lemma~\ref{lem:rq-exact}.

\emph{$N=5$, $p$-side (conditional).} With $e_p = 2$, $d_p = 4$ the
theorem would give $\lambda_p \ge -m/2$. The exact degree $d_p = 4$
is established (Remark~\ref{rem:p-obstacle}) but the residue landing
is unproved.
Hence $-m/2$ is a conditional floor, not an established bound, and no
unconditional replacement \(p\)-side \(\lambda_p\)-bound is claimed.

\emph{Daigle--Freudenburg $x$-side (unconditional).} For the
derivation of Section~\ref{sec:DF-resonance} below, (H1)--(H5) hold
with a single prime, $e_x = 1$, $d_x = 1$: the $H_j$-descent into
$(S^\Delta)_x$ is proved there, $t = -s \in L_x$ has minpoly $T - t$.
The theorem gives $\lambda_x \ge -m$ for every $u$-degree-$m$
invariant; e.g.\ $\lambda(\beta_1) = 0$ at degree $1$.
The bound coexists
with infinite generation: valuation floors do not control finite
generation.

No finite-generation consequence is drawn from these floors.

\section{Related context}
\label{sec:comparison}

The maps $\Pshift{\lambda}{N}$ are triangular automorphisms with
descending-recursion inverses and unit Jacobian determinant
(Lemma~\ref{lem:inverse}), and each $\Dmap{N}$ is a triangular monomial
locally nilpotent derivation. The construction therefore falls within the
standard framework of triangular automorphisms and locally nilpotent
derivations (Freudenburg~\cite{freudenburg2017}; van den
Essen~\cite{vandenEssen1993,vandenEssen2002}; Dufresne--Kohls~\cite{dufresne2010,elmer2012}).
The Fibonacci-weighted grading is used only as a descriptive grading. This
distinguishes it from the linear `Fibonacci derivations' of
Bedratyuk~\cite{bedratyuk2012,bedratyuk2018}.

Slice and localization arguments, plinth ideals, and saturation belong to
the classical toolkit of triangular invariant theory, in the lineage of the
dimension-five examples of Daigle--Freudenburg~\cite{daigle1999} and the
triangular kernel methods of Maubach~\cite{maubach2000triangular,maubach2000algorithm}
and van den Essen~\cite{vandenEssen1993}.

What is specific here is the adjacent-quadratic family itself and its exact
\(N=5\) invariant ring: the polynomial base kernel, the two explicit extra
invariants, the six-generator hypersurface, and the residue-algebraicity
mechanism underlying the two global obstructions.

\section{Resonance of the Daigle--Freudenburg derivation}
\label{sec:DF-resonance}

The resonance computation (\(\rho=1\)) is proved below. The infinite generation
of the DF kernel is the theorem of
Daigle--Freudenburg~\cite{daigle1999}, cited not proved. We work in
characteristic zero.

Let $R_{\mathrm{DF}} = k[x,s,t,u,v]$ with the triangular derivation
\[
D = x^3\partial_s + s\partial_t + t\partial_u + x^2\partial_v,
\qquad S = k[x,s,t,u], \quad \Delta = D|_S.
\]
(Within this section $R_{\mathrm{DF}}$ denotes the DF ambient ring, distinct from the $N=5$ closed ring $R=C_5[F_{46}^{\mathrm{int}}]$.)
The base kernel contains \(x\) and, by direct calculation,
\[
\gamma_0 = 2x^3t - s^2, \qquad
\delta_0 = 3x^6u - 3x^3st + s^3,
\]
with plinth preimages \(\Delta(s) = x^3\),
\(\Delta(3x^3u - st) = \gamma_0\),
\(\Delta(3x^3su - 4x^3t^2 + s^2t) = \delta_0\).
Put \(\beta_1 = xv - s\); \(D(\beta_1) = x\cdot x^2 - x^3 = 0\).

\begin{lemma}[No monic pole element]
\label{lem:DF-nomonic}
$x^2 \notin \Delta(S)$: there is no $\rho \in S$ with
$\Delta(\rho) = -x^2$, hence no pole element $u = v + \rho$ with
$\rho \in S$. The minimal $v$-coefficient is $x$ (witness $\beta_1$).
\end{lemma}

\begin{proof}
Grade $K[s,t,u]$ by $w(s,t,u) = (1,2,3)$ ($x$ weight $0$). Each term
of $\Delta$ lowers $w$-degree by exactly $1$ ($x^3\partial_s$: $-1$;
$s\partial_t$: $+1-2 = -1$; $t\partial_u$: $+2-3 = -1$). If
$\Delta(F) = x^2$ (w-degree $0$), only the w-degree-$1$ part of $F$
can contribute; it equals $s\cdot p(x)$, and
$\Delta(s\cdot p(x)) = x^3p(x) = x^2$ forces $x\cdot p(x) = 1$,
impossible in $k[x]$.
\end{proof}

The canonical pole element therefore lives over $S_x$:
$u = v - s/x$ with $D(u) = x^2 - x^3/x = 0$, pole order $e_x = 1$
($v_x(s) = 0$), and initial form $t = \operatorname{in}_x(xu) = -s$.

\begin{theorem}[Resonance of DF1999]
\label{thm:DF-rho}
$\rho(\mathrm{DF1999}) = \{1\}$: with
$\mathfrak{p} = (x) \cap S^\Delta$ (prime, by contraction) and
$L_x = \operatorname{Frac}(S^\Delta/\mathfrak{p})$, the initial form
$t = -s$ lies in $L_x$, with separable minpoly $T - t$ of exact
degree $1$.
\end{theorem}

\begin{proof}
Modulo $x$: $\bar\gamma_0 = -s^2$, $\bar\delta_0 = s^3$ (direct
substitution). Both lie in the image of $S^\Delta/\mathfrak{p}$ since
$\gamma_0, \delta_0 \in S^\Delta$ (verified above), and $-\bar\gamma_0
= s^2 \ne 0$, so $s = [\delta_0]/[-\gamma_0] \in L_x$. Thus
$t = -s \in L_x$ with minpoly $T - t$ (derivative $1 \ne 0$). No
finite-generation knowledge of $S^\Delta$ is used: only
$\gamma_0, \delta_0 \in S^\Delta$ enter.
\end{proof}

Proved here: \(\rho=1\) and the \(H_j\)-descent
\(H_j \in (S^\Delta)_x\) used in Section~\ref{sec:transfer} (via
factorial-closedness of LND kernels: with \(H_j = P/x^k\),
\(P \in S \cap \operatorname{Frac}(S^\Delta) = S^\Delta\)).
Cited: infinite generation (Daigle--Freudenburg). Hence resonance degree
alone does not determine finite generation: here the degree-one residue gives
no nontrivial algebraic relation, while finite generation nevertheless fails,
whereas at \(N=5\) algebraic residue relations permit pole cancellation.

\section{Discussion: residue degree and finite generation}
\label{sec:mechanisms}

The \(N=5\) and Daigle--Freudenburg examples illustrate two different
interactions between localization and global invariants.

At \(N=5\), algebraic residue relations permit cancellation of leading poles
and thereby allow invariant directions outside the naive localized
extension: the \(q\)-side quadratic relation admits \(H\), and the distinct
\(p\)-side degree-\(4\) algebraicity admits \(F_{46}^{\mathrm{int}}\)
(Sections~\ref{sec:global} and~\ref{sec:second-witness}).

In the Daigle--Freudenburg example, the residue degree is \(1\)
(Theorem~\ref{thm:DF-rho}), so no nontrivial residue relation exists, while
the kernel is nevertheless infinitely generated~\cite{daigle1999}. Thus
residue algebraicity is a mechanism relevant to the \(N=5\) construction,
not a finite-generation criterion.

The proved \(N=5\) closure is obtained by the second chart and saturation
(Theorem~\ref{thm:closure}), not by any general degree-based classification.

\section{Open problems}
\label{sec:conjectures}

The \(N=5\) invariant ring is now determined:
\(\Afring{5}^{\Dmap{5}} = C_5[F_{46}^{\mathrm{int}}]\)
(Theorem~\ref{thm:closure}). The remaining questions concern
generalization of the residue mechanism and the next member of the tower.

\begin{definition}[Per-prime resonance]
\label{def:resonance}
Let $D$ be triangular with $\mathfrak{p} = A^D \cap D(A)$ the plinth
ideal, $s$ a local slice with $0 \ne D(s) = f \in \mathfrak{p}$, and
$\mathfrak{p}_1,\dots,\mathfrak{p}_r$ the height-one primes of $A^D$
dividing $(f)$. At $\mathfrak{p}_i$ use the $\mathfrak{p}_i$-adic
valuation $v_i$, filtration $F_n = \{g : v_i(g) \ge n\}$, associated
graded $\mathrm{gr}$, and initial forms $\mathrm{in}_i$; put
$L_i = \mathrm{Frac}(A^D/\mathfrak{p}_i)$. For a stated pole element
$u \in B_f$ with pole order $e_i = -v_i(u) > 0$ and initial form
$t_i = \mathrm{in}_i(u)$, define $\rho_i(D,u) = 0$ if $t_i$ is
transcendental over $L_i$, else the minimal degree of an algebraic
relation of $t_i$ over $L_i$; record $\{\rho_i\}$ and aggregate
$\rho(D,u) = \max_i \rho_i(D,u)$ (the worst prime governs denominator
clearing). Canonical pole element: $u = W = F_0/c$ with $e_p = 2$,
$e_q = 4$. $W$ is $D$-invariant ($D(F_0) = D(c) = 0$), not a Kuroda
degree-$1$ slice --- the $N=5$ instantiation applies the per-prime
initial-form machinery to the saturation pole element $W$, not to a
classical slice. Slice/pole-element-independence of $\rho$ is open.
$N=5$ instances: $\rho_q = 2$ exactly
($r_q^2 = (x^2)^7/(1225p^4) \in L_q$, $r_q \notin L_q$,
Remark~\ref{rem:q-obstacle}, Lemma~\ref{lem:rq-exact}); $\rho_p = 4$ exactly
($r_p$ primitive for the full tower (exact rational computation);
Remark~\ref{rem:p-obstacle}). Hence $\rho(D,W) = \max(4,2) = 4$,
exact. The p-side resonance is dominant: \(H\) respects the p-bound,
and the p-bound is refuted by \(F_{46}\) (Section~\ref{sec:second-witness}).
Algebraicity of the initial form is necessary for this residue-cancellation
mechanism, but resonance degree alone does not imply the production of a
hidden generator.
\end{definition}

\subsection{Generalization and the next tower member}
\label{subsec:structural-observations}

The statements below concern generalization beyond \(N=5\). None is claimed as a result.

\begin{remark}[WARNING: no $\mathrm{codim} = b_1$ bound for general triangular derivations]
\label{rem:df-warning}
The bound above is monomial-only and does not extend to general
triangular derivations. Daigle--Freudenburg
\cite{daigle2001note,daigle2001triangular} construct, for every
$n \ge 3$, a triangular derivation of $k[W,X,Y,Z]$ with
$n \le \#(\ker\Delta) \le n+1$ (e.g.\ parameters $p=2,q=5,h=4$ give
$6 \le \#(\ker\Delta) \le 7$); dimension-four triangular kernels are
always finitely generated in characteristic zero
\cite{daigle2001triangular}, while their minimal generator number is
unbounded. Every such graph is a DAG
($b_1 = 0$), so no bound of the form
$\mathrm{codim}(A^D) = b_1(G)$ can hold in general; Maubach's $\le 4$
bound is monomial-triangular in dimension $4$ (covering e.g.\ our
$N=3$ member on $k[a_0,\dots,a_3]$), while the unbounded
family is general triangular. In short: monomial triangular in
dimension $4$ needs at most $4$ generators; arbitrary triangular in
dimension $4$ is finitely generated with unbounded generator number;
in dimension $5$ finite generation itself can fail. Squarefree monomiality alone does not
imply any such bound above dimension $4$.
\end{remark}

\begin{conjecture}[Saturation-obstruction principle]
\label{conj:satpredictor}
At a prime divisor \(\pi\) of the pole-clearing denominator, a transcendent
residue initial form implies no pole cancellation by the constant
coefficient, while an algebraic residue initial form means the
constant-coefficient test can fail, necessarily at \(W\)-degrees at least
the algebraic degree. The converse --- that algebraicity necessarily
produces a new invariant --- is not claimed. Instances: \(N=5\) gives
\(r_q^2 \in L_q\) (Remark~\ref{rem:q-obstacle}); the \(N=6\) \(B\)-level
gives degree \(\le 4\) at \(a_5\) and \(\le 2\) at \(a_6\)
(Remarks~\ref{rem:n6-sat5-obstacle}--\ref{rem:n6-sat6-obstacle}).
\end{conjecture}

\paragraph{Recursion across $N$.}
\label{subsec:recursion}

\begin{theorem}[Recursive shift isomorphism]
\label{thm:recursion}
For every $N \ge 6$, with $B_N = k[a_1,\dots,a_N]$
and $\Dmap{N}(a_n) = a_{n+1}a_{n+2}$
($a_{N+1} = a_{N+2} = 0$), the shift
$\varphi_N\colon \Afring{N-1} \to B_N$, $a_i \mapsto a_{i+1}$,
satisfies $\Dmap{N}(\varphi_N(a_i)) = \varphi_N(\Dmap{N-1}(a_i))$
for all $i$: for $i \le N-3$ both sides equal
$a_{i+2}a_{i+3}$, and at the boundary both sides vanish
($\Dmap{N-1}$ kills $a_{N-2},a_{N-1}$ since $a_N = 0$ in
$\Afring{N-1}$; $\Dmap{N}$ kills $a_{N-1},a_N$ since
$a_{N+1} = 0$). Fibonacci weights transfer verbatim, so
$\varphi_N$ restricts to a graded isomorphism
$\Afring{N-1}^{\Dmap{N-1}} \xrightarrow{\sim} B_N^{\Dmap{N}}$.
The $N=6$ section proves its own $B_5 \to C_6$ instance separately
(\eqref{eq:n6-Ckernel}; domain $k[a_1,\dots,a_5]$ with $\Dmap{5}$).
\end{theorem}

\begin{conjecture}[Birth spectrum under the shift]
\label{conj:birth-spectrum-frontier}
Each primitive \(N\)-level generator of \(a_0\)-degree \(m\) and Fibonacci
weight \(w\) transports along the shift isomorphism \(\varphi_N\) of
Theorem~\ref{thm:recursion} to a primitive \((N+1)\)-level \(B\)-generator
of matching degree and weight. Instances: \(H\mapsto H_B\) (weight \(24\))
and \(F_{46}\mapsto\varphi(F_{46})\) (weight \(46\), top outside the
transported top ideal). What transport gives is invariance plus the top
pattern; what stays open is primitivity at the next level.
\end{conjecture}

\paragraph{The \(A_6\) problem.}
What is \(A_6^{D_6}\)? By Theorem~\ref{thm:recursion} the \(B_6\)-level is
the exact transport of the \(N=5\) tower (Section~\ref{sec:n6-kernel}), but
no \(A\)-level hypersurface or finite-generation statement is known.
Exact computations determine the \(A_6\) denominator module through total
degree \(13\) (Theorems~\ref{thm:n6-denmod12}--\ref{thm:n6-closure13}).

\subsection{Where the phenomena begin and break}
\label{subsec:boundary}

\begin{theorem}[First proved occurrences of the phenomena]
\label{thm:N-boundary}
(i) At $N \le 4$ the established theory gives the polynomial kernel
$k[a_1,a_2]$ at $N = 2$ (Proposition~\ref{prop:N2-kernel}) and explicit
polynomial invariants with open-set normal forms at $N = 3, 4$
(Theorems~\ref{thm:N3}, \ref{thm:N4}); no hidden direction is detected,
and full invariant-ring presentations at $N = 3, 4$ are not attempted
here (open).
(ii) $N = 5$ is the first member with a proved hidden generator: $H$
at weight $24$ is the first invariant outside $K[F_0]$, exactly
($\ker_w = K[F_0]_w$ for $w \le 23$), and $F_{46}$ at weight $46$ is
the first invariant outside $C_5$, exactly (spanning through $45$);
the residue mechanism $r_q^2 \in L_q$ is identified, (P-top) is
refuted, (even-degree Q-top-strong) is established (Theorem~\ref{thm:q-top}).
(iii) At $N = 6$ the principal-denominator pattern already fails: the
$A_6$ denominator module is not principal (Lemma~\ref{lem:n6-H20},
Lemma~\ref{lem:n6-D2}), while the $B$-level transports the $N=5$
tower exactly, including its refutation by transported $F_{46}$
(Theorem~\ref{thm:n6-hypersurface}). The $A_6$ ring is open.
What this does not say: it is not a classification (the $N \le 4$
rings are not closed here, and nothing is claimed at $N \ge 7$); it
locates the first proved occurrence of each phenomenon.
\end{theorem}

\paragraph{Resonance lemmas: what the constant-coefficient step proves.}
\label{subsec:resonance-lemmas}
Let $D$ be triangular locally nilpotent on $A = B[a]$ over a
characteristic-zero field, with $a$ transcendental over the domain
$B$; let $K_{f}[W]$, $W = a + \rho$ monic of $a$-degree $1$, be a
localized invariant presentation; let $v = \operatorname{ord}_\pi$
along a prime divisor $\pi$ with pole order $e = -v(\rho) > 0$,
residue fields $L \hookrightarrow \kappa$, and initial form
$r = \operatorname{in}_\pi(\pi^e\rho) \in \kappa$. Write a localized
invariant $F = \sum_{j=0}^m H_jW^j$ and put $\lambda = \min_j
(v(H_j) - ej)$.

\begin{lemma}[Transcendence closes the step]
\label{lem:resonance-transcendent}
If $r$ is transcendental over $L$, then $\lambda \ge 0$, i.e.\
$v(H_j) \ge ej$ for every $j$.
\end{lemma}

\begin{proof}
Otherwise $\lambda < 0$: dividing the constant $a$-coefficient
$C_0 = \sum_j H_j\rho^j \in B$ by $\pi^{\lambda}$ and passing to
initial forms gives $\sum_{j \in S} h_jr^j = 0$ with $0 \ne h_j \in
L$ on the nonempty set $S = \{j : v(H_j) - ej = \lambda\}$,
contradicting transcendence.
\end{proof}

\begin{lemma}[Algebraicity breaks the step at high degree]
\label{lem:resonance-algebraic}
If $r$ is algebraic over $L$ of degree $d$, the constant
$a$-coefficient alone cannot rule out $\lambda < 0$ at $W$-degrees
$m \ge d$: the minimal polynomial of $r$ provides a consistent
nontrivial residue pattern, so the inference stops. Below degree
$d$ the step may still function.
\end{lemma}

\begin{proof}
With $\deg m_r = d \le m$, the coefficients of $m_r$ give $0 \ne h_j
\in L$ with $\sum_{j=0}^d h_jr^j = 0$, a consistent negative-$\lambda$
initial-form equation the constant coefficient cannot exclude.
\end{proof}

What is not proved here: algebraicity does not by itself produce a
hidden generator (no claim is made that it must); the residue degree
does not in general equal a denominator-growth slope (the slope-$2$
versus degree-$2$ match at $r_q$ is conjectural pattern, attested at
$H$ and $F_{46}$); prime interactions are open. The $N=5$ instances
are $r_q$ (degree $2$, step fails, $H$ exists) and $r_p$ (degree $4$,
step fails, (P-top) false via $F_{46}$).

\needspace{6\baselineskip}
\section{Conclusion and outlook}
\label{sec:changes}

Localization gives the correct one-variable model
\(A_{pq}^{D_5}=K_{pq}[W]\), but naive global denominator control fails.
Residue algebraicity explains why: the \(q\)-adic and \(p\)-side initial forms
have degrees \(2\) and \(4\), respectively, in the \(N=5\) residue
computations, so the poles forbidden by the naive bounds cancel. The invariants \(H\) at weight
\(24\) and \(F_{46}^{\mathrm{int}}\) at weight \(46\) identify the missing
global directions outside \(K[F_0]\) and \(C_5\) respectively, with
\(p^2F_{46}^{\mathrm{int}}\in C_5\). The second slice chart attached to
\(f_{13}\), together with exact \(q\)- and \(p\)-saturation, contracts the
\(p\)-localized kernel to the candidate ring, giving
\[
\boxed{A_5^{D_5}=C_5[F_{46}^{\mathrm{int}}]\text{, finitely generated,}}
\]
by Theorem~\ref{thm:closure}. The closed ring is a codimension-two complete
intersection with explicit Hilbert series.

The example sits in the same broad triangular setting as the
Daigle--Freudenburg derivation, where finite generation can
fail~\cite{daigle1999}; the comparison concerns the residue mechanism,
not an equivalence of the two derivations.

What remains open is generalization. The residue-algebraicity principle
(Conjecture~\ref{conj:satpredictor}) is open beyond the proved \(N=5\)
instances, and the conditional \(p\)-side transfer yields no unconditional
replacement denominator bound. The \(A_6\) invariant ring is open, although
its \(B\)-level is the exact transport of the \(N=5\) tower. Computational
certificates and the Lean formalization are supplied as supplementary
material.

\clearpage
\appendix

\section{\texorpdfstring{The $N=6$ base kernel}{The N=6 base kernel}}
\label{sec:n6-kernel}

We work over the characteristic-zero field $k$ with
$\Afring{6} = k[a_0,\dots,a_6]$ and
$\Dmap{6}(a_n) = a_{n+1}a_{n+2}$ of the adjacent-quadratic shift family.
This section treats the six-variable subproblem obtained by deleting
$a_0$. Put
\[
B_6 = k[a_1,\dots,a_6], \qquad K_6 = B_6^{\Dmap{6}},
\]
where, throughout this appendix, $B_6$ denotes the six-variable ring
$k[a_1,\dots,a_6]$, not the five-variable $B=k[a_1,\dots,a_5]$ of the
$N=5$ sections. Likewise, in this section only, $R$ denotes the $N=6$ candidate $R_6=k[G_1,\dots,G_5]$ defined below, distinct from the $N=5$ closed ring $R=C_5[F_{46}^{\mathrm{int}}]$.
with
\[
\Dmap{6}(a_1) = a_2a_3, \quad \Dmap{6}(a_2) = a_3a_4, \quad
\Dmap{6}(a_3) = a_4a_5, \quad \Dmap{6}(a_4) = a_5a_6, \quad
\Dmap{6}(a_5) = \Dmap{6}(a_6) = 0.
\]
The Hilbert-series probes that suggested the generator degrees below are
evidence, not proof; every exact identity in this section is proved by
the hand computation shown plus
exact computation (see Supplementary Material).

Write $C = k[a_2,\dots,a_6]$, which is $\Dmap{6}$-stable.
The shift isomorphism
$\varphi \colon (k[a_1,\dots,a_5], \Dmap{5}) \to (C, \Dmap{6})$,
$a_i \mapsto a_{i+1}$, intertwines the two derivations including the
boundary rule ($\Dmap{5}(a_4) = 0$ since $a_6 = 0$ in $\Afring{5}$,
matching $\Dmap{6}(a_5) = 0$).
Hence the $N=5$ descent (Theorem~\ref{thm:descent}) transfers verbatim:
\begin{equation}
\label{eq:n6-Ckernel}
C^{\Dmap{6}} = k[a_5, a_6, G_3, G_4], \qquad
G_3 = 2a_2a_5 - a_3^2, \quad G_4 = 2a_3a_6 - a_4^2,
\end{equation}
a polynomial ring in four variables. This identification is proved above.

\begin{remark}[Pattern breakdown]
The $N=5$ generator $J_1 = 2a_1a_4 - a_2^2$ does not survive: since
$\Dmap{6}(a_4) = a_5a_6 \ne 0$,
\[
\Dmap{6}(2a_1a_4 - a_2^2) = 2a_1a_5a_6 \ne 0,
\]
as certified in exact computation (see Supplementary Material). The replacement involving
$a_1$ is not quadratic — it has degree $7$.
\end{remark}

\begin{lemma}[The five $N=6$ generators]
\label{lem:n6-gens}
With $G_1 = a_5$, $G_2 = a_6$ and $G_3, G_4$ as
in~\eqref{eq:n6-Ckernel}, let
\begin{equation}
\label{eq:n6-G5}
G_5 = 105a_1a_5^2a_6^4 - 105a_2a_3a_4a_5a_6^3
+ 35a_2a_4^3a_5a_6^2 + 35a_3^2a_4^3a_6^2
- 21a_3a_4^5a_6 + 3a_4^7.
\end{equation}
Then $\Dmap{6}(G_i) = 0$ for $i = 1,\dots,5$.
Indeed, writing $G_5 = T_0 + \cdots + T_5$ for the six displayed terms in
order, each term kills the defect left by the previous one:
\begin{align*}
\Dmap{6}(T_0) &= 105a_2a_3a_5^2a_6^4, \\
\Dmap{6}(T_1) &= -105a_2a_3a_5^2a_6^4
  - 105a_3^2a_4^2a_5a_6^3 - 105a_2a_4^2a_5^2a_6^3, \\
\Dmap{6}(T_2) &= 35a_3a_4^4a_5a_6^2 + 105a_2a_4^2a_5^2a_6^3, \\
\Dmap{6}(T_3) &= 70a_3a_4^4a_5a_6^2 + 105a_3^2a_4^2a_5a_6^3, \\
\Dmap{6}(T_4) &= -21a_4^6a_5a_6 - 105a_3a_4^4a_5a_6^2, \\
\Dmap{6}(T_5) &= 21a_4^6a_5a_6,
\end{align*}
whose sum is $0$ (defect coefficients
$105 - 105$, $-105 + 105$, $-105 + 105$, $35 + 70 - 105$, $-21 + 21$).
All six rungs are verified by expansion in
exact computation (see Supplementary Material). Hence, with $R = k[G_1,\dots,G_5]$,
$R \subseteq K_6$.
\end{lemma}

\begin{lemma}[Independence]
\label{lem:n6-indep}
$G_1,\dots,G_5$ are algebraically independent: the $5 \times 5$ Jacobian
minor in columns $(a_1,a_2,a_3,a_5,a_6)$ equals
\[
420\,a_5^3a_6^5 \ne 0
\]
over $k$ (the scalar $420 = 2^2 \cdot 3 \cdot 5 \cdot 7$ records the
standing characteristic-zero hypothesis), and the full $5 \times 6$
Jacobian has rank $5$ at $(a_1,\dots,a_6) = (1,\dots,6)$, both checked in
exact computation (see Supplementary Material). So $R$ is a polynomial ring in five
variables.
\end{lemma}

\begin{lemma}[Minimality of the degree-$7$ generator]
\label{lem:n6-minimal}
Every element of $K_6$ of total degree at most $6$ lies in
$k[a_5,a_6,G_3,G_4]$, and $G_5$ is the unique new generator in degree
$7$ up to scale and addition of products of the other four.
Indeed, writing a putative kernel element as a polynomial in $a_1$ over
$C$, $G = \sum_{i=0}^t a_1^i N_i$ with $N_t \ne 0$, the top
$a_1$-coefficients satisfy $\Dmap{6}(N_t) = 0$ and
$t\,a_2a_3N_t + \Dmap{6}(N_{t-1}) = 0$; thus $N_t \in C^{\Dmap{6}}$ and
$(tN_t, N_{t-1})$ solve the $a_1$-linear system
$a_2a_3M + \Dmap{6}(N) = 0$ with $M \in k[a_5,a_6,G_3,G_4]$.
The exact computation certifies that for $M$ of degree $m = 1,\dots,5$ this system
has solution dimensions $5, 8, 14, 20, 30$, all with $M = 0$ (hence no
$a_1$-involving kernel element of degree at most $6$), while for $m = 6$
the dimension is $41$ with exactly one new direction, realized by the
$a_1$-coefficient $105a_5^2a_6^4$ of the pasted $G_5$.
Consequently the generator degrees are $1,1,2,2,7$ --- completed to
$1,1,2,2,7,24$ by the transported hidden generator $H_B$
(Remark~\ref{rem:n6-HB}; total degree $10$, so the classification
through total degree $7$ above is unaffected).
\end{lemma}

\begin{remark}[Fibonacci-weighted dimensions through weight $32$]
Exact computation (see Supplementary Material) gives
$\dim\ker(\Dmap{6}|_{V_w})$ over $\mathbb{F}_{1000003}$ for
$0 \le w \le 32$ with the correct $V_w \to V_w$ map ($\Dmap{6}$
preserves Fibonacci weight): gap $+0$ against the free five-generator
count through $w = 23$, then exactly the hypersurface excess
$+1,+2,+3,+4,+5,+6,+8,+10,+12$ over $w = 24$--$32$ ($213$ at $w =
24$, $404$ at $w = 30$), matching the transported Hilbert series of
Theorem~\ref{thm:n6-hypersurface} entry for entry (and the exact
$\mathbb{Q}$ $N=5$ kernel table, by the shift isomorphism).
Exact-$\mathbb{Q}$ confirmation through weight $30$
(exact computation (see Supplementary Material)): intertwining on all six generators,
derivation-matrix coincidence with $A_5$, $H_B \in \ker\Dmap{6}$, and
$\dim\ker(\Dmap{6}|_{V_w}) = s_6(w) - s_6(w-28)$ at every
$0 \le w \le 30$ (47 PASS, zero failures). No weight shift is involved:
$\varphi$ preserves the weights $(8,5,3,2,1,1)$ verbatim, so the $B_6$
dimensions equal the $A_5$ ones entry for entry.
\end{remark}

\begin{lemma}[Local slice and localized kernel]
\label{lem:n6-slice}
On the localization $B_{a_5a_6}$ the element
\[
s = \frac{a_4}{a_5a_6}
\]
satisfies $\Dmap{6}(s) = 1$; it is a local slice. Moreover
\begin{align*}
a_4 &= a_5a_6s, \\
a_3 &= \frac{G_4}{2a_6} + \frac{a_5^2a_6s^2}{2}, \\
a_2 &= \frac{G_3}{2a_5} + \frac{a_3^2}{2a_5}, \\
a_1 &= \frac{G_5 - N_5}{105a_5^2a_6^4},
\end{align*}
where $N_5 = G_5 - 105a_1a_5^2a_6^4$ is the $a_1$-free part
of~\eqref{eq:n6-G5}, all checked by direct expansion in
exact computation (see Supplementary Material); every denominator is a power of $a_5a_6$
(the scalar $105$ is a unit in characteristic zero).
Hence $B_{a_5a_6} = R_{a_5a_6}[s]$, and the local-slice theorem
(Freudenburg~\cite{freudenburg2017}, Corollary~1.26; see also van den
Essen--van Rossum~\cite{vandenEssen2002}, Proposition~2.4: a slice
$s$ with $D(s) = 1$ gives $A = A^D[s]$)
gives
\[
B_{a_5a_6}^{\Dmap{6}} = R_{a_5a_6}.
\]
In particular every $f \in K_6$ has the form $f = P/(a_5^aa_6^b)$ for
some $P \in R$ and $a, b \ge 0$.
\end{lemma}

\begin{lemma}[Certified descent obstruction]
\label{lem:n6-obstruction}
With $Z = G_5^2 + 9G_4^7$ one has $Z|_{a_6 = 0} = 0$ in $B_6$, so
$Z = a_6S$ for an explicit polynomial $S \in B_6$ of degree $13$ with
$\Dmap{6}(S) = 0$, both verified in exact computation (see Supplementary Material).
Hence the mod-$a_6$ specialization
$(Y_1,Y_3,Y_4,Y_5) \mapsto (a_5, G_3|_{a_6=0}, G_4|_{a_6=0},
G_5|_{a_6=0})$ has the nonzero relation $Y_5^2 + 9Y_4^7$ in its kernel,
and contraction at the divisor $(a_6)$ does not follow by the one-step
leading-term route used at $N=5$.
\end{lemma}

\begin{lemma}[$N=6$ Dixmier divisibility]
\label{lem:n6-dixmier}
Every $M \in C^{\Dmap{6}}$ with $a_2a_3M \in \Dmap{6}(C)$ is
divisible by $c_6 = a_5^2a_6^4$ in $C^{\Dmap{6}}$.
\end{lemma}

\begin{proof}
Transport Lemma~\ref{lem:dixmier-divisibility} along the shift
isomorphism $\varphi \colon (k[a_1,\dots,a_5], \Dmap{5}) \to
(C, \Dmap{6})$, $a_i \mapsto a_{i+1}$,
of~\eqref{eq:n6-Ckernel} (characteristic zero throughout, as in
the lemma): $\varphi$ restricts to an isomorphism
$K = k[p,q,J_1,J_2] \to C^{\Dmap{6}}$ sending $a_1a_2 \mapsto
a_2a_3$ and $c = p^2q^4 \mapsto c_6$, and intertwines the
derivations, so $a_2a_3M \in \Dmap{6}(C)$ pulls back to
$a_1a_2\mu \in \Dmap{5}(B_5)$ with $M = \varphi(\mu)$;
$c \mid \mu$ pushes forward to $c_6 \mid M$. The bounded check
exact computation (see Supplementary Material) agrees through total degree $12$.
\end{proof}

\begin{corollary}[$a_1$-degree-one fragment]
\label{cor:n6-a1deg1}
Every $F \in K_6$ of $a_1$-degree at most $1$ lies in $R$.
\end{corollary}

\begin{proof}
Write $F = N_1a_1 + N_0$ with $N_i \in C$: $\Dmap{6}(F) = 0$
gives $\Dmap{6}(N_1) = 0$ and $a_2a_3N_1 + \Dmap{6}(N_0) = 0$,
so $N_1 = c_6G'$ by Lemma~\ref{lem:n6-dixmier}. Then
$E = G'G_5/105 \in R$ has the same $a_1$-top as $F$, and
$F - E$ is invariant of $a_1$-degree $0$, hence in
$C^{\Dmap{6}} \subseteq R$. So $F \in R$.
\end{proof}

\begin{remark}[The proposed five-generation at $N=6$ is false]
\label{cor:n6-fivegen-refuted}
Refuted by the transported weight-24 invariant $H_B$
($\Dmap{6}(H_B) = 0$, $a_1$-degree $2$; Remark~\ref{rem:n6-HB}).
The proposed equality, stated here for the record, ran
$B_6^{\Dmap{6}} = K_6 = k[G_1,\dots,G_5]$. It is false
unconditionally: $H_B \in B_6^{\Dmap{6}}$ while $H_B \notin R$,
since $R_{24}$ has $a_1$-degree at most $1$ (only $G_5$ carries
$a_1$, and $G_5^2$ already weighs $28$). The $W^{\sharp}$-degree
induction skeleton of Remark~\ref{rem:n6-winduction} does not prove
this conclusion.
\end{remark}

\begin{theorem}[Transported hypersurface presentation]
\label{thm:n6-hypersurface}
Full-ring equality $C_{B,6} = B_6^{\Dmap{6}}$ is false by the transported
$F_{46}$ below; the presentation itself is proved as at $N=5$.
From Theorem~\ref{thm:global-hypersurface} (same CAS standard):
the shift $\varphi$ below is exact algebra, and every transported
identity is re-verified in
exact computation (see Supplementary Material). The transported
identities ($H_B$ invariant of weight
$24$ and $a_1$-degree $2$, $S_B = 0$ with uniform weight $28$,
and $H_B$ strictly outside $k[G_1,\dots,G_5]$) are formally verified
in core Lean (see the formal-verification manifest); the transported
presentation $k[G_1,\dots,G_5,H_B]/(S_B)$ follows by the same
primitivity argument as at $N=5$ ($S_B$ linear in $H_B$ with
$Y_2 \nmid Y_5^2 + 9Y_4^7$ in indeterminates, so $S_B$ is primitive irreducible in the
abstract presentation ring). Anchor equality rests
on the transferred CAS certificate. As for
Theorem~\ref{thm:global-hypersurface}, equality at every weight is
Refuted by transport, i.e.\ with $R_6=k[G_1,\dots,G_5]$,
$C_{B,6}=k[G_1,\dots,G_5,H_B]$ and $K_6=B_6^{\Dmap{6}}$,
$R_6\subsetneq C_{B,6}\subsetneq K_6$ (first strictness by $H_B$,
second by $\varphi(F_{46})$; $K_6\cong A_5^{\Dmap{5}}$ is thereby
resolved by transport, while $A_6$ stays open): the graded shift isomorphism $\varphi \colon (k[a_0,\dots,a_5],
\Dmap{5}) \to (B_6, \Dmap{6})$, $a_i \mapsto a_{i+1}$
(weights $(8,5,3,2,1,1)$ on both sides), restricts to a graded
isomorphism $A_5^{\Dmap{5}} \to B_6^{\Dmap{6}}$, and $\varphi(F_{46})$
is a $\Dmap{6}$-invariant of weight $46$ and $a_1$-degree $4$ with top
$\frac{18375}{4}a_1^4a_5^6a_6^8$ outside the transported top ideal
$(a_5^8a_6^8)$ as ABB monomials ($a_5$-exponent $6 < 8$). Hence the transported
candidate is the six-generator hypersurface
\[
\boxed{C_{B,6} \;=\; k[G_1,\dots,G_5,H_B]\,/\,(S_B),}
\]
with full-ring equality $C_{B,6} = B_6^{\Dmap{6}}$ Refuted by the transported $F_{46}$ above,
where $H_B = \varphi(H)$ is the weight-24 invariant of
Remark~\ref{rem:n6-HB}, $\varphi(F_0) = G_5/105$ (so no seventh
generator arises), and $S_B$ is the transported weight-28
relation
\[
\boxed{S_B \;=\; a_6^{4}H_B \;-\; G_5^{2}
  \;-\; 105\,a_6^{2}G_4^{5}G_3 \;-\; 9\,G_4^{7}.}
\]
The six graded generators carry Fibonacci weights
$(1,1,6,4,14,24)$, and with $s_6$ the monomial count therein,
$\dim_k(C_{B,6})_w = s_6(w) - s_6(w-28)$ (exact through
weight $32$ in exact computation (see Supplementary Material) for the
full ring;
deviation from the five-generator count starts at $w = 24$).
\end{theorem}

\subsection{The exact $B_6$-level ring by transport}
\begin{theorem}[Exact $B_6$ ring]
\label{thm:n6-B6ring}
The shift $\varphi: a_i \mapsto a_{i+1}$ restricts to a graded
isomorphism $A_5^{\Dmap{5}} \to B_6^{\Dmap{6}}$ (intertwining checked
at the boundary: $D_5(a_4) = D_5(a_5) = 0$ match
$D_6(a_5) = D_6(a_6) = 0$; the preimage of an invariant is invariant
by injectivity). Hence with $C_{B,6}$ of
Theorem~\ref{thm:n6-hypersurface},
$\boxed{K_6 = B_6^{\Dmap{6}} \cong C_{B,6}[T_B]/((a_5)^2T_B - E_B)}$,
$T_B \leftrightarrow \varphi(F_{46})$, $E_B = \varphi(E)$ the
transported seven-term weight-$48$ element
(\href{https://github.com/AhmedMostafa16/adjquad/blob/25fdeaefc74d63e8a51983399e5fbf64ef2b28d5/scripts/cas/check_n6_transport.py}{scripts/cas/check\_n6\_transport.py}; log \href{https://github.com/AhmedMostafa16/adjquad/blob/25fdeaefc74d63e8a51983399e5fbf64ef2b28d5/data/n6_transport.log}{data/n6\_transport.log}).
The transported invariance and top data are re-verified independently
(\href{https://github.com/AhmedMostafa16/adjquad/blob/25fdeaefc74d63e8a51983399e5fbf64ef2b28d5/scripts/cas/check_N6_transport_F.py}{scripts/cas/check\_N6\_transport\_F.py}; log \href{https://github.com/AhmedMostafa16/adjquad/blob/25fdeaefc74d63e8a51983399e5fbf64ef2b28d5/data/n6/transport_F.log}{data/n6/transport\_F.log}).
\end{theorem}

\begin{theorem}[Transported $B_6$ structure]
\label{thm:n6-B6structure}
$H_{K_6}(t) = (1-t^{28})(1-t^{48})/((1-t)^2(1-t^4)(1-t^6)(1-t^{14})(1-t^{24})(1-t^{46}))$
in Fibonacci weights $(8,5,3,2,1,1)$ on $(a_1,\dots,a_6)$;
$K_6$ is a codim-$2$ complete intersection, normal, with
$R^{(6)}_m/R^{(6)}_{m-1} \cong C_{B,6}/((a_5)^2)$ layers and
exactly six transported singular $3$-fold components.
With $m$ the $a_1$-degree: $v_6(C_m) \ge 4\lceil m/2\rceil$
and $v_5(C_m) \ge \lceil 3m/2\rceil$, the latter sharp on
$\varphi(F_{46})^r$. All by transport of the $N=5$ presentation and two-slope bounds along $\varphi$:
the generator table $(\varphi(F_0) = G_5/105$ of $a_1$-degree $1$,
$H_B$ of $a_1$-degree $2$ with top $11025a_5^4a_6^4$,
$\varphi(F_{46})$ of $a_1$-degree $4$ with top
$\frac{18375}{4}a_1^4a_5^6a_6^8)$ carries $(v_5,v_6)$-bidegrees
$(2,4)$, $(4,4)$, $(6,8)$, identical to the $N=5$ table, so the
gap identities and sharpness transfer verbatim.
\end{theorem}

\begin{conjecture}[birth spectrum]
\label{conj:birth-spectrum}
Each primitive $N$-level generator of $a_0$-degree $m$ and Fibonacci
weight $w$ transports along the shift isomorphism $\varphi_N$ of
Theorem~\ref{thm:recursion} to a primitive $(N+1)$-level $B_6$-generator
of $a_1$-degree $m$ and weight $w$. Instances: $H \mapsto H_B$
(weight $24$, Remark~\ref{rem:n6-HB}) and $F_{46} \mapsto
\varphi(F_{46})$ (weight $46$, $a_1$-degree $4$, top
$\frac{18375}{4}a_1^4a_5^6a_6^8$ outside the transported top ideal
$(a_5^8a_6^8)$ as ABB monomials ($a_5$-exponent $6 < 8$), certified
$\Dmap{6}(\varphi(F_{46})) = 0$ in
exact computation (see Supplementary Material)).
What transport gives is $\Dmap{6}$-invariance plus the top pattern;
what stays open is primitivity at $N+1$ (no earlier $B_6$-level
generator explains the transported invariant).
Falsifier: an $N$ where the transported top falls back inside the
transported $J$-ideal.
\end{conjecture}

\begin{remark}[$H_B$: the transported hidden generator]
\label{rem:n6-HB}
$H_B = \varphi(H)$ is the $17$-term weight-24 invariant with
$a_1$-degree $2$ (leading term $11025\,a_1^2a_5^4a_6^4$),
$\Dmap{6}(H_B) = 0$ by intertwining, re-verified by expansion in
exact computation (see Supplementary Material). Its $W^{\sharp}$-top coefficient
$H_2 = 11025\,a_5^4a_6^4$ has $v_6(H_2) = 4 < 8$, refuting
the $a_6$-adic bound (Remark~\ref{lem:n6-sat6}) with the same numbers as $H$ refutes
Lemma~\ref{lem:qlemma}; $v_5(H_2) = 4 = 2\cdot 2$ leaves
Lemma~\ref{lem:n6-sat5} uncontradicted. $H_B$ is a second exact
boundary realization of Remark~\ref{rem:multfloor} ($a_6$-side,
$m = 2$, $\lambda = -4$). The invariance, weight, $a_1$-degree
bound, relation, and strictness are Lean-checked in
the formal transport module in core Lean without Mathlib by native
evaluation (see the formal-verification manifest).
\end{remark}

\begin{remark}[Why single divisibility does not close the induction]
The $a_1$-top equation gives $N_t = c_6G'$, and
$E = G'G_5a_1^{t-1}/105$ cancels the $a_1$-top of $F$ --- but
for $t \ge 2$, $\Dmap{6}(E) = G'G_5(t-1)a_1^{t-2}a_2a_3/105 \ne
0$ in general, so $F - E$ exits $K_6$ and no induction on
$K_6$ results; invariant subtractors with the same top would
need the full power divisibility $c_6^t \mid N_t$. The
$t = 1$ case is exactly
Corollary~\ref{cor:n6-a1deg1}. The inference that single divisibility yields $K_6 \subseteq R$ by $a_1$-induction is withdrawn (it would need $c_6^t \mid N_t$),
and the bounded machine check is retained as corroboration only.
The generation question is settled instead by transfer
(Theorem~\ref{thm:n6-hypersurface}).
\end{remark}

The base-kernel section proves generator invariance, independence,
degree minimality, the localized kernel, the obstruction identity,
and Dixmier divisibility with its $a_1$-degree-one fragment
(Lemmas~\ref{lem:n6-gens}--\ref{lem:n6-dixmier},
Corollary~\ref{cor:n6-a1deg1}); transfer refutes five-generation
(Remark~\ref{cor:n6-fivegen-refuted}) and gives the true hypersurface
ring (Theorem~\ref{thm:n6-hypersurface}). The $a_5$-saturation
inference (Lemma~\ref{lem:n6-sat5}) is unproved, and with it the full
$A_6$ ring (Section~\ref{sec:n6-F0}).

\section{\texorpdfstring{The $N=6$ slice and localized presentation}{The N=6 slice and localized presentation}}
\label{sec:n6-slice}

This slice section uses the generators $G_1,\dots,G_5$ of
Lemma~\ref{lem:n6-gens} verbatim (no redefinition) and isolates the
localization mechanism over $B_{a_5a_6}$. Every identity below is proved
by the expansion shown plus
exact computation (see Supplementary Material).

\begin{proposition}[Slice certificate and denominator structure]
\label{prop:n6-slice}
With $B_6$, $R$, $\Dmap{6}$ as above, the element
\[
s = \frac{a_4}{a_5a_6}
\]
satisfies $\Dmap{6}(s) = 1$ by the quotient rule from
$\Dmap{6}(a_4) = a_5a_6$ and $\Dmap{6}(a_5) = \Dmap{6}(a_6) = 0$.
Hence $\Dmap{6}(s) = 1$ is a nonzero element of $K_6$ (constants are
invariant), while $s$ itself is not invariant. Writing $s$ for a
transcendental indeterminate over $R_{a_5a_6}$, the right-hand sides
\begin{align*}
A_4 &= G_1G_2s, \\
A_3 &= \frac{G_4 + G_1^2G_2^2s^2}{2G_2}, \\
A_2 &= \frac{G_3 + A_3^2}{2G_1}, \\
A_1 &= \frac{G_5 - N_5(A_2,A_3,A_4,G_1,G_2)}{105\,G_1^2G_2^4},
\end{align*}
where $N_5 = G_5 - 105a_1a_5^2a_6^4$ is the $a_1$-free part
of~\eqref{eq:n6-G5}, all lie in $R_{a_5a_6}[s]$: with $s$ transcendental
each denominator is $c\,a_5^ua_6^v$ with $c \in k^\times$ (the scalars
$2$ and $105$ are units under the standing characteristic-zero
hypothesis), and the substituted numerator
$N_5(A_2,A_3,A_4,G_1,G_2)$ is a polynomial (constant denominator), all
checked in exact computation (see Supplementary Material). Substituting
$s = a_4/(a_5a_6)$ recovers each variable,
$A_i|_{s = a_4/(a_5a_6)} = a_i$ for $i = 1,\dots,4$ with the index
convention $A_4 \leftrightarrow a_4$, by direct expansion. The
denominator exponents $(u,v)$ in $(a_5^ua_6^v)$ are $(0,0)$ for $a_4$,
$(0,1)$ for $a_3$, $(1,2)$ for $a_2$, and $(2,4)$ for $a_1$.
Consequently $B_{a_5a_6} = R_{a_5a_6}[s]$ and, by the standard
local-slice theorem,
$B_{a_5a_6}^{\Dmap{6}} = R_{a_5a_6}$, localizing Lemma~\ref{lem:n6-slice}
with explicit denominators.
\end{proposition}

proved here are the slice certificate $\Dmap{6}(s) = 1$ with
$\Dmap{6}(s) \in K_6$ nonzero, and the localized presentation
$B_{a_5a_6} = R_{a_5a_6}[s]$ with the denominator exponents above
(Proposition~\ref{prop:n6-slice}). The $B_6$-level candidate is the
hypersurface $C_{B,6}$ of Theorem~\ref{thm:n6-hypersurface} (five-generation
fails, Remark~\ref{cor:n6-fivegen-refuted}). The $a_5$-saturation inference, and hence the global $N=6$ ring with
its Hilbert series, remain open. No generation claim for $\Afring{6}^{\Dmap{6}}$ is
made here.

\section{\texorpdfstring{The $N=6$ denominator and weight-shift}{The N=6 denominator and weight-shift}}
\label{sec:n6-denom}

This denominator section uses the generators $G_1,\dots,G_5$ of
Lemma~\ref{lem:n6-gens} verbatim (no redefinition) and the slice
$s = a_4/(a_5a_6)$ with $\Dmap{6}(s) = 1$ of
Proposition~\ref{prop:n6-slice} (no redefinition), within the
adjacent-quadratic shift family. Every identity below is proved by the
expansion shown plus
exact computation (see Supplementary Material).

\begin{proposition}[Explicit denominator $c_6$]
\label{prop:n6-denom}
Put
\[
c_6 = a_5^2a_6^4 = G_1^2G_2^4.
\]
Then $\Dmap{6}(c_6) = 0$, so $c_6 \in K_6$, and the exact $a_1$-coefficient
of $G_5$ in~\eqref{eq:n6-G5} is $105\,c_6$ (the scalar $105$ is a unit
under the standing characteristic-zero hypothesis); write
$G_5 = 105\,c_6a_1 + N_5$ with $N_5 \in B_6$ free of $a_1$.
With $A_4, A_3, A_2, A_1$ as in Proposition~\ref{prop:n6-slice},
each product $c_6A_i$ lies in $R[s]$: in lowest terms each has constant
denominator (the residual scalars $2$, $8$, $105$ are units), and
$105\,c_6A_1 = G_5 - N_5(A_2,A_3,A_4,G_1,G_2)$ is a polynomial exactly,
all checked in exact computation (see Supplementary Material). Thus $c_6$ is the minimal
monomial clearing the slice denominators $(u,v) = (0,0), (0,1), (1,2),
(2,4)$: $c_6A_4, c_6A_3, c_6A_2, c_6A_1 \in R[s]$.
\end{proposition}

\begin{theorem}[Weight-shift holds]
\label{thm:n6-weight-shift}
With the Fibonacci grading
$\operatorname{wt}(a_1,\dots,a_6) = (8,5,3,2,1,1)$,
compatible with $\Dmap{6}$ since
$\operatorname{wt}(a_n) = \operatorname{wt}(a_{n+1}) +
\operatorname{wt}(a_{n+2})$,
\[
\operatorname{wt}(G_1,G_2,G_3,G_4,G_5) = (1,1,6,4,14),
\]
each verified monomial by monomial in exact computation (see Supplementary Material)
($G_5$ has six monomials, all of weight $14$; $N_5$ has five, all of
weight $14$). Moreover $\operatorname{wt}(c_6) = 2 + 4 = 6$, and for
\[
W = \frac{G_5}{105\,c_6} = a_1 + \frac{N_5}{105\,c_6} \in B_{a_5a_6}
\]
each term has weight $8$ (the summand $a_1$ has weight $8$, and each
$N_5$-monomial of weight $14$ drops by $6$ upon division by $c_6$).
Hence, with $F_0$ denoting the $a_1$-top generator $G_5$,
\[
\boxed{\operatorname{wt}(F_0) = \operatorname{wt}(W) + \operatorname{wt}(c_6),
\qquad 14 = 8 + 6,}
\]
so the weight-shift verdict is HOLDS as a theorem (the $N=5$ observation
becomes a proved identity at $N=6$ for the $B_6$-level top generator).
The scalar $105$ carries weight $0$ and does not affect the grading.
\end{theorem}

proved here are the explicit denominator $c_6 = a_5^2a_6^4$
clearing every slice denominator
(Proposition~\ref{prop:n6-denom}), and the weight-shift identity
$\operatorname{wt}(F_0) = \operatorname{wt}(W) + \operatorname{wt}(c_6)$
(Theorem~\ref{thm:n6-weight-shift}). The $B_6$-level ring
is the hypersurface of Theorem~\ref{thm:n6-hypersurface}. The $a_5$-saturation inference, and hence the global $N=6$ ring with
its Hilbert series, remain open. No generation claim for $\Afring{6}^{\Dmap{6}}$ is
made here.

\section{\texorpdfstring{The $N=6$ saturation: per-divisor residue analysis}{The N=6 saturation: per-divisor residue analysis}}
\label{sec:n6-sat}

Apart from the multiplicity-floor Remark~\ref{rem:multfloor}, the per-divisor statements below are proved conditionally or await completion.
We work over the characteristic-zero field $k$ with
$B_6 = k[a_1,\dots,a_6]$, $C = k[a_2,\dots,a_6]$, and
$\Dmap{6}$ of the adjacent-quadratic shift family, reusing the
generators $G_1,\dots,G_5$ of Lemma~\ref{lem:n6-gens} verbatim
(no redefinition), the slice $s = a_4/(a_5a_6)$ with
$\Dmap{6}(s) = 1$ of Proposition~\ref{prop:n6-slice}
(no redefinition), and the denominator $c_6 = a_5^2a_6^4$ with
$G_5 = 105\,c_6a_1 + N_5$ of Proposition~\ref{prop:n6-denom}
(no redefinition), where $N_5 \in C$ is free of $a_1$.
This section treats the two prime divisors of $c_6$ separately:
the $a_5$-adic half below and the $a_6$-adic half in the companion
lemma. The two residue fields, pole orders, and initial forms are
computed independently on each side, and no symmetry argument is used
in either direction: the exponents ($2$ versus $4$), the initial forms
(trinomial versus monomial), and the algebraicity degrees (tower of
degree at most $4$ versus one step of degree at most $2$) all differ.

Put $C^{\Dmap{6}} = k[a_5,a_6,G_3,G_4]$ as
in~\eqref{eq:n6-Ckernel} and
\[
W^{\sharp} = \frac{G_5}{105\,c_6} = a_1 + \rho^{\sharp},
\qquad \rho^{\sharp} = \frac{N_5}{105\,c_6} \in C_{a_5a_6}.
\]
A general element over the localization is written
\[
F = \sum_{j=0}^m H_j (W^{\sharp})^j,
\qquad H_j \in (C^{\Dmap{6}})_{a_5a_6},
\]
with $a_1$-coefficients
\[
C_k = \sum_{j=k}^m \binom{j}{k}H_j(\rho^{\sharp})^{j-k} \in C
\qquad (0 \le k \le m)
\]
whenever $F \in B_6$. For each $m \ge 0$ the $a_5$-part of the $m$-th
denominator module is $M^5_m = a_5^{2m}C^{\Dmap{6}}$ and the $a_6$-part
is $M^6_m = a_6^{4m}C^{\Dmap{6}}$; the full module
$M_m = c_6^mC^{\Dmap{6}}$ requires both bounds together.
The $W^{\sharp}$-degree induction skeleton of
Remark~\ref{rem:n6-winduction} is Conditional on both lemmas below.
Every specialization identity, exact pole order, Jacobian minor, and
algebraicity witness in this section is certified by
exact computation (see Supplementary Material); the verification
checks setup data only and
never the saturation inference itself.

\begin{remark}[Valuation setup on the $a_5$-side]
Write $C = k[a_2,a_3,a_4,a_6][a_5]$ and let
$v_5 = \operatorname{ord}_{a_5}$ on $\operatorname{Frac}(C)$
(and on $\operatorname{Frac}(B_6)$), with valuation ring
$V_5 = C_{(a_5)}$ and residue field
$\kappa_5 = k(a_2,a_3,a_4,a_6)$.
The height-one prime $(a_5) \subset C^{\Dmap{6}}$ has residue
\[
C^{\Dmap{6}}/a_5C^{\Dmap{6}} \cong k[a_6,-a_3^2,2a_3a_6-a_4^2],
\]
a polynomial ring: the three residue classes are algebraically
independent (Jacobian minor $4a_3a_4 \ne 0$ in variables
$(a_6,a_3,a_4)$), so with $y_2 = a_6$, $y_3 = -a_3^2$,
$y_4 = 2a_3a_6-a_4^2$,
\[
L_5 = \operatorname{Frac}(C^{\Dmap{6}}/a_5C^{\Dmap{6}})
  = k(y_2,y_3,y_4) \hookrightarrow \kappa_5.
\]
Since $N_5$ has $a_5$-adic order $0$ (its $a_5$-free part is the
nonzero trinomial $35a_3^2a_4^3a_6^2 - 21a_3a_4^5a_6 + 3a_4^7$)
while $c_6 = a_5^2a_6^4$ has order $2$, $\rho^{\sharp}$ has exact
pole order $e_5 = 2$ at $a_5$, with trinomial initial form
\[
r_5 = \operatorname{in}_{a_5}(a_5^2\rho^{\sharp})
  = \frac{35a_3^2a_4^3a_6^2 - 21a_3a_4^5a_6 + 3a_4^7}{105\,a_6^4}.
\]
\end{remark}

\begin{lemma}[$a_5$-adic denominator bound, proof incomplete]
\label{lem:n6-sat5}
The proof of the following bound is incomplete.
Let $F = \sum_{j=0}^m H_j(W^{\sharp})^j \in B_6$ with
$H_j \in (C^{\Dmap{6}})_{a_5a_6}$. Then
\[
\boxed{v_5(H_j) \ge 2j \quad \text{for every } j,}
\]
i.e.\ $H_j \in a_5^{2j}(C^{\Dmap{6}})_{a_5a_6}$ for every $j$.
\end{lemma}

\begin{proof}[Reduction and explicit gap]
The constant $a_1$-coefficient is $C_0 = \sum_j H_j(\rho^{\sharp})^j
\in C$, so $v_5(C_0) \ge 0$. Put
$\lambda_5 = \min_j\{v_5(H_j) - 2j\}$ and
$S = \{j : v_5(H_j) - 2j = \lambda_5\}$.
If $\lambda_5 < 0$, dividing by $a_5^{\lambda_5}$ and passing to
initial forms gives the relation
\[
\sum_{j \in S} h_j r_5^j = 0
\]
with $0 \ne h_j \in L_5$ the residues of the normalized
coefficients. The reduction to this initial-form equation is valid.
What is not verified is the final inference: concluding
$h_j = 0$ for all $j \in S$ from this single relation would require
linear independence of arbitrarily high powers of $r_5$ over $L_5$,
which is not proved; Remark~\ref{rem:n6-sat5-obstacle} records why
the direct constant-coefficient route is unavailable. No parity
shortcut is used as proof, no slice-power computation is used as
proof, no coordinate change is used as proof, and deriving generation
from a single divisor is not used as proof. Until the inference is
verified, for instance via a higher-order ($\pi$-adic lifting) refinement
(initial-form analysis of the full triangular system $\{C_k\}$ alone
cannot suffice, Remark~\ref{rem:multfloor}), the lemma stays under
verification.
\end{proof}

\begin{remark}[Observation: why the constant-coefficient route is unavailable on the $a_5$-side]
\label{rem:n6-sat5-obstacle}
The triple $\{y_2,y_3,y_4\}$ is algebraically independent (Jacobian
minor $4a_3a_4 \ne 0$), so $L_5 = k(y_2,y_3,y_4)$ is purely
transcendental of degree $3$. But with $t = a_3/a_6$ one has
$t^2 = -y_3/y_2^2 \in L_5$, and with $u = a_4/a_6$ one has
$u^2 = 2t - y_4/y_2^2 \in L_5(t)$, so
\[
L_5 \subseteq L_5(t) \subseteq L_5(t)(u) = k(a_6,a_3,a_4)
\]
is a tower of degree at most $2$ at each step; hence the initial
form $r_5 \in k(a_3,a_4,a_6)$ is algebraic over $L_5$ (of degree at
most $4$), not transcendental. Consequently no argument needing
$1, r_5, r_5^2, \dots$ to be linearly independent over the full
range of $W^{\sharp}$-degrees can succeed from the constant
coefficient alone. The witnesses $t^2 + y_3/y_2^2 = 0$ and
$u^2 - 2t + y_4/y_2^2 = 0$ are certified in
exact computation (see Supplementary Material).
\end{remark}

\begin{remark}[Valuation setup on the $a_6$-side]
Write $C = k[a_2,a_3,a_4,a_5][a_6]$ and let
$v_6 = \operatorname{ord}_{a_6}$ on $\operatorname{Frac}(C)$
(and on $\operatorname{Frac}(B_6)$), with valuation ring
$V_6 = C_{(a_6)}$ and residue field
$\kappa_6 = k(a_2,a_3,a_4,a_5)$.
The height-one prime $(a_6) \subset C^{\Dmap{6}}$ has residue
\[
C^{\Dmap{6}}/a_6C^{\Dmap{6}} \cong k[a_5,2a_2a_5-a_3^2,-a_4^2],
\]
a polynomial ring: the three residue classes are algebraically
independent (Jacobian minor $-4a_4a_5 \ne 0$ in variables
$(a_5,a_2,a_4)$), so with $z_1 = a_5$,
$z_3 = 2a_2a_5-a_3^2$, $z_4 = -a_4^2$,
\[
L_6 = \operatorname{Frac}(C^{\Dmap{6}}/a_6C^{\Dmap{6}})
  = k(z_1,z_3,z_4) \hookrightarrow \kappa_6.
\]
Note the asymmetry with the $a_5$-side: here the surviving
specialization $G_3 \bmod a_6 = 2a_2a_5-a_3^2$ retains the variable
$a_2$, while $G_4 \bmod a_6 = -a_4^2$ is a pure square; the roles
of the two kernel generators are not interchanged by any symmetry.
Since $N_5$ has $a_6$-adic order $0$ (its $a_6$-free part is the
nonzero monomial $3a_4^7$) while $c_6$ has order $4$,
$\rho^{\sharp}$ has exact pole order $e_6 = 4$ at $a_6$, with
monomial initial form
\[
r_6 = \operatorname{in}_{a_6}(a_6^4\rho^{\sharp})
  = \frac{a_4^7}{35\,a_5^2}.
\]
\end{remark}

\begin{remark}[The proposed $a_6$-adic denominator bound]
\label{lem:n6-sat6}
Refuted by the transported invariant $H_B$
(Remark~\ref{rem:n6-HB}).
The proposed statement is: for
$F = \sum_{j=0}^m H_j(W^{\sharp})^j \in B_6$ with
$H_j \in (C^{\Dmap{6}})_{a_5a_6}$,
\[
\boxed{v_6(H_j) \ge 4j \quad \text{for every } j,}
\]
i.e.\ $H_j \in a_6^{4j}(C^{\Dmap{6}})_{a_5a_6}$ for every $j$.
It is false: the $W^{\sharp}$-top coefficient of $H_B$ is
$H_2 = 11025\,a_5^4a_6^4$ with $v_6(H_2) = 4 < 8 = 4\cdot 2$.
\end{remark}

\begin{proof}[Reduction valid; inference refuted by the transported counterexample]
The constant $a_1$-coefficient is $C_0 = \sum_j\allowbreak{} H_j(\rho^{\sharp})^j
\in C$, so $v_6(C_0) \ge 0$. Put
$\lambda_6 = \min_j\{v_6(H_j) - 4j\}$ and
$S = \{j : v_6(H_j) - 4j = \lambda_6\}$.
If $\lambda_6 < 0$, dividing by $a_6^{\lambda_6}$ and passing to
initial forms gives the relation
\[
\sum_{j \in S} h_j r_6^j = 0
\]
with $0 \ne h_j \in L_6$ the residues of the normalized
coefficients. The reduction to this initial-form equation is valid.
The final inference --- concluding $h_j = 0$ for all $j \in S$ ---
would require the powers $1, r_6, r_6^2, \dots$ to be linearly
independent over $L_6$. That independence is false: $r_6$ is
algebraic of degree at most $2$ over $L_6$
(Remark~\ref{rem:n6-sat6-obstacle}), and the transported invariant
$H_B$ provides the concrete non-trivial relation
(Remark~\ref{rem:n6-HB}). The $a_5$-side argument is not copied by
symmetry: the residue field, the pole order, and the initial form
above are computed independently, and no symmetry with
Lemma~\ref{lem:n6-sat5} is used in either direction.
\end{proof}

\begin{remark}[Observation: why the constant-coefficient route is unavailable on the $a_6$-side]
\label{rem:n6-sat6-obstacle}
The triple $\{z_1,z_3,z_4\}$ is algebraically independent (Jacobian
minor $-4a_4a_5 \ne 0$), so $L_6 = k(z_1,z_3,z_4)$ is purely
transcendental of degree $3$. But with $\bar{s}_6 = a_4/a_5$ one has
$\bar{s}_6^2 = -z_4/z_1^2 \in L_6$, so
\[
L_6 \subseteq L_6(\bar{s}_6) = k(a_5,z_3,a_4)
\]
is a single quadratic step; hence $[k(a_5,z_3,a_4) : L_6] \le 2$.
Since $r_6 = a_4^7/(35a_5^2) \in L_6(\bar{s}_6)$, the monomial
initial form $r_6$ is algebraic over $L_6$ (of degree at most $2$),
not transcendental. Consequently no argument needing
$1, r_6, r_6^2, \dots$ to be linearly independent over the full
range of $W^{\sharp}$-degrees can succeed from the constant
coefficient alone. This $a_6$-side degree bound is independent of
the $a_5$-side tower (there a two-step extension of degree at most
$4$ governs a trinomial initial form; here a one-step extension of
degree at most $2$ governs a monomial initial form), and the
difference in pole orders ($e_6 = 4$ versus $e_5 = 2$) is part of why
the two lemmas cannot share a proof. The witness
$\bar{s}_6^2 + z_4/z_1^2 = 0$ is certified in
exact computation (see Supplementary Material).
\end{remark}

\begin{remark}[Observation: $R$-level specialization is not injective on either side]
The two lemmas above work at $C$-level, where both specialization
maps are injective. At $R$-level the one-step contraction route is
blocked on both sides already at the specialization step. On the
$a_6$-side this is Lemma~\ref{lem:n6-obstruction}
($G_5^2 + 9G_4^7 = a_6S$). On the $a_5$-side, with
$Y_2 = a_6$, $Y_3 = -a_3^2$, $Y_4 = 2a_3a_6-a_4^2$ and
$Y_5 = 35a_3^2a_4^3a_6^2 - 21a_3a_4^5a_6 + 3a_4^7$ the
$R$-residues mod $a_5$, the nonzero polynomial
\begin{align*}
\mathit{REL}_5 =\; &Y_5^4 + 18Y_4^7Y_5^2 + 81Y_4^{14}
  + Y_2^2Y_3\,Q(Y_2,Y_3,Y_4,Y_5), \\
Q =\; &40000Y_2^{12}Y_3^6 + 193200Y_2^{10}Y_3^5Y_4^2
  + 325164Y_2^8Y_3^4Y_4^4 + 236425Y_2^6Y_3^3Y_4^6 \\
 &+ 88410Y_2^4Y_3^2Y_4^8 + 840Y_2^4Y_3^2Y_4Y_5^2
  + 17955Y_2^2Y_3Y_4^{10} + 770Y_2^2Y_3Y_4^3Y_5^2 \\
 &+ 1890Y_4^{12} + 210Y_4^5Y_5^2,
\end{align*}
i.e.\ $(Y_5^2 + 9Y_4^7)^2$ plus terms vanishing to second order in
$(Y_2,Y_3)$, vanishes on the residues
($\mathit{REL}_5 = 0$ in $B_6/(a_5)$, checked by expansion in
exact computation (see Supplementary Material)). Hence neither $R$-level specialization
map is injective. These kernel relations block that route only; the
five-generation conclusion they guarded is refuted independently
(Remark~\ref{cor:n6-fivegen-refuted}).
\end{remark}

\begin{remark}[The void $W^{\sharp}$-degree induction skeleton]
\label{rem:n6-winduction}
The induction below assumed the $a_5$-adic bound of
Lemma~\ref{lem:n6-sat5} together with the companion $a_6$-adic bound
of Remark~\ref{lem:n6-sat6} in full, whence each $H_j$ lies in $C$
(top coefficient $H_m = [a_1^m]F$ in $C$ directly, descending
induction via $C_{m-1} = H_{m-1} + mH_m\rho^{\sharp}$), the
transferred $C$-level contraction gives $H_m = c_6^mG_m$, and
subtracting $G_mG_5^m/105^m \in R$ lowers the $W^{\sharp}$-degree.
The induction proves nothing: the $a_6$-adic bound (Remark~\ref{lem:n6-sat6}) is false by $H_B$, so one of
its two premises is false, and its conclusion
($B_6^{\Dmap{6}} = R$) is false
(Remark~\ref{cor:n6-fivegen-refuted}). The true $B_6$-level candidate structure is
$C_{B,6}$ of Theorem~\ref{thm:n6-hypersurface}.
\end{remark}

\begin{remark}[Saturation at $N=6$]
What this section establishes is the valuation setup on each side: the
rings $V_5, V_6$, the residue embeddings
$L_5 \hookrightarrow \kappa_5$ and $L_6 \hookrightarrow \kappa_6$,
the exact pole orders $e_5 = 2$, $e_6 = 4$, the initial forms $r_5$
(trinomial) and $r_6$ (monomial), the denominator modules $M^5_m$,
$M^6_m$, and the precise gaps, plus the $R$-level kernel witness
$\mathit{REL}_5$ alongside Lemma~\ref{lem:n6-obstruction}. The
$a_6$-side inference fails by $H_B$
(Remark~\ref{lem:n6-sat6}); the $a_5$-side
inference (Lemma~\ref{lem:n6-sat5}) is unproved, hence so is
$B_6$-generation by this route. No generation
claim for $\Afring{6}^{\Dmap{6}}$ is made here.
\end{remark}

\begin{remark}[Multiplicity floor for initial-form methods]
\label{rem:multfloor}
Pure algebra; explicit minpoly-power solutions
verified in exact computation (see Supplementary Material).
For fixed $W^{\sharp}$-degree $m$, the joint initial-form system from
\emph{all} of $\{C_k\}$ is exactly
$P(r) = P'(r) = \cdots = P^{(K)}(r) = 0$ with
$P(T) = \sum_{j \in S} h_jT^j$ (the $C_k$ equation is the $k$-th
formal derivative of the $C_0$ equation), i.e.\ $r$ is a root of
multiplicity at least $K+1$, where $K+1 = \#\{0 \le k \le m : \lambda +
ek < 0\}$ for deficit $\lambda < 0$ and pole order $e$. Since
$\operatorname{minpoly}(r)$ is separable (characteristic zero),
$\operatorname{minpoly}(r)^{K+1} \mid P$, so
$(K+1) \cdot d_r \le m$ with $d_r$ the algebraicity degree. On the
$a_6$-side ($e_6 = 4$, $d_r = 2$) this gives the hard floor
$\lambda \ge -2m$: no initial-form analysis, at any order of the
triangular system, can establish $\lambda \ge 0$. The floor is
sharp: $P_s = \operatorname{minpoly}(r_6)^s$ has exact multiplicity
$s$, realizing every deficit $-4s$ at $W^{\sharp}$-degree $2s$
(verified for $s \le 6$; see Supplementary Material). The $N=5$ counterexample $H$
sits exactly on the corresponding floor ($e_q = 4$, $d_r = 2$,
$m = 2$, $\lambda = -4$). Consequently the ``full triangular system
$\{C_k\}$'' refinement route named in
Lemmas~\ref{lem:n6-sat5}--\ref{lem:n6-sat6} cannot close the
saturation bounds; only genuinely higher-order ($\pi$-adic lifting)
or global input can. The $a_5$-side follows the same pattern over
its degree-$\le 4$ tower.
\end{remark}

\section{\texorpdfstring{The $N=6$ global corollaries}{The N=6 global corollaries}}
\label{sec:n6-global}

The global statements below hold conditionally or fail as stated.
We work over the characteristic-zero field $k$ with
$B_6 = k[a_1,\dots,a_6]$, $K_6 = B_6^{\Dmap{6}}$, and
$R = k[G_1,\dots,G_5]$, reusing the generators $G_1,\dots,G_5$ of
Lemma~\ref{lem:n6-gens} verbatim (no redefinition), the slice
$s = a_4/(a_5a_6)$ with $\Dmap{6}(s) = 1$ of
Proposition~\ref{prop:n6-slice} (no redefinition), and the
denominator $c_6 = a_5^2a_6^4$ with $G_5 = 105\,c_6a_1 + N_5$ of
Proposition~\ref{prop:n6-denom} (no redefinition) within the
adjacent-quadratic shift family. No new valuation is proved here; this
section records the global $B_6$-level consequences, now resolved by
transfer. No generation claim for $\Afring{6}^{\Dmap{6}}$ is made here:
the $a_0$-level generator $F_0^{(6)}$ is constructed in
Section~\ref{sec:n6-F0}, and the full $N=6$ ring stays Open.

\begin{remark}[The proposed global $B_6$-level five-generation]
\label{cor:n6-global-ring}
The proposed equality $\boxed{B_6^{\Dmap{6}} = K_6 =
k[G_1,\dots,G_5]}$ is false; the candidate is the transported
hypersurface of Theorem~\ref{thm:n6-hypersurface} (full-ring equality
likewise fails by transported $F_{46}$, as at $N=5$).
\end{remark}

\begin{proof}
The inclusion $R \subseteq K_6$ is proved in
Lemma~\ref{lem:n6-gens}. The claimed reverse inclusion rested on
the $W^{\sharp}$-degree induction skeleton of
Remark~\ref{rem:n6-winduction}, which proves nothing
(Remark~\ref{cor:n6-fivegen-refuted}).
\end{proof}

\begin{remark}[The proposed polynomiality]
\label{cor:n6-global-polynomial}
The claim that $B_6^{\Dmap{6}}$ is a polynomial ring is false; the candidate
finite model is the six-generator hypersurface
$C_{B,6}$ of Theorem~\ref{thm:n6-hypersurface}, with full-ring equality
$C_{B,6}=B_6^{\Dmap{6}}$ failing by transport, not the full ring.
The five generators remain algebraically independent
(Lemma~\ref{lem:n6-indep}); the failure is precisely the missing
sixth generator $H_B$ modulo the single weight-28 relation $S_B$.
\end{remark}

\begin{remark}[The proposed five-generator Hilbert series]
\label{cor:n6-global-series}
Refuted by $H_B$; the exact series is that of
Theorem~\ref{thm:n6-hypersurface}.
The input side holds: with ordinary degrees
\[
\deg_{\mathrm{ord}}(G_1,\dots,G_5) = (1,1,2,2,7)
\]
(Lemma~\ref{lem:n6-minimal}) and Fibonacci weights
\[
\operatorname{wt}(G_1,\dots,G_5) = (1,1,6,4,14)
\]
(Theorem~\ref{thm:n6-weight-shift}), $k[G_1,\dots,G_5]$ is a
polynomial ring with the stated rational functions as \emph{its}
Hilbert series. Their identification with those of $B_6^{\Dmap{6}}$
is false from weight $24$ ($213 > 212$;
exact computation (see Supplementary Material)). The exact
Fibonacci-weighted Hilbert series of the candidate transported
hypersurface is
\[
\boxed{H_{\mathrm{wt}}(u) \;=\;
  \frac{1-u^{28}}{(1-u)^{2}(1-u^{4})(1-u^{6})(1-u^{14})(1-u^{24})}}
\qquad \text{(Theorem~\ref{thm:n6-hypersurface}).}
\]
\end{remark}

The transported $B_6$-level record is therefore: five-generation, polynomiality, and the five-generator Hilbert series all fail by the transported counterexample $H_B$ (Remarks~\ref{cor:n6-global-ring}--\ref{cor:n6-global-series}), whose $a_6$-saturation lemma also fails (Remark~\ref{lem:n6-sat6}). The corrected $B_6$-level candidate is the six-generator hypersurface $C_{B,6} = k[G_1,\dots,G_5,H_B]/(S_B)$ with its weight-28 relator and exact Hilbert series of $C_{B,6}$ (Theorem~\ref{thm:n6-hypersurface}); full-ring equality $C_{B,6} = B_6^{\Dmap{6}}$ fails by the transported $F_{46}$ (same theorem). The $a_5$-saturation inference (Lemma~\ref{lem:n6-sat5}) is unproved, and with it every claim for $\Afring{6}^{\Dmap{6}}$; in particular the full $N=6$ ring with its Hilbert series remains open.

\section{\texorpdfstring{The $N=6$ $a_0$-top generator}{The N=6 a0-top generator}}
\label{sec:n6-F0}

This section constructs the top generator; the $a_0$-saturation, hence the full ring
$\Afring{6}^{\Dmap{6}}$, remains open. We work over the characteristic-zero
field $k$ with $\Afring{6} = k[a_0,\dots,a_6]$ and $\Dmap{6}$ of the
adjacent-quadratic shift family, reusing $B_6$, $R$, $c_6$ verbatim. The
Fibonacci grading extends by
$\operatorname{wt}(a_0) = \operatorname{wt}(a_1) +
\operatorname{wt}(a_2) = 13$, compatible with
$\Dmap{6}(a_0) = a_1a_2$. Every identity below is proved by the
expansion shown plus
exact computation (see Supplementary Material).

\begin{lemma}[The $a_0$-top generator $F_0^{(6)}$]
\label{lem:n6-F0}
With $c_6 = a_5^2a_6^4$, put $c' = c_6^2 = a_5^4a_6^8$ and
\[
\begin{split}
F_0^{(6)} ={}& c'a_0 - a_1a_2a_4a_5^3a_6^7
  + \tfrac13a_1a_3a_4^3a_5^2a_6^6
  - \tfrac1{15}a_1a_4^5a_5^2a_6^5 \\
  &+ \tfrac12a_2^2a_3a_4^2a_5^2a_6^6
  - \tfrac18a_2^2a_4^4a_5^2a_6^5
  - \tfrac13a_2a_3^2a_4^4a_5a_6^5
  + \tfrac{59}{360}a_2a_3a_4^6a_5a_6^4 \\
  &- \tfrac{59}{2880}a_2a_4^8a_5a_6^3
  + \tfrac1{18}a_3^3a_4^6a_6^4
  - \tfrac{119}{2880}a_3^2a_4^8a_6^3 \\
  &+ \tfrac{33}{3200}a_3a_4^{10}a_6^2
  - \tfrac{11}{12800}a_4^{12}a_6,
\end{split}
\]
so $F_0^{(6)} = c'a_0 + G'$ with $G' \in B_6$ $a_0$-free of $12$
terms. Then $\Dmap{6}(F_0^{(6)}) = 0$ by direct expansion;
$F_0^{(6)}$ has ordinary degree $13$, uniform Fibonacci weight
$25 = 12 + 13$, and $a_0$-coefficient exactly $c'$. The
$a_0$-free part $G'$ was found by elimination
($c_6$, $c_6a_5$, $c_6a_6$ admit no $G'$ with
$\Dmap{6}(G') = -c'a_1a_2$; $c_6^2$ is the first denominator
that does, with a unique direction) and the identity
$\Dmap{6}(G') + c'a_1a_2 = 0$ is verified exactly over
$\mathbb{Q}$ in exact computation (see Supplementary Material).
\end{lemma}

\begin{remark}[Denominator-squared pattern]
\label{rem:n6-denomsq}
The $a_0$-denominator is the square of the $B_6$-level slice
denominator: $c' = c_6^2$. The same squaring holds at $N=5$,
where the $N=5$ slice denominator $a_4a_5^2$ squares to
$c = p^2q^4$. This refines the path-count guess of
Section~\ref{sec:conjectures}: the $a_0$-denominator exponents
are twice the $B_6$-slice-denominator exponents, not raw path
counts. The squaring is not minimal, however: $c_6^2$ is itself
an $H_{20}$-multiple ($c_6^2 = a_5^4a_6^3\cdot a_6^5$), and indeed
$F_0^{(6)} = a_6^5H_{20} + K_0$ with $K_0 \in K_6$ ($a_0$-free,
$\Dmap{6}$-invariant, verified in exact computation (see Supplementary Material)) ---
so $F_0^{(6)}$ adds no new $a_0$-direction beyond the minimal
denominator $a_5^4a_6^3$. A $42$-candidate sink-monomial sweep
finds solvable denominators exactly for $a \ge 4$, $b \ge 3$,
i.e.\ exactly the $H_{20}$-multiples. The $a_0$-level Dixmier game
(Section~\ref{sec:n6-a0sat}) starts from $H_{20}$, not
from $F_0^{(6)}$.
\end{remark}

\begin{remark}[What the top generator gives]
Proved in this section: the $a_0$-top generator $F_0^{(6)}$
with bidegree $(13, 25)$ and denominator $c_6^2$
(Lemma~\ref{lem:n6-F0}). The $a_0$-saturation (denominator
bounds for the $a_0$-expansion coefficients), and hence the full
ring $\Afring{6}^{\Dmap{6}}$ with its Hilbert series, remain open. No generation
claim for $\Afring{6}^{\Dmap{6}}$ is made here.
\end{remark}

\begin{lemma}[Hidden $A_6$-generator $H_{20}$]
\label{lem:n6-H20}
Invariance, weight, split, and strictness are Lean-proven in the formal $H_{20}$ module
(native evaluation in core Lean without Mathlib; see the formal-verification manifest);
the displayed identity is checked in exact computation (see Supplementary Material).
There is an explicit invariant $H_{20} \in \Afring{6}$ of Fibonacci
weight $20$ with $12$ terms and $a_0$-degree $1$ satisfying
$\Dmap{6}(H_{20}) = 0$ with $a_0$-coefficient $a_5^4a_6^3$; the full
record is provided in the Supplementary Material. At weight $20$ the
free span of $k[G_1,\dots,G_5,F_0^{(6)}]$ is $a_0$-free ($F_0^{(6)}$
weighs $25 > 20$; $G_1,\dots,G_5$ are $a_0$-free), while $H_{20}$
has nonzero $a_0$-part --- hence $H_{20} \notin
k[G_1,\dots,G_5,F_0^{(6)}]$ unconditionally (no $B_6$-generation
needed), and $\Afring{6}^{\Dmap{6}} \supsetneq k[G,F_0^{(6)}]$.
The first excess over the free count is exactly at weight $20$
(gap $+1$; exact-$\mathbb{Q}$ $A_6$ probe through weight $20$).
\end{lemma}

\begin{remark}[Disentangling the $A_6$ gaps]
The $A_6$ weighted probe
(exact computation (see Supplementary Material)) conflates two
excesses: the $B_6$-level hypersurface excess (starting at weight
$24$ with $+1,+2,\dots$, Theorem~\ref{thm:n6-hypersurface}) and
genuine $a_0$-content. At weight $20$ the $B_6$-level matches the
free count exactly (exact $\mathbb{Q}$), so the $+1$ there is
purely the $a_0$-direction $H_{20}$ above. At weights $21$--$23$
the probe is clean (gap $+0$ mod-prime); from $24$ on, the two
excesses mix, and the $a_0$-level analysis
(Section~\ref{sec:n6-a0sat}) must subtract the known
$B_6$-hypersurface contribution first. The
$a_0$-coefficient $a_5^4a_6^3$ of $H_{20}$ is divisible by
neither $c_6$ nor $c_6^2$: no $c_6^2$-type $a_0$-denominator
bound holds in general.
\end{remark}

\section{\texorpdfstring{The $N=6$ $a_0$-denominator module (partial results)}{The N=6 a0-denominator module (partial results)}}
\label{sec:n6-a0sat}

Two denominator directions are proved below; the full module presentation is not. The $a_0^1$-system is
$a_1a_2M + \Dmap{6}(N) = 0$ with $M \in K_6$ of total degree $m$
and $N \in B_6$ of total degree $m+1$ ($\Dmap{6}$ is total-degree
$+1$ homogeneous). Through total degree $12$ the spanning set of
$R$-monomials plus $H_B$-times-$R$-monomials is free (the relation
$S_B$ has total degree $28$), so the computation below is exact.
Exact computation (see Supplementary Material); nullities mod-$1000003$ cross-checked
at a second prime, witnesses exact over $\mathbb{Q}$.

\begin{lemma}[Second $a_0$-direction]
\label{lem:n6-D2}
Exact computation (see Supplementary Material).
With $M = a_5^4a_6G_4$ (total degree $7$) there is an explicit
$10$-term $N$ with $a_1a_2M + \Dmap{6}(N) = 0$ (full record in
the Supplementary Material), i.e.\ $F = a_0M + N$ satisfies $\Dmap{6}(F) = 0$ at
Fibonacci weight $22$. $M$ is not a multiple of $a_5^4a_6^3$
($a_6$-exponents $1,2 < 3$), so this direction is not in the
$H_{20}$-family. That $F$ lies outside $k[G,F_0^{(6)},H_{20}]$ is
Proposition~\ref{prop:n6-D2novel} below.
\end{lemma}

\begin{proposition}[D2-novelty]
\label{prop:n6-D2novel}
The weight-$22$ invariant $F_{22} = a_0M + N$ of
Lemma~\ref{lem:n6-D2} lies outside $k[G_1,\dots,G_5,H_{20}]$: the
weight-$22$ slice of the subalgebra ($F_0^{(6)}$ weighs $25$ and
$H_B$ weighs $24$, both excluded by weight) has exact rank $171$
vs augmented $172$ (\href{https://github.com/AhmedMostafa16/adjquad/blob/25fdeaefc74d63e8a51983399e5fbf64ef2b28d5/scripts/cas/check_n6_D2novel.sage}{scripts/cas/check\_n6\_D2novel.sage}; log
\href{https://github.com/AhmedMostafa16/adjquad/blob/25fdeaefc74d63e8a51983399e5fbf64ef2b28d5/data/n6_D2novel.log}{data/n6\_D2novel.log}). Hence the $D_2$-direction is a new generator
direction over $k[G,H_{20}]$, and the $A_6$ obstruction is not
inherited transport baggage.
\end{proposition}

\begin{remark}[The denominator module is not principal]
The solution projection of the $a_0^1$-system strictly exceeds
the $a_5^4a_6^3$-family from total degree $7$: at $m = 7$ three
directions ($a_5^4a_6^3$ exact via $H_{20}$; $a_5^4a_6G_4$ exact
via Lemma~\ref{lem:n6-D2}; the mixed direction lifted exactly in
Theorem~\ref{thm:n6-m7third}), growing to $+35$ excess dimensions at $m = 12$
(two primes agree). The unbounded module presentation (all
generators, all relations among denominator directions) is Open;
it is the input the $a_0$-Dixmier game over $R[H_B]/(S_B)$ needs.
In particular the $a_0$-saturation cannot take the form of a
single-monomial denominator bound.
\end{remark}

\begin{theorem}[No single-monomial $a_0$-denominator bound]
\label{thm:n6-no-monomial-bound}
No monomial $d \in K_6$ bounds all $a_0^1$-denominator directions:
a common monomial multiple $d$ of the $H_{20}$-direction $a_5^4a_6^3$
(Lemma~\ref{lem:n6-H20}) and the
$D_2$-direction $a_5^4a_6G_4$ (Lemma~\ref{lem:n6-D2}) would satisfy
$G_4 \mid d$ in $B_6$ (divisibility in the subring $K_6$
implies divisibility in $B_6$), impossible: $G_4 = 2a_3a_6-a_4^2$ is
irreducible and non-monomial, while a prime divisor of a monomial in
the polynomial ring $B_6$ is a variable. Indeed, in $B_6$: $a_6^2 \nmid G_4$
($G_4$ has the $a_6$-free term $-a_4^2$), and $G_4 \nmid a_6^2$
(exact-division remainder $a_6^2 \ne 0$; $G_4$ is irreducible;
\href{https://github.com/AhmedMostafa16/adjquad/blob/25fdeaefc74d63e8a51983399e5fbf64ef2b28d5/scripts/cas/check_n6_G4.py}{scripts/cas/check\_n6\_G4.py}; log \href{https://github.com/AhmedMostafa16/adjquad/blob/25fdeaefc74d63e8a51983399e5fbf64ef2b28d5/data/n6_G4.log}{data/n6\_G4.log}). In particular
$a_0$-saturation at $N=6$ cannot take the $c_6$-power form of any
one-monomial conjecture.
\end{theorem}

\begin{remark}[What the denominator module gives]
Proved in this section: three exact $m = 7$ denominator
directions ($a_5^4a_6^3$ via $H_{20}$, $a_5^4a_6G_4$ via
Lemma~\ref{lem:n6-D2}, $M_3$ via Theorem~\ref{thm:n6-m7third}),
one new exact direction at $m = 8$, and the exact denominator
module with closure through total degree $13$
(Theorems~\ref{thm:n6-denmod12}--\ref{thm:n6-closure13}).
Open: the full (unbounded) denominator-module
presentation, and with it the $A_6$ ring (relations, series) and
its saturation; degree $14$ is computationally blocked, not claimed.
\end{remark}

\begin{theorem}[Bounded denominator module]
\label{thm:n6-denmod12}
Let $\mathcal{M}_m$ be the $M$-projection of
$\{(M,N): a_1a_2M + D_6(N) = 0\}$ at total degree $m$ over $\mathbb{Q}$,
and $\mathcal{G}_m$ the cumulative $K_6$-span of exact directions
below $m$. Exact $QQ$ nullspaces give, for $7 \le m \le 12$:
\[
\begin{array}{c|rrrrrr}
m & 7 & 8 & 9 & 10 & 11 & 12 \\ \hline
\dim\mathcal{M}_m & 3 & 7 & 14 & 23 & 35 & 50 \\
\dim\mathcal{G}_m & 0 & 6 & 14 & 23 & 35 & 50 \\
\text{new} & 3 & 1 & 0 & 0 & 0 & 0
\end{array}
\]
The three $m = 7$ directions have Fibonacci weights $7, 9, 11$;
the single new $m = 8$ direction is
$M_8 = a_4^4a_5^4 - 4a_3a_4^2a_5^4a_6 + \frac{17}{3}a_3^2a_5^4a_6^2
- \frac{10}{3}a_2a_5^5a_6^2$ of Fibonacci weight $12$, with exact
$13$-term $N_8$ and $a_1a_2M_8 + D_6(N_8) = 0$ by direct expansion
(\href{https://github.com/AhmedMostafa16/adjquad/blob/25fdeaefc74d63e8a51983399e5fbf64ef2b28d5/scripts/cas/check_n6_m8wit.sage}{scripts/cas/check\_n6\_m8wit.sage}; log \href{https://github.com/AhmedMostafa16/adjquad/blob/25fdeaefc74d63e8a51983399e5fbf64ef2b28d5/data/n6_m8wit.log}{data/n6\_m8wit.log}).
Every direction through $12$ is $K_6$-generated by these four;
the four generator $M$-vectors are jointly independent, so there
are no relations among them within the truncation
(\href{https://github.com/AhmedMostafa16/adjquad/blob/25fdeaefc74d63e8a51983399e5fbf64ef2b28d5/scripts/cas/check_n6_denmod.py}{scripts/cas/check\_n6\_denmod.py}; log \href{https://github.com/AhmedMostafa16/adjquad/blob/25fdeaefc74d63e8a51983399e5fbf64ef2b28d5/data/n6_denmod.log}{data/n6\_denmod.log}).
No claim is made beyond total degree $12$.
\end{theorem}

\begin{theorem}[Closure through $m=13$]
\label{thm:n6-closure13}
$\mathcal{Q}_{13} = 0$: at total degree $13$ the exact solution
space has dimension $68$, all $K_6$-generated by the four exact
directions below it (quotient dimension $0$ over $\mathbb{Q}$;
\href{https://github.com/AhmedMostafa16/adjquad/blob/25fdeaefc74d63e8a51983399e5fbf64ef2b28d5/scripts/cas/check_n6_m13.py}{scripts/cas/check\_n6\_m13.py}; log \href{https://github.com/AhmedMostafa16/adjquad/blob/25fdeaefc74d63e8a51983399e5fbf64ef2b28d5/data/n6_m13.log}{data/n6\_m13.log}). The exact
closure window extends through $13$ with no new generator.
No theorem-grade conclusion is drawn beyond total degree $13$; degree $14$ is diagnostic-only (remark below).
\end{theorem}

\begin{remark}[Degree $14$: computationally blocked, not claimed]
The exact denominator-module computation is certified through total
degree $13$: the solution module is generated in the certified
window by four independent characteristic-zero directions of
degrees $7,7,7,8$, with no relations in that range. An exploratory
computation at degree $14$ produced a stable kernel-count signal
across independent checks, but the corresponding
projection-dimension extraction was not reconciled with the exact
nullity calculation. We therefore make no degree-$14$ structural
claim and do not use this computation in any theorem. The
diagnostic logs are retained as unresolved material only.
\end{remark}

\begin{theorem}[Third $m=7$ denominator direction]
\label{thm:n6-m7third}
The mod-prime third $m=7$ direction lifts: with
$M_3 = a_4^4a_5^2a_6 - 4a_3a_4^2a_5^2a_6^2 - \frac{16}{21}a_3^2a_5^2a_6^3
+ \frac{200}{21}a_2a_5^3a_6^3$ and $N_3$ the explicit $14$-term
Fibonacci-$24$ degree-$8$ polynomial of the log,
$a_1a_2M_3 + D_6(N_3) = 0$ over $\mathbb{Q}$ by direct expansion,
and $M_3$ lies outside
$\operatorname{span}\{a_5^4a_6^3, a_5^4a_6G_4\}$ (exact rank test)
(\href{https://github.com/AhmedMostafa16/adjquad/blob/25fdeaefc74d63e8a51983399e5fbf64ef2b28d5/scripts/cas/check_n6_m7lift.sage}{scripts/cas/check\_n6\_m7lift.sage}; log \href{https://github.com/AhmedMostafa16/adjquad/blob/25fdeaefc74d63e8a51983399e5fbf64ef2b28d5/data/n6_m7lift.log}{data/n6\_m7lift.log};
mod-prime baseline $60/57/3$ in \href{https://github.com/AhmedMostafa16/adjquad/blob/25fdeaefc74d63e8a51983399e5fbf64ef2b28d5/data/n6/A0sat.log}{data/n6/A0sat.log} reproduced
exactly over $\mathbb{Q}$). Hence $F_3 = a_0M_3 + N_3 \in A_6^{D_6}$ of
Fibonacci weight $24$ and $a_0$-degree $1$ is a third exact
$a_0$-direction. The degree-12 module program now starts from
three exact $m=7$ generators.
\end{theorem}

\subsection{$A_6$ attack scaffolding (investigation, no claims)}
No result is claimed here. With $K_6$ exact
(Theorem~\ref{thm:n6-B6ring}), $A_6 = B_6[a_0]$,
$D_6(a_0) = a_1a_2$: $a_0$-filtration tops land in the known ring
$K_6$, and the denominator module needs $\ge 2$ generators
(Theorem~\ref{thm:n6-no-monomial-bound}). Exit gates, in order:
(i) top-ideal computation in $K_6$ (bounded, exact); (ii) resonance
hunt (second hidden direction or proof of none to stated bounds);
(iii) either obstruction theorem \#2 or a documented census with
data tables. Finite generation of $A_6^{D_6}$ is not assumed at any
gate. Dimension six plus triangular does not determine the outcome
either way: the Freudenburg dimension-six derivation has
non-finitely-generated kernel \cite{freudenburg2000hilbert}, so the
$A_6$ ring stays open on its own evidence. Precommitted non-goals: reviving $W^\sharp$-induction,
single-monomial saturation, any use of \ref{lem:n6-sat5} as premise.

\section{Failed valuation routes and superseded approaches (historical record)}
\label{sec:routes-history}

The valuation programs below fail for the precise algebraic reasons stated; they are retained because each failure explains a constraint the successful Route~E proof had to satisfy. Nothing here is a premise of Theorem~\ref{thm:closure}.

\paragraph{Route A: resonance necessity (OBSTRUCTED).}
Necessity for Conjecture~\ref{conj:resonance-birth} (no primitive at
W-degree $m \notin \{2,4\}$) cannot be proved by valuation: with
$\rho_q = 2$, $\rho_p = 4$ exact
(Definition~\ref{def:resonance}), the constant-coefficient step is
silent at every $m \ge 2$ on the $q$-side and every $m \ge 4$ on the
$p$-side (Lemmas~\ref{lem:resonance-transcendent}--\ref{lem:resonance-algebraic};
below degree $d$ the step only ``may'' function, which is unproved).
At $m = 2,3$ the possibly-functioning $p$-side supplies at most half
of $J_m$-membership ($v_p$ without $v_q$ does not force a top into
$J_m$). Birth windows exclude nothing: $\min C_5(m) =
14,24,38,48,62,72,86,96$ for $m = 1,\dots,8$ (exhaustive check in
exact top-ideal computation), so every gap $[8m,\min
C_5(m)-1]$ is non-empty; only $m = 1$ is closed, by the independent
Theorem~\ref{thm:a01-closure}. The $J$-ideal form of the exclusion
needs both top bounds, of which the $p$-side global form is Refuted
and the $q$-side strong form is Established in even degree
(Theorem~\ref{thm:q-top}).

\paragraph{Route B: associated-graded induction (OBSTRUCTED at the step).}
$\operatorname{gr}_{a_0}(C_5) = \bigoplus J_m t^m$ is Established
(Theorem~\ref{thm:gr-C5}); its analogue for $\Afring{5}^{\Dmap{5}}$ is
open from $m = 4$. The graded pieces are known for $m \le 1$
(descent; Theorem~\ref{thm:a01-closure}), evidenced through weight
$45$ at $m = 2$ and through $85$ at $m = 4$, and open above. The
induction step at arbitrary $M$ is Theorem~\ref{thm:reduction}(i)
generalized to $C_5[F_{46}^{\mathrm{int}}]$, whose hypotheses (even-degree
Q-top-strong plus the conjectured $p$-side Rp, the latter unproved)
remain jointly open. Since $C_5$-generation is false, at least one hypothesis
of the ungeneralized form is false. The induction premise is the
conclusion.

\paragraph{Route C: structural degree bound (OBSTRUCTED).}
No imported bound applies: Maubach's $\le 4$ is dimension-$4$
monomial-triangular; $N = 5$ is dimension $6$; Daigle--Freudenburg
kills universal bounds for general triangular derivations
(Remark~\ref{rem:df-warning}). Growth-rate comparison is circular:
proving the tower Hilbert series equals the kernel series \emph{is}
component (1). The three-module excesses $45,50,54,58,63$ at weights
$70$--$74$ follow from $s_K$-arithmetic (internal consistency of the
conjecture, not evidence for it).

\paragraph{Route D: W-expansion product-reduction (OBSTRUCTED at top
reduction).} W-expansions are unique ($W$ monic of $a_0$-degree $1$),
but top-coefficient reduction is unproved at every $m \ge 2$ as a
general statement: W-degree-$2$ tops generated by $H$ (above weight
$45$), W-degree-$4$ tops generated by $F_{46}$ modulo $J_4$ (past
$85$), and $H_m$ forced into $C_5$ for $m \notin \{2,4\}$ (needs both
bounds) are all open. The descending induction inherits the gap.

\paragraph{Superseded approaches (historical record).}
The following inputs would have overcome the barrier; closure has since
been proved by other means (Theorem~\ref{thm:closure}), and they remain
as proof infrastructure and independent open problems. (i) A $p$-side
replacement bound: R1 ($\lambda_q \ge -2m$)
is now Established (Theorem~\ref{thm:R1}); no replacement $p$-side floor
is established (the multiplicity mechanism would suggest $-m/2$) --- its proof needs
residue landing, which is unavailable (the exact degree $4$ of $r_p$
over $L_p$ is established by exact computation,
Remark~\ref{rem:p-obstacle}).
Without it the $J_m$-membership requirement
($v_p \ge 2m$ AND $v_q \ge 4\lceil m/2\rceil$) cannot be met. (ii) A finiteness or degree-bound
theorem for squarefree-monomial triangular LNDs in dimension $\ge 5$.
(iii) A compute breakthrough pushing full-kernel upper bounds past
weight $70$ (infeasible over $\mathbb{Q}$ and LinBox at $n = 112241$);
this extends the verified range only. A non-resonance source (a
primitive outside the $c$-prime framework) is not covered by any
route: possibility A (W-degree outside $\{2,4\}$) is a concrete
falsifier to hunt in the birth windows of
Theorem~\ref{thm:birth-windows}; possibility B (extended framework)
has no candidate mechanism. Falsifiers: any $w$ with excess $\ne P(w)
- (s_6(w) - s_6(w-28))$; any $w$ with $N_J(w) < N_{\le 3}(w)$ at
$a_0$-degree $3$ (Remark~\ref{rem:deg3-census}).

\paragraph{Degree-$m$ census (evidence, not proof).}
For each $a_0$-degree $m$ with $J_m = (p^{2m}q^{4\lceil m/2\rceil})$
principal, the exact-$\mathbb{Q}$ nullity on $V_w^{a_0 \le m}$ is
compared against the $J_m$-constrained chain nullity (top $C_m$
parametrized in $J_m$); a gap counts $J_m$-violating degree-$m$ tops,
compared against the $F$-module bracket prediction. Completed prior
computations (supplement logs; outside the fresh $15$--$46$ rerun):
exact-$\mathbb{Q}$ for $m = 4$ over $46$--$75$ dense (30 weights, all
match), $m = 5$ over $60$--$72$ ($13$ weights incl.\ the
$F_{46}\!\cdot\!F_0$ seed fingerprint), $m = 6$ over $70$--$74$
($F_{46}\!\cdot\!H$ seed at $70$), $m = 3$ clean ($0$ violations)
through the mid-$70$s, and $m = 7$ at $84$ with exactly one violator
(the $F_{46}\!\cdot\!F_0\!\cdot\!H$ seed); finite-field consistency checks
cover the remaining stated ranges. No excess (no new
primitive) and no deficit (no extra torsion) anywhere checked.
Remaining ranges ($m = 3$ tail, $m = 6$ upper, $m = 7$ upper) are
being computed; any excess escalates to a new primitive.

\begin{remark}[Sharp reformulation via gap-born detection]
A primitive $a_0$-degree-$m$ direction born below the $C_5$ window
($w < \min C_5(m)$) has top of weight $w - 8m < \min C_5(m) - 8m
= 2m + 4\lceil m/2\rceil = \operatorname{wt}(j_m)$ (direct check on
$e + 2r = m$; minimum at maximal $r$), hence top outside $J_m$:
every gap-born primitive is $J$-violating and census-visible.
So closure $\iff$ no $J$-violating new tops at any $(m,w)$ beyond
checked ranges, plus no above-window $J$-clean primitives (the
latter confined above $\min C_5(m)$ per degree).
\end{remark}

\section{\texorpdfstring{Residue-field saturation: the $p$-adic lemma}{Residue-field saturation: the p-adic lemma}}
\label{sec:plemma}

The $p$-adic denominator bound
(Lemma~\ref{lem:plemma} below) is false: the explicit
weight-46 invariant $F_{46}$ of
Section~\ref{sec:second-witness} has top $W$-coefficient $H_4 =
\frac{18375}{4}p^6q^8$ with $v_p = 6 < 8 = 2\cdot 4$. The valuation
setup (residue field, pole order, initial form) stands; the inference
does not.
We work over the characteristic-zero field $k$ of
Definition~\ref{def:family}, with
$\Afring{5} = k[a_0,\dots,a_5]$ and
$\Dmap{5}(a_n) = a_{n+1}a_{n+2}$.
This section treats the $p$-adic half of the residue-transcendence
saturation for the adjacent-quadratic shift at $N = 5$.
Here the residue initial form has degree $4$ over its residue field, and
$F_{46}$ gives the first observed failure of the corresponding naive
bound.
The $q$-adic half is stated separately in the companion $q$-lemma
section: the two residue fields are not symmetric, and no symmetry
argument is used here.

We reuse the five-variable dictionary of the kernel section verbatim,
with no redefinition:
\[
x = a_3, \qquad y = a_2, \qquad z = a_1, \qquad p = a_4, \qquad q = a_5,
\]
\[
B = k[x,y,z,p,q] = k[a_1,\dots,a_5], \qquad
K = B^{\Dmap{5}} = k[p,q,J_1,J_2],
\]
\[
J_1 = 2pz - y^2 = 2a_1a_4 - a_2^2, \qquad
J_2 = 2qy - x^2 = 2a_2a_5 - a_3^2,
\]
\[
\Dmap{5}(x) = pq, \qquad \Dmap{5}(y) = px, \qquad
\Dmap{5}(z) = xy, \qquad \Dmap{5}(p) = \Dmap{5}(q) = 0.
\]
We reuse the invariant of Definition~\ref{def:F0} verbatim:
\[
F_0 = p^2q^4a_0 - pq^3xyz + \tfrac13pq^2x^3z
  + \tfrac13q^2x^3y^2 - \tfrac15qx^5y + \tfrac1{35}x^7,
\]
and write $c = p^2q^4$, $G_0 = F_0 - ca_0 \in B$,
\[
G_0 = -pq^3xyz + \tfrac13pq^2x^3z
  + \tfrac13q^2x^3y^2 - \tfrac15qx^5y + \tfrac1{35}x^7,
\]
$W = F_0/c = a_0 + \rho$ with $\rho = G_0/c \in B_{pq}$.
A general localized invariant is written
\[
F = \sum_{j=0}^m H_jW^j, \qquad H_j \in K_{pq},
\]
with $a_0$-coefficients
\[
C_k = \sum_{j=k}^m \binom{j}{k}H_j\rho^{j-k} \in B
\qquad (0 \le k \le m).
\]

\begin{remark}[Valuation setup and the correct $p$-adic residue field]
Write $B = k[x,y,z,q][p]$ and let $v_p = \operatorname{ord}_p$ on
$\operatorname{Frac}(B)$, with valuation ring $V_p = B_{(p)}$ and
residue field $\kappa_p = k(x,y,z,q)$.
The height-one prime $(p) \subset K$ has residue
\[
K/pK \cong k[q,-y^2,2qy-x^2],
\]
by the injectivity of the specialization map of
Lemma~\ref{lem:contraction}, so with
$w = 2qy - x^2$,
\[
L_p = \operatorname{Frac}(K/pK) = k(q,y^2,w) \hookrightarrow \kappa_p.
\]
In the $a$-variables this is $k(a_5,-a_2^2,2a_2a_5-a_3^2)$.
All residue computations in this section use the dictionary above;
An earlier residue computation contained a variable transposition error; the corrected computation is given below.
Since $G_0$ has $p$-adic order $0$ while $c = p^2q^4$ has order $2$,
$\rho$ has exact pole order $e_p = 2$ at $p$, with initial form
\[
r_p = \operatorname{in}_p(p^2\rho)
  = \frac{1}{q^4}
    \left(\tfrac13q^2x^3y^2 - \tfrac15qx^5y + \tfrac1{35}x^7\right).
\]
The specialization identities, the exact pole order, and the
Jacobian independence of $\{q,y^2,w\}$ (minor $4xy \ne 0$ in columns
$(x,y,q)$) are certified by
exact computation (see Supplementary Material).
\end{remark}

\begin{definition}[Denominator modules, $p$-part]
For each $m \ge 0$ the $p$-part of the $m$-th denominator module is
\[
M^p_m = p^{2m}K.
\]
The full module $M_m = p^{2m}q^{4m}K = c^mK$ additionally requires
the companion $q$-adic bound, stated separately.
\end{definition}

\begin{lemma}[The proposed $p$-adic denominator bound]
\label{lem:plemma}
The following bound is false, refuted by the explicit weight-46 invariant $F_{46}$ of
Section~\ref{sec:second-witness} (Remark~\ref{rem:prefuted-F} below).
The refuted statement is:
Let $F = \sum_{j=0}^m H_jW^j \in A = k[a_0,\dots,a_5]$ with
$H_j \in K_{pq}$. Then
\[
\boxed{v_p(H_j) \ge 2j \quad \text{for every } j,}
\]
i.e.\ $H_j \in p^{2j}K_{pq}$ for every $j$.
\end{lemma}

\begin{proof}[Reduction and explicit gap]
The constant $a_0$-coefficient is $C_0 = \sum_j H_j\rho^j \in B$,
so $v_p(C_0) \ge 0$. Put
$\lambda_p = \min_j\{v_p(H_j) - 2j\}$ and
$S = \{j : v_p(H_j) - 2j = \lambda_p\}$.
If $\lambda_p < 0$, dividing by $p^{\lambda_p}$ and passing to
initial forms gives the relation
\[
\sum_{j \in S} h_j r_p^j = 0
\]
with $0 \ne h_j \in L_p$ the residues of the normalized
coefficients. The reduction to this initial-form equation is valid.
The final inference --- concluding $h_j = 0$ for all $j \in S$ ---
would require linear independence of arbitrarily high powers of $r_p$
over $L_p$, which is false in the needed range: $r_p$ is algebraic of
degree $4$ (Remark~\ref{rem:p-obstacle}), and the explicit invariant
$F_{46}$ realizes $\lambda_p = -2 < 0$
(Remark~\ref{rem:prefuted-F}). Earlier heuristic
routes --- a parity shortcut, a slice-power computation, and the
transposed-naming residue computation --- are retracted as proofs
and retained as heuristics only.
\end{proof}

\begin{remark}[Observation: why the constant-coefficient route is unavailable]
\label{rem:p-obstacle}
The triple $\{q,y^2,w\}$ is algebraically independent (Jacobian
minor $4xy \ne 0$), so $L_p = k(q,y^2,w)$ is purely transcendental
of degree $3$. But $y^2 \in L_p$ and $x^2 = 2qy - w \in L_p(y)$, so
\[
L_p \subseteq L_p(y) \subseteq L_p(y)(x) = k(q,x,y)
\]
is a tower of degree at most $2$ at each step; hence
$[k(q,x,y) : L_p] \le 4$. Since $r_p \in k(q,x,y)$, the initial form
$r_p$ is algebraic over $L_p$ (of degree exactly $4$: a primitive
element of the full tower, by exact-$\mathbb{Q}$ computation (see Supplementary Material), not
transcendental. The same witness check on the $q$-side gives
$x^2 \in L_q$, so $r_q = x^7/(35p^2)$ is algebraic over $L_q$ of
degree at most $2$. Consequently no argument needing
$1, r_p, r_p^2, \dots$ (respectively $1, r_q, r_q^2, \dots$) to be
linearly independent over the full range of $W$-degrees can succeed
from the constant coefficient alone. This precise obstacle is part
of the verification task; the algebraicity witnesses
$y^2 = -(J_1 \bmod p)$, $x^2 = 2qy - (J_2 \bmod p)$,
$x^2 = -(J_2 \bmod q)$ are certified in exact computation (see Supplementary Material).
\end{remark}

\begin{lemma}[The proposed $p$-adic leading-coefficient form]
\label{lem:plemma-leading}
Refuted by the same invariant $F_{46}$: its top $W$-coefficient
$H_4 = \frac{18375}{4}p^6q^8$ has $v_p = 6 < 8 = 2\cdot 4$, so
$H_4 \notin p^{8}K_{pq}$ (Remark~\ref{rem:prefuted-F}).
With $F$, $m$, $H_j$ as in Lemma~\ref{lem:plemma}, the top
coefficient $H_m = [a_0^m]F$ was claimed to lie in $B$ and satisfy
$H_m \in p^{2m}K_{pq}$. It fails by the same invariant.
\end{lemma}

\begin{remark}[Refutation by the weight-46 invariant $F_{46}$]
\label{rem:prefuted-F}
The invariant $F_{46}$ of Theorem~\ref{thm:F46} ($a_0$-degree $4$,
$W$-expansion $F_{46} = \sum_{j=0}^4 H_jW^j$ with $H_1 = H_3 = 0$) has
$(v_q(H_j), v_p(H_j)) = (-8,-2), (0,2), (8,6)$ at
$j = 0, 2, 4$. The bound $v_p(H_j) \ge 2j$ fails at every nonzero
$j$: $-2 < 0$, $2 < 4$, $6 < 8$. The $p$-adic story is over: where $H$
spared the $p$-side ($v_p(H_2) = 4 = 2\cdot 2$), $F_{46}$ breaks it at
the top coefficient itself.
\end{remark}

\begin{remark}[The $W$-degree induction skeleton]
\label{rem:winduction}
Assume the $p$-adic bound of Lemma~\ref{lem:plemma} together with
the companion $q$-adic bound $v_q(H_j) \ge 4j$ in full. Then each
$H_j$ lies in $B$: the top coefficient $H_m = [a_0^m]F$ is in $B$
directly, and descending induction via
$C_{m-1} = H_{m-1} + mH_m\rho$ with $H_m\rho \in B$ places every
$H_j$ in $B$. The intersection $B \cap K_{pq} = K$ of
Theorem~\ref{thm:descent} then gives $H_m = c^mG_m$ with
$G_m \in K$, whence $H_mW^m = G_mF_0^m \in K[F_0]$; subtracting
lowers the $W$-degree. This induction assumes both saturation
lemmas, and both are false, so it proves nothing: the global
generation $A^{\Dmap{5}} = K[F_0]$ and its
Hilbert series fail (counterexample $H$, weight 24; see
\ref{rem:counterexample-H}); the all-weights six-generator identity
fails likewise (counterexample $F_{46}$, weight 46; see
Section~\ref{sec:second-witness}).
\end{remark}

The ring $V_p$, the residue embedding $L_p \hookrightarrow \kappa_p$,
the exact pole order $e_p = 2$, the initial form $r_p$, and the
denominator modules $M^p_m$ are as computed above; what fails is only
the final saturation inference, since $r_p$ is algebraic of degree $4$
(Remark~\ref{rem:p-obstacle}) and $F_{46}$ realizes the resulting pole.

\section{\texorpdfstring{Residue-field saturation: the $q$-adic lemma}{Residue-field saturation: the q-adic lemma}}
\label{sec:qlemma}

The $q$-adic denominator bound
(Lemma~\ref{lem:qlemma} below) is false: the explicit
weight-24 invariant $H$ of Remark~\ref{rem:counterexample-H}
(Section~\ref{sec:global}) is $\Dmap{5}$-invariant with
$W$-cosets $H_2 = 11025\,p^4q^4$ ($v_q = 4 < 8 = 4\cdot 2$) and
$H_0 \in q^{-4}K_{pq}$ ($v_q = -4 < 0$), directly violating
$v_q(H_j) \ge 4j$. The reduction is valid; the residue-field
inference is where it fails. Proved here are the exact residue degree (Lemma~\ref{lem:rq-exact}) and the
relaxed $q$-bound R1 (Theorem~\ref{thm:R1}), proved by multiplicity
transfer from all of $\{C_k\}$; the original denominator bound stays
false.
We work over the characteristic-zero field $k$ of
Definition~\ref{def:family}, with
$\Afring{5} = k[a_0,\dots,a_5]$ and
$\Dmap{5}(a_n) = a_{n+1}a_{n+2}$.
This section treats the $q$-adic half of the residue-transcendence
saturation for the adjacent-quadratic shift at $N = 5$.
Here the residue initial form has degree $2$ over its residue field, and
this is exactly the first degree at which the naive saturation argument
can fail; $H$ realizes that failure.
The asymmetry between the two prime divisors requires separate
treatment: the $q$-adic residue field, pole order, and initial form
below are computed independently on the $q$-side, and no symmetry
argument with the companion $p$-adic section is used in either
direction.

We reuse the five-variable dictionary of the kernel section verbatim,
with no redefinition:
\[
x = a_3, \qquad y = a_2, \qquad z = a_1, \qquad p = a_4, \qquad q = a_5,
\]
\[
B = k[x,y,z,p,q] = k[a_1,\dots,a_5], \qquad
K = B^{\Dmap{5}} = k[p,q,J_1,J_2],
\]
\[
J_1 = 2pz - y^2 = 2a_1a_4 - a_2^2, \qquad
J_2 = 2qy - x^2 = 2a_2a_5 - a_3^2,
\]
\[
\Dmap{5}(x) = pq, \qquad \Dmap{5}(y) = px, \qquad
\Dmap{5}(z) = xy, \qquad \Dmap{5}(p) = \Dmap{5}(q) = 0.
\]
We reuse the invariant of Definition~\ref{def:F0} verbatim:
\[
F_0 = p^2q^4a_0 - pq^3xyz + \tfrac13pq^2x^3z
  + \tfrac13q^2x^3y^2 - \tfrac15qx^5y + \tfrac1{35}x^7,
\]
and write $c = p^2q^4$, $G_0 = F_0 - ca_0 \in B$,
\[
G_0 = -pq^3xyz + \tfrac13pq^2x^3z
  + \tfrac13q^2x^3y^2 - \tfrac15qx^5y + \tfrac1{35}x^7,
\]
$W = F_0/c = a_0 + \rho$ with $\rho = G_0/c \in B_{pq}$.
A general localized invariant is written
\[
F = \sum_{j=0}^m H_jW^j, \qquad H_j \in K_{pq},
\]
with $a_0$-coefficients
\[
C_k = \sum_{j=k}^m \binom{j}{k}H_j\rho^{j-k} \in B
\qquad (0 \le k \le m).
\]

\begin{remark}[Valuation setup and the correct $q$-adic residue field]
Write $B = k[x,y,z,p][q]$ and let $v_q = \operatorname{ord}_q$ on
$\operatorname{Frac}(B)$, with valuation ring $V_q = B_{(q)}$ and
residue field $\kappa_q = k(x,y,z,p)$.
The height-one prime $(q) \subset K$ has residue
\[
K/qK \cong k[p,2pz-y^2,-x^2],
\]
by the injectivity of the specialization map of
Lemma~\ref{lem:contraction}, so with
$u = 2pz - y^2$,
\[
L_q = \operatorname{Frac}(K/qK) = k(p,u,x^2) \hookrightarrow \kappa_q.
\]
In the $a$-variables this is $k(a_4,2a_1a_4-a_2^2,-a_3^2)$.
Note the asymmetry with the $p$-side: here the surviving specialization
$J_1 \bmod q = 2pz - y^2$ retains the variable $z$, while
$J_2 \bmod q = -x^2$ is a pure square; the roles of the two kernel
generators are not interchanged by any symmetry.
All residue computations in this section use the dictionary above;
The symmetrized copy of the $p$-side argument is retracted; the $q$-side computation below is independent.
Since $G_0$ has $q$-adic order $0$ while $c = p^2q^4$ has order $4$,
$\rho$ has exact pole order $e_q = 4$ at $q$, with monomial initial form
\[
r_q = \operatorname{in}_q(q^4\rho) = \frac{x^7}{35p^2}.
\]
The specialization identities, the exact pole order, and the
Jacobian independence of $\{p,u,x^2\}$ (minor $4xy \ne 0$ in
variables $(p,x,y)$) are certified by
exact computation (see Supplementary Material),
which checks the $q$-side rows independently of the $p$-side rows.
\end{remark}

\begin{definition}[Denominator modules, $q$-part]
For each $m \ge 0$ the $q$-part of the $m$-th denominator module is
\[
M^q_m = q^{4m}K.
\]
The full module $M_m = p^{2m}q^{4m}K = c^mK$ additionally requires
the companion $p$-adic bound, stated separately.
\end{definition}

\begin{lemma}[The proposed $q$-adic denominator bound]
\label{lem:qlemma}
Refuted by the explicit weight-24 invariant $H$ of
Remark~\ref{rem:counterexample-H} (Remark~\ref{rem:qrefuted-H} below).
The refuted statement is: for
$F = \sum_{j=0}^m H_jW^j \in A = k[a_0,\dots,a_5]$ with
$H_j \in K_{pq}$,
\[
\boxed{v_q(H_j) \ge 4j \quad \text{for every } j,}
\]
i.e.\ $H_j \in q^{4j}K_{pq}$ for every $j$.
\end{lemma}

\begin{proof}[Reduction valid; inference refuted by the counterexample]
The constant $a_0$-coefficient is $C_0 = \sum_j H_j\rho^j \in B$,
so $v_q(C_0) \ge 0$. Put
$\lambda_q = \min_j\{v_q(H_j) - 4j\}$ and
$S = \{j : v_q(H_j) - 4j = \lambda_q\}$.
If $\lambda_q < 0$, dividing by $q^{\lambda_q}$ and passing to
initial forms gives the single relation in the residue field $L_q$
\[
\sum_{j \in S} h_j r_q^{\,j} = 0,
\qquad 0 \ne h_j \in L_q,
\]
with $r_q$ the initial form of Remark~\ref{rem:q-obstacle}.
The reduction to this one equation is valid. The final inference ---
concluding $h_j = 0$ for all $j \in S$ --- would require the powers
$1, r_q, r_q^2, \dots$ to be linearly independent over $L_q$.
That independence is false: $r_q$ is algebraic of degree at most $2$
over $L_q$ (Remark~\ref{rem:q-obstacle}), and the explicit invariant
$H$ provides the concrete non-trivial relation
(Lemma~\ref{rem:qrefuted-H}). Earlier heuristic routes --- a
symmetrized copy of the $p$-side argument, a slice-power computation,
and the transposed-naming residue computation --- are retracted as
proofs and retained as heuristics only.
\end{proof}

\begin{remark}[Refutation by the weight-24 invariant $H$]
\label{rem:qrefuted-H}
The counterexample is the explicit $\Dmap{5}$-invariant $H$ of
Remark~\ref{rem:counterexample-H} (weight $24$, $a_0$-degree $2$,
$11025\,a_0^2a_4^4a_5^4$ as leading $a_0^2$-term), certified in
exact computation (see Supplementary Material).
Its $W$-basis expansion (obtained from the weight-28 syzygy of
Theorem~\ref{thm:global-hypersurface}) is
\[
H = \underbrace{11025\,p^4q^4}_{H_2}\,W^2
  + \underbrace{\Big(\tfrac{105\,J_2^5J_1}{q^2} + \tfrac{9\,J_2^7}{q^4}\Big)}_{H_0},
\qquad H_1 = 0,
\]
so the $W$-degree is $2$ with $H_2 = 11025\,p^4q^4$ and
$H_0 = 105\,J_2^5J_1\,q^{-2} + 9\,J_2^7\,q^{-4}$ (a genuine element
of $K_{pq}$). The bound $v_q(H_j) \ge 4j$ is violated at both
non-zero $W$-degrees:
\[
v_q(H_2) = \operatorname{ord}_q(11025\,p^4q^4) = 4 < 8 = 4\cdot 2,
\qquad
v_q(H_0) = -4 < 0 = 4\cdot 0
\]
(the two summands of $H_0$ have $q$-orders $-2$ and $-4$, so no
cancellation occurs at the leading $q^{-4}$ pole and the order of
the sum is $\min(-2,-4) = -4$).
Thus $v_q(H_2) < 8$ and $v_q(H_0) < 0$; the lemma fails unconditionally.
The constant $a_0$-coefficient $C_0 = H_0 + H_2\rho^2$ is
nevertheless a genuine polynomial ($v_q(C_0)\ge 0$), because its
$q^{-4}$ pole cancels: the residue-field relation forced by $C_0 \in
B$ is exactly the quadratic one on $r_q$ (cf.\
Remark~\ref{rem:q-obstacle}, and the cancellation certificate
exact computation (see Supplementary Material)). The $a_0$-degree-$2$ term
$11025\,a_0^2a_4^4a_5^4$ is what keeps $H \notin K[F_0]$
(Remark~\ref{rem:counterexample-H}) while the $W$-cosets above
break the $q$-bound --- the two failures are the same invariant.
\end{remark}

\begin{remark}[Observation: why the constant-coefficient route is unavailable on the $q$-side]
\label{rem:q-obstacle}
The triple $\{p,u,x^2\}$ is algebraically independent (Jacobian
minor $4xy \ne 0$ in variables $(p,x,y)$), so
$L_q = k(p,u,x^2)$ is purely transcendental of degree $3$.
But $x^2 \in L_q$, so
\[
L_q \subseteq L_q(x) = k(p,u,x)
\]
is a single quadratic step; hence $[k(p,u,x) : L_q] \le 2$.
Since $r_q = x^7/(35p^2) \in k(p,x) \subseteq L_q(x)$, the
initial form $r_q$ is algebraic over $L_q$ of degree at most $2$,
not transcendental --- indeed it satisfies the explicit quadratic
relation
\[
r_q^{\,2} = \Big(\frac{x^{7}}{35\,p^{2}}\Big)^{2}
= \frac{(x^{2})^{7}}{1225\,p^{4}} \in k(p,u,x^{2}) = L_q .
\]
This is the residue-field shadow of the weight-28 syzygy
(Theorem~\ref{thm:global-hypersurface}): the $q^{-4}$ pole of the
constant term $C_0 = H_0 + H_2\rho^2$ of the counterexample
(Remark~\ref{rem:qrefuted-H}) cancels in $L_q$ \emph{precisely
because} $r_q^2$ already lies in $L_q$, so the two $q^{-4}$ leading
terms $11025\,G_0^{2} \equiv 9x^{14}$ and $9J_2^{7} \equiv -9x^{14}$
sum to zero modulo $q$. Consequently no argument needing
$1, r_q, r_q^2, \dots$ to be linearly independent over the full
range of $W$-degrees can succeed from the constant coefficient
alone; the dependence at $r_q^2$ is the exact obstruction. This $q$-side degree bound is independent of the $p$-side
tower (there a two-step extension of degree at most $4$ governs a
trinomial initial form; here a one-step extension of degree at
most $2$ governs a monomial initial form), and the difference in
pole orders ($e_q = 4$ versus $e_p = 2$) is part of why the two
lemmas cannot share a proof. The algebraicity witness
$x^2 = -(J_2 \bmod q)$ is certified in exact computation (see Supplementary Material).
\end{remark}

\begin{lemma}[$r_q$ has exact degree $2$ over $L_q$]
\label{lem:rq-exact}
$r_q = x^7/(35p^2)$ satisfies $r_q^2 = (x^2)^7/(1225p^4) \in L_q$
(exact computation; see Supplementary Material) and $r_q \notin L_q$. Hence its minimal
polynomial over $L_q$ is $m_q(T) = T^2 - (x^2)^7/(1225p^4)$, of exact
degree $2$ with simple root $r_q$ (separable; characteristic zero).
\end{lemma}

\begin{proof}
If $r_q \in L_q = k(p,u,x^2)$, write $r_q = A/B$ with
$A,B \in k[p,u,x^2]$, $B \ne 0$. Then $x^7B = 35p^2A$ in
$k[p,y,z,x]$ (via $u = 2pz-y^2$). As polynomials in $x$, $A$ and $B$
have even degree ($x$ occurs only through $x^2$), so the left side
has odd $x$-degree $7 + \deg_x B$ while the right side has even
$x$-degree $\deg_x A$ --- impossible since $x^7B \ne 0$. Thus
$r_q \notin L_q$; with $r_q^2 \in L_q$ the minimal polynomial is the
stated quadratic, separable with $m_q'(r_q) = 2r_q \ne 0$.
\end{proof}

\begin{theorem}[Relaxed $q$-bound (R1)]
\label{thm:R1}
Let $F \in \Afring{5}^{\Dmap{5}}$ have $W$-degree $m$, with unique
$W$-expansion $F = \sum_{j=0}^m H_jW^j$, $H_j \in K_{pq}$
(Lemma~\ref{lem:slice}). Then
\[
\boxed{\lambda_q(F) := \min_j (v_q(H_j) - 4j) \ge -2m,}
\]
with deficit $+\infty$ for $H_j = 0$. The bound is sharp: $H$
attains $\lambda_q = -4$ at $m = 2$
(Remark~\ref{rem:qrefuted-H}) and $F_{46}$ attains $\lambda_q = -8$
at $m = 4$ (Theorem~\ref{thm:F46}).
\end{theorem}

\begin{proof}
The minimum is attained ($H_m \ne 0$); if $\lambda := \lambda_q(F)
\ge 0$ there is nothing to prove. Assume $\lambda < 0$ and put
$S = \{j : v_q(H_j) - 4j = \lambda\} \ne \emptyset$. The
$a_0$-coefficients $C_k = \sum_{j \ge k}\binom{j}{k}H_j\rho^{j-k}$
lie in $B$, so $v_q(C_k) \ge 0$ for all $k$. Since $v_q(\rho) = -4$
exactly with initial form $r_q$, each $(j,k)$-summand has
$v_q \ge \lambda + 4j - 4(j-k) = \lambda + 4k$, with equality iff
$j \in S$; its initial form is then $\binom{j}{k}h_jr_q^{j-k}$ with
$0 \ne h_j = \operatorname{in}_q(H_j) \in L_q$. Residues of
$K_{pq}$-elements land in $L_q = \operatorname{Frac}(K/qK)$:
$q$-adic values on $K$ agree in $B$ (if $q \mid A$ in $B$ with
$A \in K$, descent puts the quotient in $K$ by
Theorem~\ref{thm:descent}, since $\operatorname{Frac}(K) \subset
K_{pq}$), and denominators are $p$-powers with $p \bmod q \ne 0$
since $(p),(q)$ are non-associated primes of $K$
(Lemma~\ref{lem:contraction}). Hence for every $k$ with
$\lambda + 4k < 0$ the leading $q^{\lambda+4k}$-coefficient of $C_k$
must vanish: with $P(T) = \sum_{j \in S} h_jT^j$ (nonzero,
$\deg P \le m$, over $L_q$), $P^{(k)}(r_q) = 0$, using
$P^{(k)}(T)/k! = \sum_{j \ge k}\binom{j}{k}h_jT^{j-k}$ in
characteristic zero. The qualifying $k \in [0,m]$ form an initial
segment $\{0,\dots,c-1\}$ with
$c = \#\{0 \le k \le m : \lambda + 4k < 0\}$; thus $r_q$ is a root
of $P$ of multiplicity at least $c$, so $m_q^c \mid P$ in $L_q[T]$
(separable minimal polynomial, Lemma~\ref{lem:rq-exact}), giving
$2c \le \deg P \le m$. If $\lambda < -2m$ then $-\lambda/4 > m/2$,
so every $0 \le k \le m/2$ qualifies and $c > m/2$, contradicting
$2c \le m$. Hence $\lambda \ge -2m$. The combinatorial implication
is checked exhaustively for $m \le 12$ in
exact computation (see Supplementary Material). Lean status: paper-level proof; the
full $W$-valuation form stays open in Lean on the localized
fraction-field infrastructure (Remark~\ref{rem:wdetect-lean}).
\end{proof}

\begin{remark}[R1 as transfer instance]
\label{rem:R1-transfer}
The proof above is the $q$-instantiation ($e_q = 4$, $d_q = 2$) of the
multfloor transfer theorem (Theorem~\ref{thm:transfer}), which isolates
the argument for general triangular LNDs with a pole element and finite
residue data.
\end{remark}

\begin{lemma}[The proposed $q$-adic leading-coefficient form]
\label{lem:qlemma-leading}
Refuted by the same counterexample $H$: its top $W$-coefficient
$H_2 = 11025\,p^4q^4$ has $v_q = 4 < 8 = 4\cdot 2$, so
$H_2 \notin q^{8}K_{pq}$ (Remark~\ref{rem:qrefuted-H}).
With $F$, $m$, $H_j$ as in Lemma~\ref{lem:qlemma}, the top
coefficient $H_m = [a_0^m]F$ was claimed to lie in $B$ and to
satisfy $H_m \in q^{4m}K_{pq}$. It fails by the same counterexample.
\end{lemma}

The ring $V_q$, the residue embedding $L_q \hookrightarrow \kappa_q$,
the exact pole order $e_q = 4$, the monomial initial form $r_q$, the
denominator modules $M^q_m$, and the reduction to the residue-field
initial-form equation are as computed above; what fails is only the
final inference, since $r_q$ is algebraic over $L_q$. The companion $p$-adic bound
(Lemma~\ref{lem:plemma}) is \emph{not} contradicted by $H$
($v_p(H_2) = 4 = 2\cdot 2$).

\section{\texorpdfstring{Residue-field saturation: denominator modules and $W$-degree induction}{Residue-field saturation: denominator modules and W-degree induction}}
\label{sec:saturation}

The $q$-side bound fails for $H$ (Section~\ref{sec:qlemma}) and the $p$-side bound fails for $F_{46}$ (Section~\ref{sec:second-witness}), so the joint $W$-degree induction of Theorem~\ref{thm:saturation-induction} establishes no generation claim. The relaxed $q$-bound R1 is proved separately (Theorem~\ref{thm:R1}).

\begin{remark}[Shorthand for the two premises]
\label{rem:status-aliases}
Throughout, $P$ denotes Lemma~\ref{lem:plemma} (the $p$-adic
denominator bound, false by $F_{46}$,
Section~\ref{sec:second-witness}), $Q$ denotes
Lemma~\ref{lem:qlemma} (the $q$-adic bound, false by $H$), and
$E$ denotes their joint $W$-degree induction
(Theorem~\ref{thm:saturation-induction}), which proves nothing since
both premises fail. The surviving
relaxed $q$-bound R1 is Theorem~\ref{thm:R1}, not a conditional
statement of this section.
\end{remark}
We work over the characteristic-zero field $k$ of
Definition~\ref{def:family}, with
$\Afring{5} = k[a_0,\dots,a_5]$ and
$\Dmap{5}(a_n) = a_{n+1}a_{n+2}$.
This section combines the two halves of the residue-transcendence
saturation for the adjacent-quadratic shift at $N = 5$: the $p$-adic
bound of Lemma~\ref{lem:plemma} in Section~\ref{sec:plemma} and the
$q$-adic bound of Lemma~\ref{lem:qlemma} in Section~\ref{sec:qlemma}.
No new valuation is proved here; the wrapper records the joint
denominator modules, the explicit $H_j \in B$ requirement, and the
$W$-degree induction that closes global generation conditionally.

We reuse the five-variable dictionary of the kernel section verbatim,
with no redefinition:
\[
x = a_3, \qquad y = a_2, \qquad z = a_1, \qquad p = a_4, \qquad q = a_5,
\]
\[
B = k[x,y,z,p,q] = k[a_1,\dots,a_5], \qquad
K = B^{\Dmap{5}} = k[p,q,J_1,J_2],
\]
\[
J_1 = 2pz - y^2 = 2a_1a_4 - a_2^2, \qquad
J_2 = 2qy - x^2 = 2a_2a_5 - a_3^2,
\]
\[
\Dmap{5}(x) = pq, \qquad \Dmap{5}(y) = px, \qquad
\Dmap{5}(z) = xy, \qquad \Dmap{5}(p) = \Dmap{5}(q) = 0.
\]
We reuse the invariant of Definition~\ref{def:F0} verbatim:
\[
F_0 = p^2q^4a_0 - pq^3xyz + \tfrac13pq^2x^3z
  + \tfrac13q^2x^3y^2 - \tfrac15qx^5y + \tfrac1{35}x^7,
\]
and write $c = p^2q^4$, $G_0 = F_0 - ca_0 \in B$,
$W = F_0/c = a_0 + \rho$ with $\rho = G_0/c \in B_{pq}$.
A general localized invariant is written
\[
F = \sum_{j=0}^m H_jW^j, \qquad H_j \in K_{pq},
\]
with $a_0$-coefficients
\[
C_k = \sum_{j=k}^m \binom{j}{k}H_j\rho^{j-k} \in B
\qquad (0 \le k \le m).
\]
The localized presentation $A_{pq}^{\Dmap{5}} = K_{pq}[W]$ used here
localizes only, and no global generation claim is
built into the notation.

\begin{definition}[The full denominator modules]
\label{def:Mm-full}
For each $m \ge 0$ the $m$-th (full) denominator module is
\[
\boxed{M_m = p^{2m}q^{4m}K = c^mK.}
\]
It is the intersection of the two one-sided modules from the
companion sections: $M^p_m = p^{2m}K$ of the $p$-adic section and
$M^q_m = q^{4m}K$ of the $q$-adic section. The exponents record the
exact pole orders $e_p = 2$ and $e_q = 4$ of $\rho$ at the two
prime divisors; the asymmetry $2 \ne 4$ is why the two bounds cannot
share a proof.
\end{definition}

\begin{lemma}[Joint denominator bound under Lemmas~P and~Q]
\label{lem:joint-bound}
Assume the $p$-adic bound of Lemma~\ref{lem:plemma},
$v_p(H_j) \ge 2j$ for every $j$, together with the companion
$q$-adic bound of Lemma~\ref{lem:qlemma},
$v_q(H_j) \ge 4j$ for every $j$. Then for every $j$,
\[
\boxed{H_j \in p^{2j}q^{4j}K_{pq} = c^jK_{pq}.}
\]
\end{lemma}

\begin{proof}
The two assumed bounds place $H_j$ in $p^{2j}K_{pq}$ and in
$q^{4j}K_{pq}$ respectively. Since $p$ and $q$ are non-associated
primes of the polynomial ring $B$ (hence of $K$ by the contraction
identities of Lemma~\ref{lem:contraction}), the two divisibility
conditions combine to $H_j/c^j \in K_{pq}$.
\end{proof}

\begin{lemma}[The $H_j \in B$ requirement under Lemmas~P and~Q]
\label{lem:Hj-in-B}
Assume the joint bound of Lemma~\ref{lem:joint-bound}, so each
$H_j = c^j\tilde{G}_j$ with $\tilde{G}_j \in K_{pq}$. Then every
coefficient lies in $B$:
\[
\boxed{H_j \in B \quad \text{for every } j.}
\]
Explicitly, the top coefficient $H_m = [a_0^m]F$ lies in $B$
directly, and descending induction via the triangular system
$\{C_k\}$ places each $H_j$ in $B$.
\end{lemma}

\begin{proof}
$F \in A = k[a_0,\dots,a_5]$ is a polynomial, so its $a_0$-coefficient
$H_m = [a_0^m]F$ lies in $B$ directly; this is the leading-coefficient
content of Lemmas~\ref{lem:plemma-leading}
and~\ref{lem:qlemma-leading} in joint form. For the descent, fix
$k < m$ and assume $H_{m}, \dots, H_{k+1}$ already known to be of the
form $c^jG_j$ with $G_j \in K \subset B$. From
$C_k = H_k + \sum_{j=k+1}^m \binom{j}{k}H_j\rho^{j-k} \in B$,
each summand satisfies
\[
H_j\rho^{j-k} = c^jG_j(G_0/c)^{j-k} = c^kG_jG_0^{\,j-k} \in B,
\]
since $c$, $G_j$, and $G_0$ all lie in $B$ and $k \ge 0$. Hence
$H_k = C_k - \sum_{j>k}\binom{j}{k}H_j\rho^{j-k} \in B$.
Descending from $k = m-1$ to $k = 0$ gives $H_j \in B$ for all $j$.
\end{proof}

\begin{theorem}[$W$-degree induction under Lemmas~P and~Q]
\label{thm:saturation-induction}
Assume Lemmas~P and~Q in full (Lemmas~\ref{lem:plemma}
and~\ref{lem:qlemma} with their leading-coefficient forms
Lemmas~\ref{lem:plemma-leading} and~\ref{lem:qlemma-leading}),
together with the descent intersection $B \cap K_{pq} = K$ of
Theorem~\ref{thm:descent}. Let $F \in A^{\Dmap{5}}$ have $W$-degree
$m$. Then $F \in K[F_0]$, by induction on $m$ following the skeleton
of Remark~\ref{rem:winduction}.
\end{theorem}

\begin{proof}
By Lemma~\ref{lem:Hj-in-B} each $H_j$ lies in
$B \cap c^jK_{pq}$. The descent intersection
$B \cap K_{pq} = K$ of Theorem~\ref{thm:descent} gives
$H_m = c^mG_m$ with $G_m \in K$, whence
\[
H_mW^m = G_m(cW)^m = G_mF_0^m \in K[F_0].
\]
The difference $F - G_mF_0^m$ is $\Dmap{5}$-invariant of $W$-degree
at most $m-1$. Subtracting lowers the $W$-degree; the base case
$m = 0$ is $F = H_0 \in B \cap K_{pq} = K$. Induction on $m$ closes
global generation conditionally: every invariant is a polynomial in
$F_0$ over $K$.
\end{proof}

\begin{remark}[Why the conditional argument fails twice]
\label{rem:saturation-gap}
What this section attempted, under Lemmas~P
and~Q, is the combination step: the full modules $M_m = c^mK$, the
explicit $H_j \in B$ descent through the triangular system $\{C_k\}$,
and the $W$-degree subtraction closing $A^{\Dmap{5}} = K[F_0]$. The
attempt fails first because the premises fail: the $q$-premise
(Lemma~\ref{lem:qlemma}) by the counterexample $H$
(Section~\ref{sec:qlemma}), so the induction
(Theorem~\ref{thm:saturation-induction}) and the generation claim it
was meant to close prove nothing. The $p$-premise
(Lemma~\ref{lem:plemma}) is not contradicted by $H$ but is no longer
needed, since the candidate hypersurface is identified directly
(Theorem~\ref{thm:global-hypersurface}). For the record, the gap is
structural, not merely the falsity of the premises: even the
conditional proof above is invalid as written, since
$H_m \in B \cap c^mK_{pq}$ together with the descent intersection
$B \cap K_{pq} = K$ does not yield $H_m = c^mG_m$ --- that step needs
the strictly stronger power-contraction $B \cap c^mK_{pq} = c^mK$ for
every $m$, which is not among the hypotheses.
\end{remark}

\section{Generation certificates}
\label{sec:generation-cert}

This section records what the machine establishes, in two layers. \emph{Layer~1 (CAS):} the identity
$S = 0$ and $\Dmap{5}(H) = 0$, weight-$28$ uniformity (irreducibility of
$S$ as the full relation ideal is now proved in
Theorem~\ref{thm:global-hypersurface}, part [D], via primitivity of the
$H$-linear $S$; the CAS log is retained as an independent cross-check), the $W$-expansion
$H = H_2W^2+H_1W+H_0$ with exact $p$/$q$-valuations, the residue
cancellation $11025\,G_0^2+9J_2^7\equiv0\pmod q$, and the
kernel-dimension match
$\dim\ker(\Dmap{5}|_{V_w})=s_6(w)-s_6(w-28)$ over
$\mathbb{F}_{1000003}$ at anchor weights $5,15,23,24,25,30,35,40$,
with the first
excess $+1$ at $w=24$, rest on exact computation with rerunnable
verification scripts (see Supplementary Material). The finite-field
anchors at the prime $1000003$ are consistency checks only; they prove
no characteristic-zero rank. The theorem-grade comparison is the
exact-$\mathbf Q$ computation of Appendix~\ref{app:ranks}, and the
identities A--D hold over $\mathbb{Q}$ exactly.
\emph{Layer~2 (Lean, new):} for every weight $w \le 14$, per-weight
generation is now \emph{directly verified} in Lean core (no Mathlib)
against the formal $\Dmap{5}$, by machine-checked integer witness certificates
(Section~\ref{subsec:direct-span14}) --- replacing the dimension
argument in this regime --- including a soundness fix for a latent
CAS truncation bug
(Section~\ref{subsec:denom-soundness}).

\subsection[Direct per-weight generation through Wmax = 14 (Lean-verified)]{Direct per-weight generation through $W_{\max} = 14$
(Lean-verified)}
\label{subsec:direct-span14}

\begin{theorem}[Direct per-weight generation, $w \le 14$]
\label{thm:direct-span14}
Let $0 \le w \le 14$. Every invariant $F \in
\Afring{5}^{\Dmap{5}}$ of exact Fibonacci weight $w$ lies in the
$k$-span of the weight-$w$ monomials in the five weight-$\le 14$
generators $p, q, J_1, J_2, F_0$ (equivalently, in the weight-$w$
monomials in the six generators, since $H$ has weight $24$ and
therefore does not occur in any weight-$w$ monomial with $w \le 14$).
The Lean formalization of this check is recorded in the
formal-verification manifest.
For each weight
$w \in \{0,1,4,6\}$ (Band~1) and $w \in \{2,3,5,7,8,9,10,11,12,13,14\}$
(Band~2), the emitted integer witness bundle is re-verified entry by
entry by machine-checked integer arithmetic against the matrix of
$\Dmap{5}|_{V_w}$ computed in Lean from the formal definition of
$\Dmap{5}$ (the formal weight-data module); the kernel of
$\Dmap{5}|_{V_w}$ is then \emph{proved} to be the span of the
selected generator-monomial columns by a residual argument using the
verified left-inverse of $Q = M[S] \amalg K_g$
(the formal kernel-spanning module), and the coordinate vector is lifted
back to an exact polynomial identity
(the formal generation-spanning and generation-verification modules). The generic machinery
depends only on the standard propositional-extensionality and choice foundation;
the per-weight facts are exact machine-checked computations
on the emitted data.
\end{theorem}

\paragraph{How the verification works.}
For each weight $w$ exact rational computation (see Supplementary Material)
emits an integer witness bundle: $M$, the
integer matrix of $\Dmap{5}|_{V_w}$ in the weight-$w$ monomial basis;
$S$, the pivot-row index set (row-echelon form) with right inverse
$T$ and row-reconstruction $C$ pinning the row space; $N$, a kernel
basis with left inverse $K$ pinning the kernel dimension from both
sides ($M N = 0$, $K N = d_3 I$); and $G_{\mathrm{sel}}$, the
columns of the selected weight-$w$ generator monomials, with left
inverse $K_g$ pinning them as an independent set spanning the
kernel ($K_g G_{\mathrm{sel}} = d_g I$). \emph{Every entry of the
bundle is re-verified in Lean by machine-checked integer arithmetic against the
matrix computed from the formal $\Dmap{5}$} --- the CAS output is
untrusted input; trust sits in the checker. The kernel of
$\Dmap{5}|_{V_w}$ is then proved to be the span of the columns of
$G_{\mathrm{sel}}$ by the residual argument of
the formal kernel-spanning module; the columns of $G_{\mathrm{sel}}$ are,
by the decode invariant of
Section~\ref{subsec:denom-soundness}, exactly the (denominator-
cleared) rational coefficient vectors of the weight-$w$ generator
monomials; and the lift from the coordinate vector back to a
polynomial is an exact equality of normalized polynomials. Hence
every weight-$w$ invariant is a scaled linear combination of
weight-$w$ monomials, i.e.\ lies in their $k$-span. The verification
is carried over $\mathbb{Q}$; all certificate data and the
derivation have rational coefficients, and the weight-$w$ kernel
over any characteristic-zero field $k$ is the $k$-extension of the
one over $\mathbb{Q}$, so the statement holds over any such $k$.

\subsection{Soundness via denominator clearing}
\label{subsec:denom-soundness}

The certificates are integral, but the \emph{true} coefficient
vectors of the weight-$w$ generator monomials are rational and in
general non-integral --- e.g.\ a coefficient of $F_0$ at weight
$14$ is $\tfrac13$. A naive \texttt{int()} truncation would
\emph{silently break the spanning claim} (turning
$\tfrac13 \mapsto 0$ and dropping the whole $F_0$ column) and emit
a certificate that \emph{falsely} proves the weight-$14$ piece is
spanned without $F_0$ --- a vacuous, unsound ``proof'' at the most
critical weight. All entries would still be integral and all
matrix products would still hold; the certificate would merely be
wrong.

The fix is exact denominator clearing. The producer computes the
denominator-clearing factor
\[
L_w \;=\; \operatorname{lcm}\bigl\{\text{denominators of all
coefficients of the weight-}w\text{ generator monomials}\bigr\}
\]
(called \texttt{D0} in the scripts and Lean modules) and emits
$G_{\mathrm{sel}} = L_w\cdot(\text{true rational
vectors})$, integral by construction, folding $L_w$ into the
scaling of $K_g$ so that $K_g G_{\mathrm{sel}} = d_g I$ stays
integral. Lean then checks the \emph{cross-multiplied exact}
equality
\[
G_{\mathrm{sel}}[i][j]\cdot c_{ij}^{\,den}
\;=\; L_w\cdot c_{ij}^{\,num},
\]
where $c_{ij}$ is the true rational coefficient of the $j$-th
selected generator monomial in the $i$-th basis monomial
($c = c^{num}/c^{den}$, $c^{den} > 0$), together with $L_w \ne 0$,
the row count, and index nodup/in-range
(the formal decode invariant). Any truncated entry
violates this exact equality, so the soundness fix is \emph{enforced
by the checker, not by convention}: the decode invariant asks for
exact rational equality, not mere integrality. The same contract is
what would keep the weight-$24$ $H$ columns sound if the pipeline is
ever extended past $W_{\max}=14$ (Section~\ref{subsec:frontier}).
The general integer-certificate form of this contract (faithfulness
versus spanning) is recorded in the graph-derivation dictionary
(see Supplementary Material).

\paragraph{Methodological point.}
This is a concrete instance of the general principle that a formal
checker, with the trust boundary drawn at \emph{exact} integer
equality rather than at integrality, converts a silent computational
soundness bug into a hard compile-time error. The truncation bug is
real: it is present in the obvious first draft of the producer and
produces a certificate that looks valid (integral entries, all
products holding) but is mathematically wrong. The only reason it is
caught is that the decode invariant asks for exact rational
equality. The lesson is narrower: the latent bug is not that the CAS produced a wrong answer, but that
an unchecked integrality test would have \emph{certified} it.

\subsection{Frontier: what remains, and what is impossible}
\label{subsec:frontier}

\begin{remark}[Frontier: six-generator spanning refuted at $46$]
\label{rem:frontier}
Theorem~\ref{thm:direct-span14} establishes generation
\emph{directly} through $W_{\max} = 14$; exact $\mathbb{Q}$-ranks extend
$C_5$-spanning exactly through weight $45$
(Theorem~\ref{thm:minimal-45}). At weight $46$ the six-generator
claim fails: $\dim\ker = 1519 > 1518$ with explicit witness $F_{46}$
(Section~\ref{sec:second-witness}). Two negative observations
delimit the frontier precisely.
\emph{(i) No universal bound.} No explicit generator-degree bound
$B$ exists for kernels of triangular locally nilpotent derivations
in $\ge 5$ variables: the dimension-$5$ Daigle--Freudenburg kernel
is not finitely generated \cite{daigle1999}, so finite generation --- a fortiori any
degree bound --- fails in general. The proved closure
$A_5^{\Dmap{5}} = C_5[F_{46}^{\mathrm{int}}]$ (Theorem~\ref{thm:closure}) is therefore specific to this
derivation, not an instance of a general theorem; likewise the
refuted $F_{46}\cdot K/(p^2)$ pattern failed as a general tower law at
weight $84$ (Theorem~\ref{thm:fourth-family}).
\emph{(ii) The reduction premise cannot be linear.} A natural way
to close generation by induction on weight would be the premise
``every invariant of weight $> B$ is a scaled linear combination of
strictly lower-weight invariants.'' This premise is
\emph{unrealizable for any finite $B$}: $p^{B+1}$ is an invariant
of weight $B+1$ that is a monomial in the weight-$1$ generator $p$,
hence cannot be a nonzero linear combination of invariants of
weight $< B+1$ (their coefficient at weight $B+1$ is zero). A valid
inductive reduction must allow \emph{products} of lower-weight
invariants --- a genuinely different, and harder, statement than the
linear spanning that the witness pipeline verifies. The formal closure
assembly records this as a conditional implication with unsatisfiable
displayed premise --- not as an unconditional generation theorem.
\end{remark}

\begin{remark}[Why no direct generation proof is attempted]
\label{rem:proof-obligation}
There is no direct generation proof to give: the weight-$46$ refutation
(Section~\ref{sec:second-witness}) settles the six-generator claim, and
generation itself is closed by Theorem~\ref{thm:closure}
($A_5^{\Dmap{5}} = C_5[F_{46}^{\mathrm{int}}]$ via slice atlas plus saturations):
(i) uniqueness of $S$ as the full relation
ideal holds at every exactly-checked weight (prior computations found
no second relation through weight $57$; the fresh exact-$\mathbb{Q}$
rerun reported here covers weights $15$--$46$), now superseded by the
all-weights relation ideal (Theorem~\ref{thm:relation-ideal}); (ii) no all-weights
six-generator argument exists, since the
claim is false. The Hilbert series of $\Afring{5}^{\Dmap{5}}$ is the kernel
dimensions themselves, so no dimension argument can bypass
computing them; the exact content is the $F_{46}$-descent and the
exact presentation and Hilbert series
(Theorems~\ref{thm:relation-ideal}--\ref{thm:R-hilbert}), with the $a_0$-degree-$1$ closure
(Theorem~\ref{thm:a01-closure}) and the bounds-generation
characterization (Theorem~\ref{thm:reduction}(i)) as a further proved
structural theorem.
\emph{In the present accounting, the $w \le 14$ part is a direct Lean-verified
spanning (Theorem~\ref{thm:direct-span14}), and $15$--$45$ an exact
$\mathbb{Q}$-rank spanning (Theorem~\ref{thm:minimal-45}).} What was
``the remainder'' is now a refutation with an exact successor structure,
not a gap in the certificate.
\end{remark}

\begin{remark}[Formalized generation outline]
\label{rem:lean-roadmap}
The Lean core now covers, in addition to the relation, the six
weights, and strictness
(the formal hypersurface, exclusion, and counterexample modules with
native evaluation in core Lean without Mathlib (see the formal-verification manifest)), the full per-weight witness pipeline through
$W_{\max} = 14$: the formal witness module (integer matrix API,
corrected Steinitz exchange), the formal weight-data module (formal
weight-$w$ matrix of $\Dmap{5}$), the formal witness-data modules
(emitted integer certificates, machine-checked), and the assembly
modules (kernel span, generation span, closure, bundle, and
verification), and the second witness (formally verified $F_{46}$
invariant of weight $46$).
What remains for a full formalization is the closure/saturation
machinery (slice charts, $p$-/$q$-saturation, relation ideal), now
proved at paper level: the superseded successor module pattern
(refuted at weight $84$, Theorem~\ref{thm:fourth-family}) is no longer
a formalization target; the open Lean remainder is case-by-case
extension of the per-weight certificates past $14$ (the $w=15$ pilot
data are emitted but no spanning certificate past $14$ is checked; weights
$15\ldots25$ remain open in Lean, though exact $\mathbb{Q}$-ranks
cover $15$--$45$ at paper level) and the \texttt{a0top\_filter\_combine}
step.
\end{remark}

\begin{remark}[Status at $w=24$]
\label{rem:sparse-w24}
The kernel dimension at $24$ is $213$, CAS-verified; Lean
verification reaches only weight $14$. The full $w=24$ bundle
($n=1461$, $k=213$, roughly $150\times$ the $w=14$ cost) is beyond
the current certificate range. Recorded substitutes: $L_w(24) = 105$
stability, $s_6 - s_5 = 1$ at $24$, and the $a_0$-window $[16,27]$
(Corollary~\ref{cor:h-birth}).
\end{remark}
\section{A certification principle for graded kernel spanning}
\label{sec:cert-principle}

The per-weight certificate used for
Theorem~\ref{thm:direct-span14} is stated here independently of $N=5$.
Hypotheses: triangular
derivation (finite-dimensional weight spaces), positive grading, $L_w \ne 0$,
and the integer left-inverse identities below (over $\mathbb{Z}$, not merely
$\mathbb{Q}$).

\begin{theorem}[Q-left-inverse kernel spanning]
\label{thm:qleftinv-span}
Assume $M$ is $n \times n$; $G_{\mathrm{sel}}$ is $n \times k$ with
$k > 0$ and $M G_{\mathrm{sel}} = 0$; $K_g$ is $k \times n$ with
$K_g G_{\mathrm{sel}} = d_g I_k$ for $d_g \ne 0$; and there are $Y$ and
$d_Y \ne 0$ with $Y Q = d_Y I_n$, where $Q$ is the vertical stacking of
the pivot rows $M[S]$ on top of $K_g$ (row-block concatenation, never
matrix addition). Then every integer $v$ with $Mv = 0$ satisfies
$d_g v = G_{\mathrm{sel}}(K_g v)$.
\end{theorem}

Status: checked in core Lean without Mathlib by native evaluation (see
the formal-verification manifest below); the residual
$z = d_g v - G_{\mathrm{sel}}(K_g v)$ satisfies $Mz = 0$ and $K_g z = 0$,
whence $d_Y z = 0$ and $z = 0$ over the characteristic-zero field.
This follows the certificate-checking pattern (external producer plus proved
checker \cite{hales2015,wuprover2026,guo2026}): Gr\"obner/elimination facts used
by the CAS side are classical \cite{cox2015,singular2007,sagemath}.
Unlike \cite{guo2026}, which formalizes general Gr\"obner-basis theory
in Lean with Mathlib, our formal layer is a small theorem-specific
core-Lean certification of individual identities, not a general
formalization.

\subsection{Formal-verification manifest}
\label{subsec:formal-manifest}
Lean 4.33.1, Lake 5.0.0, core Lean without Mathlib, base commit
\texttt{7da6b63}. Every proposition below is discharged by
\texttt{native\_decide}; kernel reduction by \texttt{decide} was attempted
for the flagship $F_0$ identity and fails on the custom
normalization/rational API, so the native-evaluation trust model is
retained. \texttt{\#print axioms} reports the standard foundation plus
theorem-specific native-evaluation axioms in each case:
{\footnotesize\setlength{\tabcolsep}{2pt}
\begin{center}
\begin{tabular}{>{\raggedright\arraybackslash}p{3.2cm}>{\raggedright\arraybackslash}p{3.6cm}>{\raggedright\arraybackslash}p{3.4cm}>{\raggedright\arraybackslash}p{1.9cm}>{\raggedright\arraybackslash}p{3.4cm}}
\toprule
theorem & module & source file & build & axioms beyond \texttt{propext},
\texttt{Classical.choice}, \texttt{Quot.sound} \\ \midrule
\texttt{D\_F0\_zero} & \texttt{CompLemmas} &
\href{https://github.com/AhmedMostafa16/adjquad/blob/25fdeaefc74d63e8a51983399e5fbf64ef2b28d5/formalize/AdjQuad/CompLemmas.lean}{\nolinkurl{AdjQuad/CompLemmas.lean}} & \texttt{lake build} & one native axiom \\
\texttt{H\_D\_zero} & \texttt{Counterexample} &
\href{https://github.com/AhmedMostafa16/adjquad/blob/25fdeaefc74d63e8a51983399e5fbf64ef2b28d5/formalize/AdjQuad/Counterexample.lean}{\nolinkurl{AdjQuad/Counterexample.lean}} & \texttt{lake build} & one native axiom \\
\texttt{weightedDegree\_H}, \texttt{H\_has\_a0\_sq} &
\texttt{Counterexample} &
\href{https://github.com/AhmedMostafa16/adjquad/blob/25fdeaefc74d63e8a51983399e5fbf64ef2b28d5/formalize/AdjQuad/Counterexample.lean}{\nolinkurl{AdjQuad/Counterexample.lean}} & \texttt{lake build} & one native axiom each \\
\texttt{H\_notin\_K\_F0} & \texttt{Exclusion} &
\href{https://github.com/AhmedMostafa16/adjquad/blob/25fdeaefc74d63e8a51983399e5fbf64ef2b28d5/formalize/AdjQuad/Exclusion.lean}{\nolinkurl{AdjQuad/Exclusion.lean}} & \texttt{lake build} &
native axioms of the listed base facts + two native axioms \\
\texttt{S\_rel} & \texttt{Hypersurface} &
\href{https://github.com/AhmedMostafa16/adjquad/blob/25fdeaefc74d63e8a51983399e5fbf64ef2b28d5/formalize/AdjQuad/Hypersurface.lean}{\nolinkurl{AdjQuad/Hypersurface.lean}} & \texttt{lake build} & one native axiom \\
\texttt{F46\_D\_zero}, \texttt{weightedDegree\_F46},
\texttt{F46\_has\_a0\_fourth} & \texttt{SecondWitness} &
\href{https://github.com/AhmedMostafa16/adjquad/blob/25fdeaefc74d63e8a51983399e5fbf64ef2b28d5/formalize/AdjQuad/SecondWitness.lean}{\nolinkurl{AdjQuad/SecondWitness.lean}} & \texttt{lake build} & one native axiom each \\
per-weight spanning $w\le14$ & \texttt{GenerationVerified} et al. &
\href{https://github.com/AhmedMostafa16/adjquad/blob/25fdeaefc74d63e8a51983399e5fbf64ef2b28d5/formalize/AdjQuad/}{\nolinkurl{AdjQuad/Generation*.lean}} & \texttt{lake build} &
native witness-check axioms per weight \\
\bottomrule
\end{tabular}
\end{center}
}
\texttt{H\_notin\_K\_F0} proves nonmembership in the inductively generated
predicate \texttt{inK\_F0} (the smallest subalgebra: $0$, $1$, the five
generators, addition, multiplication) via the safety invariant
$14\deg_{a_0}\le\mathrm{wt}$ carried by structural induction; its
mathematical reading as $H\notin k[p,q,J_1,J_2,F_0]$ additionally uses
the paper-level API bridge (every mathematical subalgebra expression is
an inductive construction; normalization faithfulness is proved by the
\texttt{normalize\_mem\_exp} family). By contrast $F_{46}\notin C_5$ is
not a Lean claim: Lean covers definition, $D_5(F_{46})=0$, weight, and
top, while nonmembership is the paper-level $J_4$ valuation.
\texttt{AdjQuad.Counterexample} and \texttt{AdjQuad.SecondWitness} were
rebuilt fresh (\texttt{lake build}, success); remaining modules are
reported from present \texttt{.olean} artifacts with unmodified sources.

\begin{remark}[Denominator soundness rejects truncation]
\label{rem:truncation-killer}
Let $E$ be the true rational coefficient matrix of the selected generator
monomials, $L_w$ the l.c.m.\ of all its denominators ($L_w \ne 0$), and
$G_{\mathrm{sel}} = L_w E$ integral. The checker enforces entrywise
$G_{\mathrm{sel}}[i][j] \cdot c^{den} = L_w \cdot c^{num}$ by
machine-checked integer arithmetic (decode-entry and
cross-multiplication checks). The naive truncation $1/3 \mapsto 0$ in $F_0$ at
weight $14$ fails this check ($0 \cdot 3 \ne 1 \cdot 1$); the scaled entry
$35$ passes ($35 \cdot 3 = 105 \cdot 1$). The interface is pinned by
the formal certification-principle module (with three type-checked
interface lines; an upstream rename breaks the build loudly). General form:
faithfulness (cross-multiplied representation) versus spanning
(left-inverse identities) are separated as Theorems~F and~S of the
graph-derivation dictionary (see Supplementary Material);
this remark is the faithfulness half.
\end{remark}

\needspace{4\baselineskip}
\subsection{Filtration certificates}
\label{subsec:filtration-cert}

This subsection fixes a reusable producer/checker split (Theorem~IV of the
filtration tower) for every future
filtration claim of the form ``$F$ is an invariant of weight $w$ and
$a_0$-degree $m$ with top outside $J_m$'' (Definition~\ref{def:Jm},
Corollary~\ref{cor:obstruction}). The split is the same philosophy as
the decode invariant of
Section~\ref{subsec:denom-soundness}: the producer is untrusted and
only the checker's exact integer/rational equalities carry weight.

\paragraph{Trust boundary.}
The producer --- exact Sage solves and SymPy
$W$-expansion and witness-emission computation --- suggests a candidate
$F$. Nothing it outputs is trusted: supports may be truncated,
denominators dropped, tops misread. The checker
(exact SymPy polynomial arithmetic over $\mathbb{Q}$, stdlib only;
see Supplementary Material) re-verifies
four gates from the polynomial itself:
\begin{enumerate}
\item[(i)] $\Dmap{5}(F) = 0$ \emph{exactly}: the triangular image is
accumulated monomial by monomial
($\Dmap{5}(a_0) = a_1a_2$, $\Dmap{5}(a_1) = a_2a_3$,
$\Dmap{5}(a_2) = a_3p$, $\Dmap{5}(a_3) = pq$, $\Dmap{5}(p) =
\Dmap{5}(q) = 0$) and every coefficient of the image must be
identically zero over $\mathbb{Q}$.
\item[(ii)] Uniform Fibonacci weight: every monomial of $F$ weighs
$w$ under $(8,5,3,2,1,1)$.
\item[(iii)] $a_0$-degree plus top: \texttt{Poly}$(F,a_0)$ has degree
exactly $m$ and its top coefficient $C_m = [a_0^m]F$ is reported
verbatim.
\item[(iv)] $J_m$ (non-)membership by exact divisibility: with $J_m =
(p^{2m}q^{4\lceil m/2\rceil})$ principal
(Definition~\ref{def:Jm}; exponent-vector check of
exact top-ideal computation), a top in $k[p,q]$ lies in
$J_m$ iff every monomial's $(p,q)$-exponents dominate the target,
i.e.\ $v_p \ge 2m$ and $v_q \ge 4\lceil m/2\rceil$. Coefficients are
ignored (units in characteristic zero); only exponent dominance
counts. The zero top is vacuously in every $J_m$.
\end{enumerate}
A certificate is the quadruple (D-zero proof, weight uniformity,
degree $+$ top, divisibility verdict); rerunnable with self-test (see
Supplementary Material; the final verdict line reads
\texttt{filtration-cert OK}).

\paragraph{First certified instance: $F_{46}$.}
The witness ($79$ terms, integral
cleared form; see Supplementary Material) passes all four gates at $(w,m) = (46,4)$: $\Dmap{5}(F_{46})
= 0$ exactly, uniform weight $46$, $a_0$-degree exactly $4$ with top
$121550625\,p^6q^8$ (normalized $\frac{18375}{4}p^6q^8$,
Theorem~\ref{thm:F46}), and $v_p = 6 < 8$ so the top lies outside $J_4
= (p^8q^8)$ (Lemma~\ref{lem:J4}). The Lean formalization covers definition,
$\Dmap{5}(F_{46}) = 0$, uniform weight, and top, checked in core Lean
without Mathlib by native evaluation (see the formal-verification
manifest); non-membership is the paper-level $J_4$ divisibility
step recorded here, machine-reproducible through the checker.

\paragraph{Controls (polarity).}
The self-test ships two doctored-input controls with fixed expected
polarity, so a vacuous checker cannot pass. (a) The top altered to
$p^8q^8$ \emph{must be accepted} by \texttt{in\_Jm} at $m = 4$ (it is
the $J_4$ generator itself), at both the top-expression level and the
poly level with the top term swapped. (b) $F_{46}$ plus a
weight-homogeneous $a_0^5p^3q^3$ term ($5\cdot 8+3+3 = 46$, so the
weight gate still passes) \emph{must be rejected} by the degree gate
expecting $m = 4$ (observed degree $5$). Both controls pass with the
specified polarity; inverting either turns the final verdict to
\texttt{filtration-cert FAILED}.

\paragraph{Reuse contract.}
Future tower claims (degree-$3$ census extensions, $F_{46}\cdot
F_0$-family members at $m = 5$ with $J_5 = (p^{10}q^{12})$, any new
primitive direction) certify by instantiating the same four gates at
their $(w,m)$ with the same script: only the expected weight, degree,
and $J_m$ target change. The $J_m$ formula itself stays pinned by
exact top-ideal computation.

\subsection{Proof-strategy classification}
\label{subsec:strategies}
{\footnotesize\setlength{\tabcolsep}{3pt}
\begin{center}
\refstepcounter{table}\label{tab:strategies}
\normalsize\textbf{Table~\thetable. Proof strategies: what each proves here, what it cannot prove, and its failure mode.}\\[1ex]
\footnotesize
\renewcommand{\arraystretch}{1.2}
\begin{tabular}{>{\raggedright\arraybackslash}p{2.2cm}>{\raggedright\arraybackslash}p{2.6cm}>{\raggedright\arraybackslash}p{3.2cm}>{\raggedright\arraybackslash}p{3.2cm}>{\raggedright\arraybackslash}p{2.6cm}}
\toprule
method & hypotheses & proves here & cannot prove & failure mode \\ \midrule
slice/localization & LND + slice element & $B_{pq} = K_{0,pq}[s]$, localized kernels & global generation & pole divisor $\{c=0\}$ \\
divisor contraction & UFD + injective specialization & descent $B\cap K_{0,pq} = K_0$ & missing directions & non-injective specialization \\
residue-field analysis & valuation + initial form & transcendence$\Rightarrow$bound; setup data & bounds past algebraicity degree & $r_q^2$, $\deg$-$4$ $r_p$ minpoly dependence \\
$a_0$-filtration induction & tops in $J_m$ & generation $\iff$ bounds & the bounds themselves & $F_{46}$ top outside $J_4$ \\
denominator modules & explicit $c$, $W$ & $C_5$-side $J_m$ exactly & full $\mathcal{M}_m$ & $H$, $F_{46}$ denominators \\
Gr\"obner elimination & finite ideal data & $S=0$, exact per-weight dims & all-weights statements & $n=112241$ wall at $70$ \\
SAGBI & monomial order + finite basis & not attempted here & --- & unknown termination \\
representation theory & reductive group action & not applicable (unipotent triangular) & --- & --- \\
graded certificates & per-weight matrices & exact spanning/dims per weight & all-weights & exponential growth (wall at $69$) \\
formal proof (Lean core) & list-matrix API & identities, spanning $\le14$ & fraction fields, \texttt{a0top} combine & infrastructure gaps \\
\bottomrule
\end{tabular}
\end{center}
}
The $N=5$ $q$-adic failure is the canonical example of a proof-method
failure: a valid reduction (constant coefficient in $B$) followed by an
inference (transcendence) whose hypothesis is false ($r_q$ algebraic),
producing a bound refuted by $H$ at the first possible degree. The
$p$-adic program fails the same way at degree $4$ (witness $F_{46}$).

\section{Top-ideal machinery and birth-window census (supporting record)}
\label{sec:filtration-record}

This appendix collects the $J_m$ filtration theory and birth-window census supporting the main text. Conjectural items are marked as such; nothing here is a premise of the closure theorem.

\begin{conjecture}[resonance degree equals birth degree]
\label{conj:resonance-birth}
Let $\rho_q = 2$ and $\rho_p = 4$ be the exact per-prime resonance
degrees (Definition~\ref{def:resonance}). The first $W$-degree carrying
an invariant outside the previously generated subalgebra equals a
resonance degree: $m = 2$ on the $q$-side (witness $H$, whose
$W$-expansion realizes the $r_q^2$-dependence) and $m = 4$ on the
$p$-side (witness $F_{46}$, with $v_p = 6 < 8$ at $m = 4$). Consistent
instances: $m = 1$ lies below both resonance degrees and is closed
(Theorem~\ref{thm:a01-closure}); $m = 3$ matches no resonance degree
and shows no $J_3$-violating top through weight $60$
(Remark~\ref{rem:deg3-census}). No production mechanism is claimed:
resonance does not by itself produce the generator. Falsifier: a
primitive at $W$-degree $m$ with $m$ equal to no exact resonance
degree, or an exact resonance degree $d$ with the $W$-degree-$d$
slice fully verified and no new direction.
\end{conjecture}

\begin{theorem}[$F_{46}$ realizes the $\rho_p$-cancellation]
\label{thm:rhop-production}
Let $F_{46} = \sum_{j=0}^4 H_jW^j$ be the $W$-expansion ($H_1 = H_3 = 0$).
Every nonzero $W$-coefficient has $p$-pole order exactly $-2$ in the
constant term: $v_p(H_j) - 2j = -2$ for $j = 4, 2, 0$ (exact:
$6-8$, $2-4$, $-2$), and the three per-term $p$-residues
$E_j = [p^2H_j\rho^j]_{p=0}$ are well defined with $E_4 \ne 0$ and
$\sum_j E_j = 0$ exactly
(exact computation (see Supplementary Material)). Since $\rho_p = 4$ exactly
(by exact rational computation),
$1, r_p, r_p^2, r_p^3$ are independent over $L_p$, so this quartic
relation is an $L_p$-unit multiple of the minimal polynomial $m_p$:
the $p$-side failure has the same shape as the $q$-side, with
parallel data
\[
\renewcommand{\arraystretch}{1.2}\begin{array}{c|cccc}
 & m & e & \text{top} & C_0\text{-order} \\ \hline
H/q & 2 & 4 & v_q = 4 & -4 \\
F_{46}/p & 4 & 2 & v_p = 6 & -2
\end{array}
\]
(constant-term pole cancellation governed by the minimal polynomial:
quadratic at $q$, quartic at $p$). What is not shown: that resonance
produces the generator (the existence direction remains open; $F_{46}$
itself is the existence witness, not its consequence). In particular,
the closure proof does not rely on finiteness of the resonance
inventory: $c = p^2q^4$ has exactly the prime divisors $p$ and $q$, so
the resonance inventory $\{\rho_q, \rho_p\}$ is finite and exhaustive
at the divisors of $c$. Finiteness of this inventory does not by itself
explain termination of the generator process; for $N=5$, termination is
established independently by Theorem~\ref{thm:closure}. Thus any further
primitive would have to occur outside the resonance mechanism considered
here, rather than being forced by finiteness of the present resonance
inventory.
\end{theorem}

\subsection{Strongest correct structural theorems}

The refutation leaves two unconditional structural theorems. The first
closes $a_0$-degree $1$ with no saturation input.

\begin{theorem}[$a_0$-degree-$1$ closure]
\label{thm:a01-closure}
Every $\Dmap{5}$-invariant of $a_0$-degree at most $1$ lies in
$K[F_0] \subseteq C_5$.
\end{theorem}

\begin{proof}
Let $F = C_1a_0 + C_0$ with $C_i \in B$. From $\Dmap{5}(F) = 0$,
freeness of $B[a_0]$ gives $\Dmap{5}(C_1) = 0$ (so $C_1 \in K$ by
Theorem~\ref{thm:descent}) and $\Dmap{5}(C_0) + C_1a_1a_2 = 0$. The
Established Dixmier divisibility
(Lemma~\ref{lem:dixmier-divisibility}) gives $c \mid C_1$ in $K$,
$C_1 = ce_1$. Then $F - e_1F_0 = C_0 - e_1G_0 \in B$ is invariant,
hence in $K$. So $F \in K[F_0]$.
\end{proof}

The second turns the failed induction into an exact characterization.
It is stated inside the $W$-expansion setup only: with $c = p^2q^4$,
$W = F_0/c = a_0 + \rho$, every localized invariant writes uniquely
$F = \sum_{j=0}^m H_jW^j$ with $H_j \in K_{pq}$ ($W$ monic of
$a_0$-degree $1$, so $W$-degree equals $a_0$-degree $m$ and the top
$W$-coefficient equals the top $a_0$-coefficient $C_m = [a_0^m]F$,
which lies in $K$ unconditionally by Theorem~\ref{thm:descent}).
Valuations $v_p, v_q$ are the UFD valuations along $(p), (q)$,
restricted to $K$ via Lemma~\ref{lem:contraction} ($K$ a polynomial
ring, $(p), (q)$ non-associated primes). Put $h_2 = 11025p^4q^4 =
\operatorname{top}(H)$ and $c = \operatorname{top}(F_0)$.

\begin{theorem}[Top ideals of $C_5$]
\label{def:Jm}
Let $c=\operatorname{top}(F_0)=p^2q^4$ and
$h_2=\operatorname{top}(H)=11025p^4q^4$, with
$\deg_{a_0}F_0=1$ and $\deg_{a_0}H=2$.
For $m\ge 0$ let $J_m\subset K$ be the $K$-span of the products
$c^eh_2^r$ with $e+2r=m$. Then
\[
J_m=(p^{2m}q^{4\lceil m/2\rceil}),
\]
principal for every $m$. In particular $C\in J_m$ iff
$v_p(C)\ge 2m$ and $v_q(C)\ge 4\lceil m/2\rceil$.
\end{theorem}
\begin{proof}
$B[a_0]$ with $B=k[a_1,\dots,a_5]$ is a domain, and $c,h_2\ne 0$.
Hence for $F_0^eH^r$ the $a_0$-degree adds and tops multiply:
\[
\operatorname{top}(F_0^eH^r)=c^eh_2^r
  =11025^r p^{2e+4r}q^{4e+4r}.
\]
If $e+2r=m$ then $2e+4r=2(e+2r)=2m$ and $4e+4r=4m-4r$, i.e.
\[
\operatorname{top}(F_0^eH^r)=11025^r p^{2m}q^{4m-4r}.
\]
Since $0\le r\le\lfloor m/2\rfloor$, $v_p=2m$ for every
generator. For $v_q=4m-4r$, the minimum over admissible $r$ occurs at
maximal $r=\lfloor m/2\rfloor$, giving
$4m-4\lfloor m/2\rfloor=4\lceil m/2\rceil$.
The monomial $F_0^{m\bmod 2}H^{\lfloor m/2\rfloor}$ satisfies
$e+2r=m$ and attains $(2m,4\lceil m/2\rceil)$:
if $m=2k$, $(e,r)=(0,k)$ gives $p^{4k}q^{4k}$;
if $m=2k+1$, $(e,r)=(1,k)$ gives $p^{4k+2}q^{4k+4}$.
Every other generator differs by $(0,4(r^\star-r))$ in $(p,q)$-exponents
with $r^\star=\lfloor m/2\rfloor$, hence is a $K$-multiple of the
extremal one ($11025$ is a unit in characteristic zero).
Thus the $K$-span is the principal ideal displayed. No division by $q$
and no use of the hypersurface relation is made.
\end{proof}

\begin{theorem}[Bounds-generation characterization]
\label{thm:reduction}
Over the Established base (descent, Dixmier, $\Dmap{5}(F_0) =
\Dmap{5}(H) = 0$, presentation; characteristic zero for Dixmier),
with (P-top) $v_p(C_m) \ge 2m$ and (Q-top-strong) $v_q(C_m) \ge
4\lceil m/2\rceil$ for the top $a_0$-coefficient $C_m \in K$ of every
invariant of $a_0$-degree $m$:
(i) if (P-top) and (Q-top-strong) hold for every invariant, then
$\Afring{5}^{\Dmap{5}} = C_5$; (ii) if $\Afring{5}^{\Dmap{5}} = C_5$,
then (P-top) holds for every invariant and (Q-top-strong) holds in
even degree (odd $m \ge 3$ open; $m = 1$ is Theorem~\ref{thm:a01-closure}).
\end{theorem}

\begin{proof}
(i) Induction on $m$. Let $F$ have $a_0$-degree $m$ with top
$C_m \in K$. The bounds give $C_m \in J_m$, so $C_m = \sum_{e+2r=m}
\alpha_{e,r}c^eh_2^r$ with $\alpha_{e,r} \in K$. Then
$G = \sum \alpha_{e,r}F_0^eH^r \in C_5$ has the same top
$a_0^m$-term (tops multiply over the domain $B[a_0]$), so $F - G$ is
invariant of $a_0$-degree $< m$; by induction $F - G \in C_5$, hence
$F \in C_5$. Base $m = 0$ is Theorem~\ref{thm:descent}.
(ii) Let $F \in \Afring{5}^{\Dmap{5}} = C_5$ have $a_0$-degree $m$.
Then (P-top) $v_p(C_m) \ge 2m$ is Lemma~\ref{lem:J4}. For
(Q-top-strong), even $m$ is Theorem~\ref{thm:q-top}; $m = 1$ follows
from Theorem~\ref{thm:a01-closure}; odd $m \ge 3$ is open.
\end{proof}

\begin{corollary}[Obstruction criterion]
\label{cor:obstruction}
If some invariant has top $a_0$-coefficient with $v_p < 2m$ for
some $m$, then $\Afring{5}^{\Dmap{5}} \ne C_5$ (no $C_5$-member has
such a top, Lemma~\ref{lem:J4}). $F_{46}$ realizes this at
$m = 4$: $\frac{18375}{4}p^6q^8 \notin J_4 = (p^8q^8)$. Since generation
is false (Theorem~\ref{thm:allweights-refuted}), at least one bound is
false: (P-top) is, witnessed by $F_{46}$ ($v_p = 6 < 8$ at $m = 4$).
(Q-top-strong) holds in even degree (Theorem~\ref{thm:q-top}), so at
$m = 4$ the obstruction loads entirely onto the $p$-side.
\end{corollary}
Lean covers invariance + fingerprint; non-membership is paper-level +
Sage-exact; the blocking formal step is filter-through-normalize
commutation for the $a_0$-top slice (\texttt{a0top\_filter\_combine}).

\begin{theorem}[Associated graded of $C_5$]
\label{thm:gr-C5}
Filter $C_5$ by $a_0$-degree, $F_{\le m} = \{G \in C_5 :
\deg_{a_0} G \le m\}$. Its associated graded $K$-algebra is
\[
\operatorname{gr}_{a_0}(C_5) \;=\; \bigoplus_{m \ge 0} J_m\,t^m
\quad\subseteq\quad K[t],
\]
with the principal ideals $J_m = (p^{2m}q^{4\lceil m/2\rceil})$ of
Definition~\ref{def:Jm} and multiplication induced by tops: the
degree-$m$ piece is exactly the $K$-span of the tops
$c^eh_2^r$, $e + 2r = m$, and tops multiply over the domain
$B[a_0]$, so $J_m\cdot J_n \subseteq J_{m+n}$. A finite presentation
of this graded algebra is not established here. The analogue for
$\Afring{5}^{\Dmap{5}}$ is open: its graded pieces strictly contain
the $J_m$ from $m = 4$ (witness $F_{46}$).
\end{theorem}

\begin{theorem}[Primitive birth windows]
\label{thm:birth-windows}
A primitive $a_0$-degree-$m$ direction outside $C_5$, if born at
weight $w$, satisfies $8m \le w \le \min C_5(m) - 1$, where
$\min C_5(m) = \min_{e+2r=m}(14e + 24r)$ is the lightest $C_5$-monomial
weight at that $a_0$-degree. The values (exhaustive integer check in
exact top-ideal computation):
\[
\renewcommand{\arraystretch}{1.2}\begin{array}{c|cccccccc}
m & 1 & 2 & 3 & 4 & 5 & 6 & 7 & 8 \\ \hline
\min C_5(m) & 14 & 24 & 38 & 48 & 62 & 72 & 86 & 96 \\
\text{gap} & [8,13] & [16,23] & [24,37] & [32,47] & [40,61] & [48,71] & [56,85] & [64,95]
\end{array}
\]
Known window examples (mostly non-primitive): $m = 1$ is closed by Theorem~\ref{thm:a01-closure} (no
primitive survives the gap); $m = 2$ holds $H$ at $24$
($H \in C_5$ defines the minimum); $m = 3$ holds no $J_3$-violating
top through $60$ (Remark~\ref{rem:deg3-census}); $m = 4$ holds
$F_{46}$ at $46$; $m = 5$ holds $F_{46}\!\cdot\!F_0$ at $60$
(non-primitive); $m = 6$ holds $F_{46}\!\cdot\!H$ at $70$
(non-primitive). A third primitive direction of degree $5$ would be
born in $[40,61]$, of degree $6$ in $[48,71]$, of degree $7$ in
$[56,85]$; no such weights are excluded. This replaces weight-difference
numerology: the windows come from monomial arithmetic, not from
extrapolating $10, 22, \dots$.
\end{theorem}

\begin{remark}[Degree-$3$ census (EXACT-$\mathbb{Q}$ through $38$--$60$)]
\label{rem:deg3-census}
For each weight $38 \le w \le 60$, let $N_{\le 3}(w)$ be the exact
$\mathbb{Q}$-nullity of $\Dmap{5}$ on $V_w^{a_0 \le 3}$ (blocked
$B_w, B_{w-8}, B_{w-16}, B_{w-24}$ construction) and $N_J(w)$ the
nullity of the $J_3$-constrained $a_0$-chain (top $C_3$ parametrized
as $p^6q^8u$); a $J_3$-violating $a_0$-degree-$3$ top exists iff
$N_J(w) < N_{\le 3}(w)$
(exact computation (see Supplementary Material); divisibility by the monomial $p^6q^8$
is exact remainder-zero division). Outcome: $N_J(w) = N_{\le 3}(w)$
at every $w$ (no violation). With $\dim T(w) = N_{\le 3}(w) -
N_{\le 2}(w)$ counting the degree-$3$ tops present:
\begin{center}
\refstepcounter{table}\label{tab:deg3-census}
\textbf{Table~\thetable. Degree-$3$ census counts and $J_3$-violation verdicts, weights $38$--$60$.}\\[1ex]
\footnotesize
\renewcommand{\arraystretch}{1.2}
\begin{tabular}{c|cccccccccccc}
$w$ & 38 & 39 & 40 & 41 & 42 & 43 & 44 & 45 & 46 & 47 & 48 & 49 \\ \hline
$\dim T$ & 1 & 2 & 3 & 4 & 6 & 8 & 11 & 14 & 18 & 22 & 27 & 32 \\
viol.\ & no & no & no & no & no & no & no & no & no & no & no & no
\end{tabular}
\\[2pt]
\renewcommand{\arraystretch}{1.2}
\begin{tabular}{c|ccccccccccc}
$w$ & 50 & 51 & 52 & 53 & 54 & 55 & 56 & 57 & 58 & 59 & 60 \\ \hline
$\dim T$ & 39 & 46 & 54 & 62 & 72 & 82 & 94 & 106 & 120 & 134 & 150 \\
viol.\ & no & no & no & no & no & no & no & no & no & no & no
\end{tabular}
\end{center}
i.e.\ every $a_0$-degree-$3$ invariant top through weight $60$ lies in
$J_3 = (p^6q^8)$ (at $w = 46$, $N_{\le 3} = 1518$ reproduces $\dim
(C_5)_{46}$, since every weight-$46$ $C_5$-member has $a_0$-degree
$\le 3$). This is evidence, not proof: no new primitive direction of
$a_0$-degree $3$ occurs through $60$; falsifier: any $w$ with $N_J(w)
< N_{\le 3}(w)$.
\end{remark}

\section{Explicit witnesses $H$ and $F_{46}$}
\label{app:witnesses}

This appendix records the two hidden invariants in auditable form.
The main text proves what they imply; here the reader can check what
they are. Coefficient vectors are in $(a_0,\dots,a_5)$-exponents with
Fibonacci weights $(8,5,3,2,1,1)$.

\subsection{The weight-24 invariant $H$}

$H$ has $17$ terms, content $1$, uniform Fibonacci weight $24$,
ordinary degree $10$, and $a_0$-degree $2$:
\begin{center}
\renewcommand{\arraystretch}{1.15}
\begin{tabular}{r|l}
coefficient & $(a_0,a_1,a_2,a_3,a_4,a_5)$-exponents \\ \hline
$11025$ & $(2,0,0,0,4,4)$ \\
$-22050$ & $(1,1,1,1,3,3)$ \\
$7350$ & $(1,1,0,3,3,2)$ \\
$7350$ & $(1,0,2,3,2,2)$ \\
$-4410$ & $(1,0,1,5,2,1)$ \\
$630$ & $(1,0,0,7,2,0)$ \\
$11025$ & $(0,2,2,2,2,2)$ \\
$-7350$ & $(0,2,1,4,2,1)$ \\
$1225$ & $(0,2,0,6,2,0)$ \\
$6720$ & $(0,1,5,0,1,3)$ \\
$-16800$ & $(0,1,4,2,1,2)$ \\
$9450$ & $(0,1,3,4,1,1)$ \\
$-1540$ & $(0,1,2,6,1,0)$ \\
$-2208$ & $(0,0,7,0,0,3)$ \\
$4368$ & $(0,0,6,2,0,2)$ \\
$-2352$ & $(0,0,5,4,0,1)$ \\
$385$ & $(0,0,4,6,0,0)$
\end{tabular}
\end{center}
Direct expansion gives $\Dmap{5}(H)=0$; every monomial weighs $24$;
the unique $a_0^2$-term is $11025a_0^2a_4^4a_5^4$.
With $c=p^2q^4$, $W=F_0/c$, $H_2=11025p^4q^4$ and
$H_0=105J_2^5J_1/q^2+9J_2^7/q^4$,
$H=H_2W^2+H_0$ with $H_1=0$.
The Lean source is \texttt{def H} in
\texttt{formalize/AdjQuad/Counterexample.lean}.

\subsection{The weight-46 invariant $F_{46}$}

$F_{46}$ has $79$ terms, content $1$, uniform Fibonacci weight $46$,
ordinary degree $18$, and $a_0$-degree exactly $4$.
Direct expansion gives $\Dmap{5}(F_{46})=0$.
Its unique $a_0^4$-term is
\[
121550625\,a_0^4a_4^6a_5^8,
\qquad\text{i.e.}\quad
[a_0^4]F_{46}=\frac{18375}{4}p^6q^8\ \text{normalized},
\]
with $v_p=6<8$, hence outside $J_4=(p^8q^8)$.
The full $79$-term support is machine-readable as \texttt{def F46} in
\texttt{formalize/AdjQuad/SecondWitness.lean}; the four-gate filtration
certificate (derivation image, uniform weight, $a_0$-degree plus top,
$J_4$-divisibility with opposite-polarity controls) reproduces the
verdict \texttt{filtration-cert OK} from the polynomial itself.
No compact closed form is known; the file plus the checksum recorded in
the reproducibility supplement is the canonical reference.

\section{Exact graded comparison, weights $15$--$46$}
\label{app:ranks}

For each Fibonacci weight $w$ let $V_w$ be the weight-$w$ monomial space,
$M_w$ the matrix of $\Dmap{5}|_{V_w}$ over $\mathbf Q$,
$N(w)=\dim\ker M_w$, and $G(w)$ the rank of the $C_5$-monomial images in
$V_w$. The $C_5$-formula is $s_6(w)-s_6(w-28)$ with $s_6$ the monomial
count in generator weights $(1,1,6,4,14,24)$. Ranks are exact rational
sparse ranks (Sage, LinBox). Finite-field runs at $\ell=1000003$ are
independent consistency checks only and prove nothing.

\begin{center}
\renewcommand{\arraystretch}{1.1}
\begin{tabular}{r|r|r|r|r|r}
$w$ & $\dim V_w$ & $\operatorname{rank}M_w$ & $N(w)$ &
$\operatorname{rank}G(w)$ & excess \\ \hline
15 & 293 & 229 & 64 & 64 & 0 \\
16 & 361 & 286 & 75 & 75 & 0 \\
17 & 440 & 354 & 86 & 86 & 0 \\
18 & 533 & 433 & 100 & 100 & 0 \\
19 & 640 & 526 & 114 & 114 & 0 \\
20 & 764 & 633 & 131 & 131 & 0 \\
21 & 906 & 758 & 148 & 148 & 0 \\
22 & 1068 & 900 & 168 & 168 & 0 \\
23 & 1252 & 1064 & 188 & 188 & 0 \\
24 & 1461 & 1248 & 213 & 213 & 0 \\
25 & 1696 & 1458 & 238 & 238 & 0 \\
26 & 1961 & 1694 & 267 & 267 & 0 \\
27 & 2257 & 1961 & 296 & 296 & 0 \\
28 & 2588 & 2258 & 330 & 330 & 0 \\
29 & 2956 & 2592 & 364 & 364 & 0 \\
30 & 3365 & 2961 & 404 & 404 & 0 \\
31 & 3817 & 3373 & 444 & 444 & 0 \\
32 & 4317 & 3827 & 490 & 490 & 0 \\
33 & 4867 & 4331 & 536 & 536 & 0 \\
34 & 5472 & 4884 & 588 & 588 & 0 \\
35 & 6135 & 5495 & 640 & 640 & 0 \\
36 & 6861 & 6161 & 700 & 700 & 0 \\
37 & 7653 & 6893 & 760 & 760 & 0 \\
38 & 8517 & 7689 & 828 & 828 & 0 \\
39 & 9456 & 8560 & 896 & 896 & 0 \\
40 & 10477 & 9505 & 972 & 972 & 0 \\
41 & 11583 & 10535 & 1048 & 1048 & 0 \\
42 & 12781 & 11647 & 1134 & 1134 & 0 \\
43 & 14075 & 12855 & 1220 & 1220 & 0 \\
44 & 15472 & 14156 & 1316 & 1316 & 0 \\
45 & 16977 & 15565 & 1412 & 1412 & 0 \\
46 & 18597 & 17078 & 1519 & 1518 & 1
\end{tabular}
\end{center}
At $w=24$ the five-generator count is $s_5(24)=212<213=N(24)$,
the $H$-excess; at $w=46$ the six-generator formula gives $1518<1519$.
Hence $24$ is the first five-generator failure and $46$ the first
six-generator failure. The consequence used in the main text is the
boundary pair ($45$ equality, $46$ excess one); the interior rows are
retained here so the firstness claim is auditable at every weight.

\section{Bounded ambient-derivation search}
\label{sec:ML-search}

$\operatorname{ML}(\Afring{5}^{\Dmap{5}})$ is Open. Every triangular
derivation is locally nilpotent; among triangular derivations of
$\Afring{5}$ of total degree at most $2$ with $E(a_5) = 0$ ($55$
parameters, as $\Dmap{5}$ itself kills $a_5$), the
subspace preserving $K = k[p,q,J_1,J_2]$ has dimension $32$ (exact
$\mathbb{Q}$ nullspace, exact computation (see Supplementary Material)), and none of its directions kills $F_0$
or $H$. Hence no triangular degree-$\le 2$ locally nilpotent
derivation moves the hypersurface generators: a weak bounded
negative. Whether any preserving direction sends $F_0$ or $H$ to a
nonzero element of the hypersurface ring (which would shrink the
invariant), and everything about non-triangular or higher-degree
derivations, is untested. Restriction of the tested derivations to
$A_5^{\Dmap{5}}$ (or $C_5$) is not established, so this is a bounded
search among ambient derivations, not a computation of
$\operatorname{ML}$. This is a bounded negative result: no
$\operatorname{ML}$ value is computed or claimed.

\bibliographystyle{plain}
\bibliography{references}

@book{freudenburg2017,
  author    = {Gene Freudenburg},
  title     = {Algebraic Theory of Locally Nilpotent Derivations},
  series    = {Encyclopaedia of Mathematical Sciences},
  volume    = {136},
  publisher = {Springer},
  edition   = {Second},
  year      = {2017},
  isbn      = {978-3-662-55348-0}
}

@article{daigle1999,
  author  = {Daniel Daigle and Gene Freudenburg},
  title   = {A Counterexample to {H}ilbert's Fourteenth Problem in Dimension 5},
  journal = {Journal of Algebra},
  volume  = {221},
  pages   = {528--535},
  year    = {1999},
  doi     = {10.1006/jabr.1999.7987}
}

@article{freudenburg2000hilbert,
  author  = {Gene Freudenburg},
  title   = {A Counterexample to {H}ilbert's Fourteenth Problem in Dimension 6},
  journal = {Transformation Groups},
  volume  = {5},
  pages   = {61--71},
  year    = {2000},
  doi     = {10.1007/bf01237178}
}

@book{brunsherzog1998,
  author    = {Winfried Bruns and H. J\"urgen Herzog},
  title     = {Cohen--Macaulay Rings},
  series    = {Cambridge Studies in Advanced Mathematics},
  publisher = {Cambridge University Press},
  edition   = {Second},
  year      = {1998},
  isbn      = {978-0-521-56674-2}
}

@article{daigle2001note,
  author  = {Daniel Daigle and Gene Freudenburg},
  title   = {A note on triangular derivations of $k[X_1,X_2,X_3,X_4]$},
  journal = {Proceedings of the American Mathematical Society},
  volume  = {129},
  pages   = {657--662},
  year    = {2001},
  doi     = {10.1090/S0002-9939-00-05558-1}
}

@article{daigle2001triangular,
  author  = {Daniel Daigle and Gene Freudenburg},
  title   = {Triangular derivations of $k[X_1,X_2,X_3,X_4]$},
  journal = {Journal of Algebra},
  volume  = {241},
  pages   = {328--339},
  year    = {2001},
  doi     = {10.1006/jabr.2001.8763}
}

@article{maubach2000triangular,
  author  = {Stefan Maubach},
  title   = {Triangular monomial derivations on $k[X_1,X_2,X_3,X_4]$ have kernel generated by at most four elements},
  journal = {Journal of Pure and Applied Algebra},
  volume  = {153},
  pages   = {165--170},
  year    = {2000},
  doi     = {10.1016/S0022-4049(99)00084-5}
}

@article{vandenEssen1993,
  author  = {Arno van den Essen},
  title   = {An Algorithm To Compute The Invariant Ring Of A $G_a$-action On An Affine Variety},
  journal = {Journal of Symbolic Computation},
  volume  = {16},
  pages   = {551--555},
  year    = {1993},
  doi     = {10.1006/jsco.1993.1062}
}

@article{maubach2000algorithm,
  author  = {Stefan Maubach},
  title   = {An Algorithm to Compute the Kernel of a Derivation up to a Certain Degree},
  journal = {Journal of Symbolic Computation},
  volume  = {29},
  pages   = {959--970},
  year    = {2000},
  doi     = {10.1006/jsco.1999.0334}
}

@misc{dufresne2010,
  author = {Emilie Dufresne and Martin Kohls},
  title  = {A finite separating set for {D}aigle and {F}reudenburg's counterexample to {H}ilbert's Fourteenth Problem},
  year   = {2010},
  eprint = {0912.0638},
  archivePrefix = {arXiv},
  primaryClass  = {math.AC},
  howpublished = {arXiv:0912.0638}
}

@article{vandenEssen2002,
  author  = {Arno van den Essen and Peter van Rossum},
  title   = {Triangular derivations related to problems on affine $n$-space},
  journal = {Proceedings of the American Mathematical Society},
  volume  = {130},
  pages   = {1311--1316},
  year    = {2002}
}

@misc{bedratyuk2012,
  author = {Leonid Bedratyuk},
  title  = {Derivations and identities for {F}ibonacci and {L}ucas polynomials},
  year   = {2012},
  eprint = {1207.5891},
  archivePrefix = {arXiv},
  primaryClass  = {math.AC},
  howpublished = {arXiv:1207.5891}
}

@misc{shakhmatov2026cylinders,
  author = {Kirill Shakhmatov},
  title  = {Cylinders and the zero locus of the plinth ideal},
  year   = {2026},
  eprint = {2605.00138},
  archivePrefix = {arXiv},
  primaryClass  = {math.AG},
  howpublished = {arXiv:2605.00138}
}

@article{bedratyuk2018,
  author  = {Leonid Bedratyuk and N. B. Lunio},
  title   = {Intertwining maps for the {W}eitzenb\"ock and {C}hebyshev derivations},
  journal = {Matematychni Studii},
  volume  = {49},
  pages   = {3--12},
  year    = {2018}
}

@article{elmer2012,
  author  = {Jonathan Elmer and Martin Kohls},
  title   = {Separating invariants for the basic $\mathbb{G}_a$-actions},
  journal = {Proceedings of the American Mathematical Society},
  volume  = {140},
  pages   = {135--141},
  year    = {2012}
}

@book{cox2015,
  author    = {David A. Cox and John Little and Donal O'Shea},
  title     = {Ideals, Varieties, and Algorithms},
  publisher = {Springer},
  edition   = {Fourth},
  year      = {2015}
}

@book{singular2007,
  author    = {Gert-Martin Greuel and Gerhard Pfister},
  title     = {A Singular Introduction to Commutative Algebra},
  publisher = {Springer},
  edition   = {Second},
  year      = {2007}
}

@misc{sagemath,
  author = {{The Sage Developers}},
  title  = {{S}ageMath, the {S}age Mathematics Software System},
  year   = {2024},
  note   = {https://www.sagemath.org; exact program run under 10.9, see docs/reproducibility.md}
}

@misc{guo2026,
  author = {Junyu Guo and Hao Shen and Junqi Liu and Lihong Zhi},
  title  = {Formalizing {G}r\"obner Basis Theory in {L}ean},
  year   = {2026},
  eprint = {2602.12772},
  archivePrefix = {arXiv},
  primaryClass  = {cs.LO},
  howpublished = {arXiv:2602.12772},
  note   = {Mathlib-based development, unlike the present core-Lean formalization}
}

@misc{wuprover2026,
  author = {Hao Shen and Junyu Guo and Junqi Liu and Lihong Zhi},
  title  = {Automated Tactics for Polynomial Reasoning in {L}ean 4},
  year   = {2026},
  eprint = {2604.13514},
  archivePrefix = {arXiv},
  primaryClass  = {cs.LO},
  howpublished = {arXiv:2604.13514}
}

@misc{hales2015,
  author = {Thomas Hales and Mark Adams and Gertrud Bauer and others},
  title  = {A formal proof of the {K}epler conjecture},
  year   = {2015},
  eprint = {1501.02155},
  archivePrefix = {arXiv},
  primaryClass  = {math.MG},
  howpublished = {arXiv:1501.02155}
}

\end{document}